\documentclass[10.5pt,reqno]{amsart}
\usepackage{longtable} 
\usepackage{hyperref}
\usepackage[T1]{fontenc}
\usepackage[utf8]{inputenc}
\usepackage[english]{babel} 
\usepackage{textcomp}
\usepackage{dsfont}
\usepackage{latexsym}
\usepackage{amssymb}
\usepackage{amsthm}
\usepackage{amsmath}
\DeclareMathAlphabet{\mathpzc}{OT1}{pzc}{m}{en}
\usepackage{yfonts}
\usepackage{xfrac}
\usepackage{newlfont}
\usepackage{graphicx}
\usepackage{mathtools}
\usepackage{comment}
\usepackage{indentfirst}
\usepackage{braket}
\usepackage{mathrsfs} 
\usepackage{xcolor}

\usepackage{etoolbox}

\usepackage{scalerel}[2014/03/10]
\usepackage[usestackEOL]{stackengine}
\newcommand{\dashint}{\,\ThisStyle{\ensurestackMath{%
			\stackinset{c}{.2\LMpt}{c}{.5\LMpt}{\SavedStyle-}{\SavedStyle\phantom{\int}}}%
		\setbox0=\hbox{$\SavedStyle\int\,$}\kern-\wd0}\int}

\DeclareMathOperator{\card}{Card}

\DeclareMathOperator{\supp}{Supp}

\DeclareMathOperator{\ad}{ad}

\renewcommand{\Re}{\mathrm{Re}\,}

\newcommand{\ee}{\mathrm{e}}

\newcommand{\loc}{\mathrm{loc}}
\newcommand{\vect}[1]{\mathbf{{#1}}}
\newcommand{\dd}{\mathrm{d}}

\DeclarePairedDelimiter{\abs}{\lvert}{\rvert}

\DeclarePairedDelimiter{\norm}{\lVert}{\rVert}

\let\originalleft\left
\let\originalright\right
\renewcommand{\left}{\mathopen{}\mathclose\bgroup\originalleft}
\renewcommand{\right}{\aftergroup\egroup\originalright}

\newcommand{\grado}{\Df}
\newcommand{\N}{\mathds{N}}
\newcommand{\Z}{\mathds{Z}}
\newcommand{\Q}{\mathds{Q}}

\newcommand{\C}{\mathds{C}}

\newcommand{\R}{\mathds{R}}

\newcommand{\gf}{\mathfrak{g}}

\newcommand{\Df}{\mathfrak{D}}

\newcommand{\unif}{\mathrm{unif}}

\newcommand{\Bc}{\mathcal{B}}

\newcommand{\Cc}{\mathcal{C}}
\newcommand{\Dc}{\mathcal{D}}

\newcommand{\Fc}{\mathcal{F}}

\newcommand{\Ic}{\mathcal{I}}

\newcommand{\Kc}{\mathcal{K}}
\newcommand{\Lc}{\mathcal{L}}
\renewcommand{\Mc}{\mathcal{M}}
\newcommand{\Nc}{\mathcal{N}}
\newcommand{\Oc}{\mathcal{O}}

\newcommand{\Rc}{\mathcal{R}}
\newcommand{\Sc}{\mathcal{S}}

\newcommand{\Sr}{\mathscr{S}}

\newcommand{\meg}{\leqslant}
\newcommand{\Meg}{\geqslant}
\newcommand{\eps}{\varepsilon}
\renewcommand{\phi}{\varphi}
\newcommand{\mi}{\mu}

\thanks{{\em Math Subject Classification 2020}: 46E36, 22E30.}
\thanks{The author is a member of the 	Gruppo Nazionale per l'Analisi
	Matematica, la Probabilit\`a e le	loro Applicazioni (GNAMPA) of
	the Istituto Nazionale di Alta Matematica (INdAM). The author was partially funded by the INdAM-GNAMPA Project CUP\_E5324001950001.
}

\begin{document}
	
	\title[Besov and Triebel--Lizorkin Spaces]{Besov and Triebel--Lizorkin Spaces on Filtered Lie Groups with Polynomial Growth, II: Interpolation, Localization, Pointwise Multiplication}
	
	\author[M.\ Calzi]{Mattia Calzi} 
	\address{Dipartimento di Matematica, Universit\`a degli Studi di
		Milano, Via C. Saldini 50, 20133 Milano, Italy}
	\email{{\tt mattia.calzi@unimi.it}}
	
	\theoremstyle{definition}
	\newtheorem{deff}{Definition}[section]

	\newtheorem{oss}[deff]{Remark}
	
	\newtheorem{ass}[deff]{Assumptions}
	
	\newtheorem{nott}[deff]{Notation}

	\theoremstyle{plain}
	\newtheorem{teo}[deff]{Theorem}
	
	\newtheorem{lem}[deff]{Lemma}
	
	\newtheorem{prop}[deff]{Proposition}
	
	\newtheorem{cor}[deff]{Corollary}
	
	\begin{abstract}
		We continue to develop a theory of  Besov and Triebel--Lizorkin spaces associated with weighted subcoercive operators on a real connected Lie group, specializing to the case of groups with polynomial volume growth. We consider the full scale of spaces and consider interpolation, localization, and pointwise multiplication.
	\end{abstract}
	
	\maketitle

	\section{Introduction}

	Besov and Triebel--Lizorkin spaces form a large class of function spaces on the Euclidean spaces which provides a uniform, albeit somewhat technical, way to study simultaneously several classical function spaces, such as Sobolev spaces with integer regularity, fractional Sobolev spaces, both in the version of Sovolev--Slobodeckij--Gagliardo spaces and in the version of Bessel potential spaces, Lipschitz spaces, Hardy and BMO spaces, etc. These spaces have been extensively studied in the classical Euclidean setting (cf., e.g.,~\cite{TriebelFS,TriebelFS2,TriebelFS3}), but have also been extended to more general contexts, such as: open subsets of $\R^n$ (cf., e.g.,~\cite[Chapter 5]{TriebelFS2}); Riemannian manifolds with bounded geometry (cf., e.g.,~\cite[Chapter 7]{TriebelFS2}); Lie groups endowed with a left-invariant Riemannian (cf., e.g.,~\cite[Chapter 7]{TriebelFS2}) or sub-Riemannian metric (cf., e.g.,~\cite{BPV,BPV2,BPV3}); metric spaces endowed with suitable operators resembling a (sub-)Laplacian (cf., e.g.,~\cite{TriebelFS3,Hu}).\footnote{The literature on the subject is quite extensive and the above mentioned references should only be intended as a very short list of examples, and not as an exhaustive list of all (or only of the main) contributions.} 
	There are nonetheless some contexts where `measuring regularity' in a `Riemannian way,' that is, grouping together all differential operators of the same order, or, more generally, in a `sub-Riemannian way,' appears to be inconvenient since it either clashes with the underlying geometry of the space or with the structure of the differential operator at hand. For instance, let $G$ be a homogeneous group, that is, a simply connected nilpotent Lie group whose Lie algebra $\gf$ has a graduation $(\gf_\lambda)_{\lambda>0}$; in other words, $\gf=\bigoplus_{\lambda>0} \gf_\lambda$ and $[\gf_\lambda,\gf_\nu]\subseteq \gf_{\lambda+\mi}$ for every $\lambda,\mi>0$. Then, $G$ may be identified with $\gf$ by means of the exponential map and $\gf$ may be endowed with a family of automorphic dilations $(\delta_r)_{r>0}$ defined so that $\delta_r(X)=r^\lambda X$ for every $X\in \gf_\lambda$. Operators which are compatible with these dilations (that is, homogeneous operators) are consequently quite natural in this context and one is therefore led to consider (`Goodman type') Sobolev spaces of the form $\Set{f\in L^p(G)\colon\forall \alpha\; (d_\alpha\meg k \implies\vect X^\alpha f\in L^p(G))}$, where $\vect X^\alpha=X_1^{\alpha_1}\cdots X_n^{\alpha_n}$ for some homogeneous basis $(X_1,\dots, X_n)$ of $\gf$, and where $d_\alpha=\sum_j \alpha_j \deg(X_j)$. It turns out, however, that these spaces behave quite weirdly for general $k$ -- for instance, they do not interpolate as one may expect. In fact, a different class of `Bessel potential' Sobolev spaces exhibiting a more natural behviour was introduced in~\cite{FischerRuzhansky}, and was shown to coincide with the previous `Goodman type' Sobolev spaces only for specific values of $k$. It is now worthwhile remarking that, whereas the usual `Bessel potential' Sobolev spaces (and, more generally, several of the Besov and Triebel--Lizorkin spaces briefly mentioned above) are essentially constructed using a (sub-)Laplacian, these `homogeneous' Sobolev spaces were constructed using positive Rockland operators instead, that is, homogeneous and hypoelliptic left-invariant differential operators. A definition using second order differential operator would simply not be possible. 
	Notice, by the way, that replacing a second-order subelliptic differential operator with a higher-order one provides several technical difficulties, since the associated heat kernel cannot be positive, and there is no longer any finite speed property for the corresponding wave propagator -- tools which often lie at the core of several proofs in the literature.
	
	It is therefore natural to wonder whether there is a more general framework which allows to deal at the same time with sub-Laplacians and positive Rockland operators. As a matter of fact, ter Elst and Robinson showed that weighted subcoercive operators provide a quite reasonable and natural choice (cf.~\cite{ElstRobinson}). Indeed, given a group $G$ whose Lie algebra is endowed with a suitable increasing filtration $(\gf_\lambda)_{\lambda>0}$, it is possible to associate a homogeneous group $G_*$ (its `contraction') to $G$ in a natural way, and to associate to every left-invariant differential operator $\Lc$ of degree $d$ some homogeneous left-invariant differential operator $P$ of degree $d$ on $G_*$ (which plays the r\^ole of the `principal part' of $\Lc$). The operator $\Lc$ is then said to be weighted subcoercive if $P+P^*$ is a positive Rockland operator. Notice that this definition mimics closely that of elliptic operators, and that Rockland operators play the r\^ole of homogeneous elliptic operators. Weighted subcoercive operators then enjoy several useful properties. For example, they generate a heat semigroup $(\ee^{-t\Lc})$ whose convolution kernel satisfies suitable Gaussian estimates (even though the exponential decay depends on the degree $d$ and is milder than the classical one). If, in addition, $\Lc$ is formally self-adjoint, then the closure of $\Lc$ on the space of test functions is self-adjoint on $L^2$, hence generates a functional calculus which enjoys particularly interesting properties when $G$ has polynomial growth.
	As shown in~\cite{BCP}, using the heat semigroup associated with a weighted subcoercive operator allows one to define natural Besov and Triebel--Lizorkin  spaces  $B^{p,q}_\alpha$ and $F^{p,q}_\alpha$ on a general connected filtered Lie group. The resulting spaces then do not depend on the chosen operator. In~\cite{BCP,Calzi2,Calzi3} several properties of these spaces were proved, including:
	\begin{itemize}
		\item $B^{p,q}_\alpha$ and $F^{p,q}_\alpha$ are Banach spaces;
		
		\item the space $C^\infty_c(G)$ of test functions is dense in $B^{p,q}_\alpha$ and $F^{p,q}_\alpha$ for $p,q<\infty$;
		
		\item $B^{p',q'}_{-\alpha}$ and $F^{p',q'}_{-\alpha}$ may be canonically identified with the duals of $B^{p,q}_\alpha$ and $F^{p,q}_\alpha$, respectively, when $p,q<\infty$;
		
		\item $B^{p_1,q_1}_{\alpha_1}\subseteq B^{p_2,q_2}_{\alpha_2}$ when $p_1\meg p_2$, $\alpha_2 -Q_*/p_2\meg \alpha_1-Q_*/p_1$, and either $q_1\meg q_2$ or $\alpha_2 -Q_*/p_2< \alpha_1-Q_*/p_1$, where $Q_*$ denotes the homogeneous dimension of $G_*$;\footnote{This and the following fact are actually true when $G$ is endowed with a left Haar measure, and should be slightly modified in the general case.}
		
		\item   $F^{p_1,q_1}_{\alpha_1}\subseteq F^{p_2,q_2}_{\alpha_2}$ when $p_1\meg p_2$, $\alpha_2 -Q_*/p_2\meg \alpha_1-Q_*/p_1$, and either $q_1\meg q_2$ or $\alpha_2 -Q_*/p_2< \alpha_1-Q_*/p_1$ or $p_1<p_2$;
		
		\item $B^{p,q}_\alpha,F^{p,q}_\alpha \subseteq L^p$ when $\alpha>0$, and $F^{p,2}_0=L^p$ for $p\in (1,\infty)$;
		
		\item if $\omega\in \R$ is sufficiently large, then $(\Lc+\omega I)^{-\alpha'}$ induces canonical isomorphisms of $B^{p,q}_\alpha$ and $F^{p,q}_\alpha$ onto $B^{p,q}_{\alpha+\alpha'}$ and $F^{p,q}_{\alpha+\alpha'}$, respectively;
		
		\item if $(X_j)$ is a family of elements of $\gf$ such that the corresponding elements $Y_j$ of $\gf_*$ induce a basis of $\gf_*/[\gf_*,\gf_*]$, and if $\dd$ is the least common multiple of the $d_j=\deg(X_j)$, then $f\in B^{p,q}_\alpha$ (resp.\ $f\in F^{p,q}_\alpha$) if and only if $\ee^{-\Lc}f\in L^p$ and $X_j^{\dd/d_j}\in B^{p,q}_{\alpha-\dd}$ (resp.\ $X_j^{\dd/d_j}\in F^{p,q}_{\alpha-\dd}$) for every $j$. In particular $F^{p,2}_{k\dd}$ coincides with the above-mentioned `Goodman type' Sobolev spaces when $k\in\N$ and $p\in (1,\infty)$;
		
		\item the spaces $B^{p,q}_\alpha$ and the spaces $F^{p,q}_\alpha$ interpolate as the classical ones;
		
		\item $B^{p,q}_\alpha\cap L^\infty$ and $F^{p,q}_\alpha\cap L^\infty$ are algebras under pointwise multiplication for every $\alpha>0$. In particular, $B^{p,q}_\alpha$ and $F^{p,q}_\alpha$ are algebras for every $\alpha>Q_*/p$;
		
		\item the spaces of pointwise multipliers of $B^{p,q}_\alpha$ and $F^{p,q}_\alpha$ may be characterized for $\alpha>Q_*/p$;
		
		\item  the spaces $F^{p,q}_\alpha$  enjoy a localization property as the classical ones;

		\item some Besov and Triebel--Lizorkin spaces may be described in terms of (finite) differences.
	\end{itemize}
	
	In~\cite{Calzi4}, the full scale of Besov and Triebel--Lizorkin spaces was considered on a filtered Lie group with polynomial volume growth, and the extension of the first eight of the above properties was considered. In addition, also a discretization procedure, in the spirit of~\cite{FrazierJawerth}, was studied. Furthermore, in~\cite{CalziRizzo}, the spaces $F^{p,2}_0$ were compared with suitably defined local Hardy (for $p<\infty$) and bmo (for $p=\infty$) spaces. 
	
	The purpose of this paper is to consider the full scale of the Besov and Triebel--Lizorkin spaces (that is, for $p,q\in(0,\infty]$) when $G$ has polynomial volume growth, and to establish the remaining properties in this more general context. 
	As in~\cite{Calzi4}, the proofs are sometimes similar in spirit to those of~\cite{BCP,Calzi2,Calzi3}, but require different techniques, since the heat semigroup proves to be quite difficult to handle when $\min(p,q)<1$. 
	For example, interpolation is harder to handle, since we cannot simply directly reduce to the known interpolation theorems on spaces of sequences (of which Besov and Triebel--Lizorkin space no longer appear to be retracts, at least in general). We shall on the contrary make use of the discretization techniques developed in~\cite{Calzi4}, following~\cite{FrazierJawerth}.
	Analogously, even though we shall again rely on suitable paraproducts to study pointwise multiplication, we shall also need to implement a suitable discretization procedure in order to complete the proof in the general case. In addition, in order to achieve an appropriate localization, we shall need to provide an additional characterization of the quasi-norms of Besov and Triebel--Lizorkin spaces which is somewhat independent of the functional calculi associated with weighted subcoercive operators, and is modelled on the techniques developed in~\cite{Jordi}. 
	Concerning the characterization by differences, we were not able to develop any technique which could allow us to extend the results proved in~\cite{Calzi2}.

	\smallskip

	\section{Preliminaries}
	
	\subsection{Relatively Invariant Measures and Convolution}

	Throughout the paper, we shall denote with $G$ a connected (finite-dimensional, real) Lie group  with Lie algebra $\gf$. We shall denote with $\beta$ a left Haar measure and we shall assume that $\beta$ has polynomial volume growth, that is, there is $Q_G \Meg 0$ such that, for every compact neighbourhood $V$ of the identity $e$ in $G$, one has
	\[
	\beta(V^n)\asymp n^{Q_G}
	\]
	for $n\to \infty$. Then, $\beta$ is also a right Haar measure.
	
	We now recall some facts about convolution. Given two convolvable\footnote{We shall not provide a precise definition of this concept, since this would drag us too far away from the main topic. } distributions $f,g$ on $G$, one has
	\[
	\langle f*g, \phi\rangle= \langle f\otimes g, (x,y)\mapsto \phi(xy)\rangle
	\]
	for every $\phi\in C^\infty_c(G)$.
	We shall identify each $f\in L^1_\loc(G)$ with $f\cdot \beta$, that is, the measure with density $f$ with respect to $\beta$. Given two convolvable functions $f,g$ such that $f*g$ is absolutely continuous with respect to $\beta$, we shall generally identify $f*g$ with its density with respect to $\beta$. Thus, under very mild conditions which will be always verified in the applications,
	\begin{equation}\label{eq:1}
		(f* g)(x)=\int_G f(x y^{-1}) g(y)  \,\dd \beta(y)= \int_G f(y) g(y^{-1}x) \,\dd \beta(y).
	\end{equation} 
	
	Similar formulae apply when either $f$ or $g$ is a distribution (and the convolution is still a function).
	Observe that, if $f,g$ are convolvable distributions, $X$ is a left-invariant differential operator, and $Y$ is a right-invariant differential operator, then $Yf $ and $X g$ are convolvable and
	\[
	YX(f*g)=(Yf)*(Xg).
	\]
	
	If $f,g,h$ are distributions then, under some reasonable conditions that will always be satisfied in the applications, 
	\[
	\langle f * g, h \rangle =\langle f, h*  \check g \rangle =\langle g,   \check f * h\rangle,
	\]
	where $\langle \check f,\phi\rangle=\langle f,\check \phi \rangle$ for every $\phi\in C^\infty_c(G)$, and $\check \phi=\phi(\,\cdot\,^{-1})$.  
	
	We now recall Young's inequality in this context. Take $p_1,p_2,p_3\in [1,\infty]$ so that $\frac{1}{p_1'}+\frac{1}{p_2'}=\frac{1}{p_3'}$. Then,
	\[
	\norm{f*g}_{L^{p_3}(G)}\meg \norm{f}_{L^{p_1}(G)}\norm{g}_{L^{p_2}(G)}
	\] 
	for every two positive $\beta$-measurable functions $f,g$ (for positive measurable functions, convolution may be defined by means of~\eqref{eq:1}).  
	
	In order to simplify the notation, given a measure $\mi$ on a measurable space $X$, $p\in (0,\infty]$, and a $\mi$-measurable function $f$ on $X$, we shall also write 
	\[
	\norm{f(x)}_{L^p_x(\mi)} \qquad \text{instead of} \qquad\norm{f}_{L^p(\mi)}.
	\]
	This will be particularly useful when dealing with nested norms.

	\subsection{Differential Operators}
	
	\begin{deff}
		We shall generally identify the (complexification of the) enveloping algebra $U(G)$ of $\gf$ with the algebra of left-invariant differential operators. We shall fix a scalar product on $\gf$, and we shall endow $U(G)$ with the corresponding scalar product, namely the quotient of the natural scalar product on the (complexfication of the) tensor algebra over $\gf$.\footnote{The actual scalar product on $U(G)$ will not matter in the sequel.}
		
		Let $X$ be a left-invariant differential operator. We  denote with $X^R$ the right-invariant differential operator which induces the same point distribution as $X$ at $e$. In other words, $(X f)(e)=(X^R f)(e)$ for every $f\in C^\infty (G)$. We denote with $X^\dag$ the transpose of $X$ in the enveloping algebra $U(G)$. In other words, the mapping $X\mapsto X^\dag$ is the unique anti-automorphism of $U(G)$ (that is, such that $(XY)^\dag=Y^\dag X^\dag$) which extends the automorphism $X\mapsto -X$ of $\gf$. 
		Then, $X^\dag$ the formal transpose of $X$ (with respect to $\beta$), that is, the unique left-invariant differential operator such that
		\[
		\int_G (X f) g\,\dd \beta=\int_G f X^\dag g\,\dd \beta
		\]
		for every $f,g\in C^\infty_c(G)$. We denote with $X^*$ the formal adjoint of $X$, that is, $\overline X^\dag$. We define the formal transpose and the formal adjoint of right-invariant differential operators in a similar way. If $u$ is a distribution, we then define $X u$ so that
		\[
		\langle X u,\phi\rangle =\langle u, X^\dag \phi\rangle
		\]
		for every $\phi\in C^\infty_c(G)$. In this way, if $f\in C^\infty(G)$, then $(X f)\cdot \beta=X(f\cdot \beta)$.

		In order to simplify the notation, we write $X^{R\dag}$ instead of $(X^R)^\dag$, etc.
	\end{deff}
	
	Cf.~\cite[Proposition 2.2]{BCP} for a proof of the following result.
	
	\begin{prop}\label{prop:9}
		The following hold:
		\begin{enumerate}
			\item[\textnormal{(1)}]   $X^{R\dag}  =  X^{\dag R} $ for every $X\in U(G)$;
			
			\item[\textnormal{(2)}] $X\delta_e= X^R\delta_e=  (X^\dag \delta_e)\check{\;}$   for every $X\in U(G)$;
			
			\item[\textnormal{(3)}] $X^\dag f=(X^R \check f) \check{\,}$ for every $X\in U(G)$ and for every $f\in C^\infty(G)$.
		\end{enumerate}
	\end{prop}
	
	Notice that, by (2), 
	\[
	\begin{split}
		(Xf)*g=(f*X\delta_e)*g=f*(X\delta_e*g)=f*(X^{R} \delta_e*g)=  f*(X^R g)
	\end{split}
	\]
	under some reasonable conditions on $f$ and $g$ (which are needed to grant the associativity of convolution). 
	
	In addition, observe that, by our choice of the scalar product on $U(G)$, one has $\abs{X}=\abs{X^+}=\abs{\overline X}$, hence also $\abs{X}=\abs{X^*}$ for every $X\in U(G)$.
	
	\subsection{Filtrations and Weighted Subcoercive Operators}
	
	Throughout the paper, $(\gf_\lambda)_{\lambda\Meg 0}$ will denote an increasing filtration of $\gf$ (that is, $[\gf_\lambda, \gf_\mi]\subseteq \gf_{\lambda+\mi}$ for every $\lambda, \mi\Meg 0$) such that $\gf_\lambda=0$ for every $\lambda<1$, $\bigcup_{\lambda\Meg 0} \gf_\lambda=\gf$, and $\bigcap_{\mi>\lambda} \gf_\mi=\gf_\lambda$ for every $\lambda\Meg 0$.\footnote{We consider the full range $\lambda\Meg 0$  for notational convenience, in analogy with the filtration $(U_\lambda)$, for which $U_\lambda\neq \Set{0}$ for every $\lambda\Meg 0$.}
	
	\emph{In order for a weighted subcoercive operator adapted to the filtration $(\gf_\lambda)$ to exist, we shall assume that the set $\Lambda\coloneqq \Set{\lambda\Meg 1\colon \gf_\lambda \neq \bigcup_{\mi<\lambda} \gf_\mi}$ generate a $\Q$-vector space of dimension $1$.} In other words, setting $\dd_0\coloneqq \min \Lambda$, we require that $\Lambda \subseteq \Q \dd_0$.
	
	For every $X\in \gf$, we define  $\deg X\coloneqq \min\Set{\lambda\Meg0\colon X\in \gf_\lambda}$ and we call $\deg X$ the degree of $X$.
	Define, for every $\lambda>0$, $\gf_{\lambda^-}\coloneqq \bigcup_{\mi<\lambda} \gf_\mi$, $\gf_{*,\lambda}\coloneqq\gf_\lambda/\gf_{\lambda^-} $, and 
	\[
	\gf_*\coloneqq \bigoplus_{\lambda>0} \gf_{*,\lambda}.
	\]
	Define a Lie algebra structure on $\gf_*$ as follows: if $X=\sum_{\lambda>0} (X_\lambda+ \gf_{\lambda^-})$ and $Y= \sum_{\lambda>0} (Y_\lambda+ \gf_{\lambda^-})$ for some $(X_\lambda),(Y_\lambda)\in \prod_{\lambda>0} \gf_\lambda$, then
	\[
	[X,Y]\coloneqq \sum_{\lambda,\mi>0}\left(  [X_{\lambda},Y_{\mi}]+\gf_{(\lambda+\mi)^-}\right) .
	\]
	It is easily seen that $(\gf_{*,\lambda})_{\lambda>0}$ is a graduation of type $((0,+\infty),+)$ of $\gf_*$, that is, $[\gf_{*,\lambda}, \gf_{*,\mi}]\subseteq \gf_{*,\lambda+\mi}$ for every $\lambda,\mi>0$. One may then endow $\gf_*$ with the automorphic dilations $(\delta_r)_{r>0}$ defined so that $\delta_r(X)=r^\lambda X$ for every $X\in \gf_{*,\lambda}$ and for every $\lambda>0$. Thus, $\gf_*$ is the Lie algebra of some homogeneous group $G_*$, with homogeneous dimension $Q_*\coloneqq \sum_{\lambda>0} \dim \gf_{*,\lambda}$.
	
	We observe explicitly that we required $\gf_\lambda=\Set{0}$ for $\lambda<1$ in order for the control modulus $\abs{\,\cdot\,}_*$ (cf.~Definition~\ref{def:3} below) to induce a left-invariant \emph{distance} on $G$ (rather than a quasi-distance). There may be also good reasons to require $\gf_1\neq \Set{0}$. We preferred to avoid imposing this condition in order to keep a natural comparison with graded (or, more generally, homogeneous) groups: if $\gf$ has a graduation $(\tilde \gf_j)_{j\in \Z_+^*}$ (with integer degrees), then it is natural to set $\gf_\lambda=\bigoplus_{j=1}^{[\lambda]} \tilde \gf_j$ for every $\lambda \Meg 0$, but there is no guarantee that $\tilde \gf_1$ should be non-trivial.

	We now extend this filtration to the  enveloping algebra $U(G)$.	
	For every $\lambda\Meg 0$, define $U_\lambda$ as the vector space generated by the products of the form $X_1\cdots X_k$, for $k\Meg 0$, $X_1,\dots, X_k\in \gf$ and $\deg X_1+\cdots+ \deg X_k\meg \lambda$.\footnote{Thus, the identity operator, corresponding to the case $k=0$, belongs to all $U_\lambda$.} 
	Then, $(U_\lambda)$ is an increasing filtration of $U(G)$, that is, $U_\lambda, U_\mi\subseteq U_\lambda U_\mi\subseteq U_{\lambda+\mi}$ for every $\lambda, \mi\Meg 0$.
	For every $X\in U(G)$, we define $\deg X\coloneqq \min\Set{\lambda\Meg 0\colon X\in U_\lambda}$. Notice that $\gf\cap U_\lambda=\gf_\lambda$ for every $\lambda\Meg 0$ as a consequence of Proposition~\ref{prop:8} below (and its proof), so that this definition is consistent with the previous one.
	
	Define $U_{\lambda^-}\coloneqq \bigcup_{\mi<\lambda} U_\mi$ for $\lambda>0$, $U_{0^-}\coloneqq\Set{0}$, $U_{*,\lambda}\coloneqq U_\lambda/U_{\lambda^-}$ for every $\lambda\Meg 0$, and 
	\[
	U_*\coloneqq \bigoplus_{\lambda\Meg 0} U_{*,\lambda} .
	\] 
	We define an algebra structure on $U_*$ as follows:  if $X=\sum_{\lambda\Meg 0} (X_\lambda+ U_{\lambda^-})$ and $Y= \sum_{\lambda\Meg 0} (Y_\lambda+ U_{\lambda^-})$ for some $(X_\lambda),(Y_\lambda)\in \prod_{\lambda\Meg 0} U_\lambda$, then
	\[
	XY\coloneqq \sum_{\lambda,\mi\Meg 0}\left(  X_{\lambda}Y_{\mi}+U_{(\lambda+\mi)^-}\right) .
	\]
	It is easily seen that $U_*$ becomes a graded algebra of type $([0,+\infty),+)$ with this structure.
	
	Cf.~\cite[Proposition 4.2]{BCP} for a proof of the following result.
	
	\begin{prop}\label{prop:8}
		The canonical inclusions $\gf_\lambda \subseteq U_\lambda$, $\lambda> 0$, induce a linear mapping $\pi \colon \gf_*\to U_*$. The canonical extension $U(\pi)\colon U(G_*)\to U_*$ of $\pi$ is an  isomorphism of graded algebras.
	\end{prop}

	From now on, we shall identify $U_*$ and $U(G_*)$ by means of $U(\pi)$.

	\begin{deff}	
		We say that a family $(X_j)_{j\in J}$ of elements of $\gf$ is a minimal basis if, setting $Y_j\coloneqq X_j+ \gf_{(\deg X_j)^-}$, the family $(Y_j)$ induces a (homogeneous) basis of $\gf_*/[\gf_*,\gf_*]$.
		We denote with $\dd$ the least common multiple of the degrees of the $X_j$, $j\in J$.\footnote{Notice that, since we assumed that $\Lambda \subseteq \Q \dd_0$, the $\deg X_j$ all belong to $\Q \dd_0$, so that their least common multiple is well defined.}
		
		We say that $\Lc \in U(G)$ is weighted subcoercive if $ \Lc+\Lc^*+U_{\grado^-}$ is a positive Rockland operator on $G_*$.
	\end{deff}
	
	Notice that, if $(X_j)$ is a minimal basis of $\gf$, then $(X_j)$ generates $\gf$ as a Lie algebra, and also generates the filtrations $(\gf_\lambda)$ and $U_\lambda$ (cf.~\cite[Proposition 4.4]{BCP}). In other words, $\gf_\lambda$  is the vector space generated by the vector fields of the form $\ad(X_{j_1})\cdots \ad(X_{j_{k-1}})X_{j_k}$, where $k\Meg 1$, $j_1,\dots, j_k\in J$, and $\deg(X_{j_1})+\cdots +\deg(X_{j_k})\meg \lambda$. Analogously, $U_\lambda$ is the vector space generated by the differential operators for the form $X_{j_1}\cdots X_{j_k}$, where $k\Meg 0$, $j_1,\dots, j_k\in J$, and $\deg(X_{j_1})+\cdots +\deg(X_{j_k})\meg \lambda$.
	In particular, $(X_j)$ is a `reduced weighted algebraic basis' of $\gf$, in the terminology of~\cite{ElstRobinson}.
	
	Observe that $\dd$ does \emph{not} depend on the choice of $(X_j)$, since it is the least common multiple of the degrees of the non-zero elements of $\gf_*/[\gf_*,\gf_*]$.
	
	Finally, as shown in~\cite[Theorem 4.7]{BCP}, the above   definition of weighted subcoercive operators is consistent with the one given in~\cite{ElstRobinson}. In particular, if $\Lc$ is weighted subcoercive, then $\deg(\Lc)/\dd$ is an even integer. 
	
	We recall the following example from~\cite[Proposition 4.9]{BCP} for the construction of weighted subcoercive operators

	\begin{prop}\label{prop:4}
		Let $(X_j)_{j\in J}$ be a minimal basis of $\gf$. Then, $\Lc\coloneqq \sum_{j\in J} (X_j^{\dd/d_j})^\dag X^{\dd/d_j}_j$, where $d_j=\deg X_j$ for every $j\in J$, is a real,  positive and formally self-adjoint  weighted subcoercive operator.
	\end{prop}
	
	Here, by `positive' we mean that $\int \Lc f \overline f \,\dd \beta\Meg 0$ for every $f\in C^\infty_c(G)$.
	In particular, if $(\gf_\lambda)$ is the filtration generated by $\gf_1$, then $\Lc$ is a sub-Laplacian (a Laplacian, if $\gf_1=\gf$).

	We shall now define a control modulus on $G$. 
	\begin{deff}\label{def:3}
		Given an absolutely continuous curve $\gamma\colon [0,1]\to G$, we  define the content of $\gamma$ as the greatest lower bound of the $\eps>0$ such that 	
		\[
		\abs{P_\lambda \dd L_{\gamma(t)}^{-1}\gamma'(t)} \meg \min(\eps, \eps^{\lambda}) 
		\]
		for almost every $t\in [0,1]$ and for every $\lambda>0$, where $P_\lambda$ is the orthogonal projector of $\gf$ onto $\gf_\lambda \ominus \gf_{\lambda^-}$ (this is non-zero only for finitely many $\lambda>0$) and $L_{\gamma(t)}$ is the left translation by $\gamma(t)$. 
		
		Given $x\in G$, we shall define $\abs{x}_*$  as the greatest lower bound of the contents of the absolutely continuous curves $\gamma\colon [0,1]\to G$ such that $\gamma(0)=e$ and $\gamma(1)=x$.
		
		We  endow $G$ with the left-invariant distance $d(x,y)\coloneqq \abs{y^{-1}x}_*$.
	\end{deff} 
	
	Choosing a different scalar product on $\gf$   gives rise to bi-Lipschitz equivalent control moduli. Equivalence at infinity follows from the fact that all these control moduli are `connected moduli,'\footnote{In fact, every $x\in G$ may be written as $x_1\cdots x_k$, where $\abs{x_j}_*\meg 1$ for $j=1,\dots, k$, and $k\meg \abs{x}_*+1$.}  whereas equivalence near $e$ follows from~\cite[Corollary 6.5]{ElstRobinson}.  
	Notice that here we are not requiring $\gamma$ to be `horizontal.' One may also require $\gamma$ to be horizontal with respect to some fixed weighted algebraic basis compatible with the filtration $(\gf_\lambda)$, in the terminology of~\cite{ElstRobinson}, and still get an equivalent control modulus. This would provide a better mean value theorem for the corresponding `horizontal gradient,' but we shall not need this kind of more precise estimates.

	It is known that 
	\[
	\beta(B(e,r))\asymp r^{Q_*}
	\]
	for $ r\to 0^+$, 
	while
	\[
	\beta(B(e,r))\asymp r^{Q_G}
	\]
	for $r\to +\infty$.

	\begin{deff}
		From now on, we shall fix a \emph{formally self-adjoint} weighted subcoercive operator $\Lc$ with degree $\grado$. We shall denote with $(h_t)_{t>0}$ the corresponding heat kernel. In other words, $\ee^{-t\Lc}f=f*h_t$ for every $f\in L^2(G)$, where $(\ee^{-t\Lc})_{t>0}$ is the semigroup generated by the closure of $\Lc$, with initial domain $C^\infty_c(G)$, in $L^2(G)$ (cf.~Theorem~\ref{teo:7} below).
		
		For every $\omega\in \R$, we shall set $\Lc_\omega \coloneqq \Lc+ \omega I$.
	\end{deff}
	
	Cf.~\cite[Theorems 4.8 and 5.4]{BCP} for a proof of the following result.
	
	\begin{teo}\label{teo:7}
		There is $\omega\in \R$ such that the following hold:
		\begin{enumerate} 
			\item[\textnormal{(1)}] the closure of $\Lc$, with initial domain $C^\infty_c(G)$,  generates a semigroup of operators of $L^2(G)$; in addition, $\Lc$ is essentially self-adjoint on $C^\infty_c(G)$;
			
			\item[\textnormal{(2)}] for every $\lambda,\lambda'\Meg 0$ there are $b,C>0$ such that  
			\[
			\abs{X Y^R h_t(x)}\meg C \abs{X}\abs{Y} t^{-(Q_* + \deg(X)+\deg(Y))/\grado} \ee^{\omega t} \ee^{- b (\abs{x}_*^{\grado}/t)^{1/(\grado-1)}}
			\]
			for every $X\in U_\lambda$, for every $Y\in U_{\lambda'}$, and for every $x\in G$;
		\end{enumerate}
	\end{teo}

	\subsection{The Schwartz Space}\label{sec:Schwartz}
	
	\begin{deff}\label{def:1bis}
		We define $\Sr(G)$ as the space of $f\in C^\infty(G)$ such that the seminorms $\norm{ (1+ \abs{\,\cdot\,}_*)^k X^R f}_{L^1(G)}$, $k\in \N$, $X\in U(G)$, are finite, endowed with the corresponding topology.
		
		We denote with $\Sr'(G)$ the dual of $\Sr(G)$, endowed with the topology of uniform convergence on the bounded subsets of $\Sr(G)$.
	\end{deff}
	
	Notice that, as in~\cite{Calzi4,CalziRizzo}, we denote the Schwartz space as $\Sr(G)$ instead of $\Sc(G)$ since we wish to avoid any confusion with the \emph{entirely different} `Schwartz space' $\Sc(G)$ which was used in~\cite{BCP, Calzi2, Calzi3}.
	
	Cf.~\cite[Theorem 2.15 and Corollary 2.16]{Calzi4} for a proof of the following results.
	
	\begin{teo}\label{teo:8bis}
		The following hold:
		\begin{enumerate}
			\item[\textnormal{(1)}] $\Sr(G)$ is a nuclear Fréchet space;
			
			\item[\textnormal{(2)}] $\Sr(G)$ is reflexive;
			
			\item[\textnormal{(3)}] the bounded subsets of $\Sr(G)$ are relatively compact;
			
			\item[\textnormal{(4)}] $\Sr(G)$ is a Fréchet $*$-algebra under convolution and under pointwise multiplication;
			
			\item[\textnormal{(5)}] for every $p\in [1,\infty]$, $\Sr(G)$ is the space of $f\in C^\infty(G)$ such that the seminorms  $\norm{(1+ \abs{\,\cdot\,}_*)^k  X Y^R f}_{L^p(G)}$, $k\in \N$, $X,Y\in U(G)$ (resp.\ $\norm{(1+ \abs{\,\cdot\,}_*)^k X   f}_{L^p(G)}$, $k\in \N$, $X\in U(G)$; $\norm{(1+ \abs{\,\cdot\,}_*)^k  X^R f}_{L^p(G)}$, $k\in \N$, $X\in U(G)$), are finite, and has the corresponding topology.
		\end{enumerate} 
	\end{teo}

	\begin{cor}
		$\Sr'(G)$ is a complete, reflexive, bornological, and nuclear space.  The bounded subsets of $\Sr'(G)$ are relatively compact.
	\end{cor}

	\begin{deff}
		Define $\Oc'_{C,L}(G)$ and $\Oc'_{C,R}(G)$ as the spaces of $u\in \Dc'(G)$ such that the families $((1+\abs{x}_*)^k L_x u)_{x\in G}$ and $((1+\abs{x}_*)^k R_x u)_{x\in G}$  are bounded in $\Dc'(G)$ (with respect to the topology of uniform convergence on the bounded subset of $\Dc(G)$, or, equivalently, with respect to the weak dual topology), respectively, for every $k\in\N$. Here, $L_x$ and $R_x$ denote the left and right translation operators, so that $\langle L_x u, \phi \rangle= \langle u, \phi(x\,\cdot\,)\rangle$ and $\langle R_x u,\phi\rangle=\langle u, \phi(\,\cdot\,x^{-1})\rangle$ for every $\phi \in C^\infty_c(G)$. We endow $\Oc'_{C,L}(G)$ and $\Oc'_{C,R}(G)$ with the corresponding topology.
	\end{deff}
	
	In the following result we collect, without proof, some basic properties of the spaces $\Oc'_{C,L}(G)$ and $\Oc'_{C,R}(G)$. Cf.~\cite{Schwartz,Grothendieck} for more information on these spaces in the case $G=\R^n$.
	
	\begin{prop}
		The following hold:
		\begin{enumerate}
			\item[\textnormal{(1)}] $\Oc'_{C,R}(G)$ and $\Oc'_{C,L}(G)$ are complete nuclear Lusin spaces;
			
			\item[\textnormal{(2)}] every bounded subset of $\Oc'_{C,R}(G)$ and $\Oc'_{C,L}(G)$ is relatively compact,  carries the topology induced by $\Dc'(G)$, and is contained in the closure of some bounded subset of $\Dc(G)$;
			
			\item[\textnormal{(3)}] the mapping $u\mapsto \check u$ induces an isomorphism of $\Oc'_{C,R}(G)$ onto $\Oc'_{C,L}(G)$;
			
			\item[\textnormal{(4)}] convolution induces hypocontinuous  bilinear mappings
			\[
			\begin{aligned}
				&\Sr(G)\times \Oc'_{C,R}(G)\to \Sr(G) & & \Oc'_{C,L}(G)\times \Sr(G)\to \Sr(G)\\
				&\Sr'(G)\times \Oc'_{C,L}(G)\to \Sr'(G) & & \Oc'_{C,R}(G)\times \Sr'(G)\to \Sr'(G)\\
				&\Oc'_{C,R}(G)\times \Oc'_{C,R}(G)\to \Oc'_{C,R}(G) & &\Oc'_{C,L}(G)\times \Oc'_{C,L}(G)\to \Oc'_{C,L}(G)
			\end{aligned}
			\]
			relative to the bounded subsets of each factor;
			
			\item[\textnormal{(5)}] the mappings $\Phi_R\colon \Oc_{C,R}'(G)\to \Lc(\Sr(G))$ and $\Phi_L\colon\Oc_{C,L}'(G)\to \Lc(\Sr(G))$,  defined so that $\Phi_R(u)\phi=\phi*u$ and $\Phi_L(u')\phi=u'*\phi$ for every $\phi\in \Sr(G)$, induce isomorphisms of $\Oc'_{C,R}(G)$ onto the space of left-invariant endomorphisms of $\Sr(G)$ and of  $\Oc'_{C,L}(G)$ onto the space of right-invariant endomorphisms of $\Sr(G)$, respectively;
			
			\item[\textnormal{(6)}] if $(u_1,u_2,u_3)$ belongs to either 
			\[
			\Oc'_{C,R}(G)\times \Oc'_{C,R}(G)\times \Sr'(G)  \qquad \text{or} \qquad \Sr'(G)\times \Oc'_{C,L}(G)\times \Oc'_{C,L}(G),
			\]
			then
			\[
			(u_1*u_2)*u_3=u_1*(u_2*u_3).
			\]
		\end{enumerate}
	\end{prop}

	\subsection{`Goodman' Sobolev Spaces}

	\begin{deff}
		Suppose  $\alpha\Meg0$ and $p\in [1,\infty]$. We define $L^p_0(G)$ as the closure of $C_c(G)$ in $L^p(G)$. Then, we define the Sobolev space $W^{\alpha,p}(G)$ (resp.\ $W^{\alpha,p}_0(G)$) as the space of $f\in L^p(G)$ (resp.\ $f\in L^p_0(G)$) such that $X f\in L^p(G)$ (resp.\ $X f\in L^p_0(G)$) for every $X\in U_\alpha$, endowed with the norm 
		\[
		f\mapsto \max_{\substack{X\in U_\alpha, \;  \abs{X}\meg 1}} \norm{X f}_{L^p(G)}.
		\]
		
		We define $W^{\alpha,p}_\loc(G)$ as the space of $f\in L^p_\loc(G)$ such that $f\phi\in W^{\alpha,p}$ for every $\phi\in C^\infty_c(G)$.
	\end{deff}
	Observe that by definition $L^p_0(G)= L^p(G)$ if $p<\infty$, while $L^\infty_0(G) =  C_0(G)$.

	\subsection{Auxiliary Results}
	
	We shall now consider some results on the functional calculus associated with $\Lc$.
	Observe first that, by Theorem~\ref{teo:7}, the operator $\Lc$, with initial domain $C^\infty_c(G)$, is essentially self-adjoint on $L^2(G)$, so that we may consider the associated functional calculus. In particular,  for every Borel function $m$ on $\R$ with polynomial growth, that is, such that $(1+\abs{\,\cdot\,})^{-k} m$ is bounded for some $k\in\N$, there is a unique  distribution $\Kc_\Lc(m)$ on $G$ such that $m(\Lc)\phi=\phi*\Kc_\Lc(m)$ for every $\phi\in C^\infty_c(G)$. In fact, $\Kc_\Lc(m)$ is necessarily a tempered distribution, so that the previous equality may be extended to every $\phi$ in the Schwartz space.

	\begin{deff}
		We define $D$ as the maximum of $Q_*$ and $Q_G$.\footnote{Recall that $Q_G$ is defined so that $\beta(B(e,r))\asymp r^{Q_G}$ for $r\to +\infty$, while $Q_*$ is the homogeneous dimension of $G_*$, so that $\beta(B(e,r))\asymp r^{Q_*}$ for $r\to 0^+$}.
	\end{deff}

	Cf.~\cite[Lemma 2.25]{Calzi4} for a proof of the following simple result.
	
	\begin{lem}\label{lem:37}
		There is a constant $C>0$ such that
		\[
		\norm{(1+\abs{\,\cdot\,}_*/t)^{-\alpha-D/p}  }_{L^p(G)}\meg 3 \left( \frac{C}{1-2^{-p\alpha}}\right) ^{1/p} t^{Q_*/p}\max(1,t)^{(Q_G-Q*)/p}
		\]
		for every $p\in (0,\infty]$ and for every $t,\alpha>0$.
	\end{lem}
	 
	Cf.~\cite[Corollary 2.27]{Calzi4} for a proof of the following result.
	
	\begin{prop}\label{cor:12} 
		Take $p_0\in (0,1]$, $\lambda,c\Meg 0$ and $\rho,\eps>0$. Then, there is a constant $C>0$ such that for every $p\in [p_0,\infty]$, for every $s\Meg \eps+ c+D(1/p-1/2)_+$, for every $m\in B^{\infty,\infty}_s(\R)$ supported in $[-\rho,\rho]$, for every $r\in (0,\rho]$, and for every $X,Y\in U_\lambda$,
		\[
		\norm{XY^R\Kc_{r^{\grado}\Lc } (m ) (1+\abs{\,\cdot\,}_*/r)^c}_{L^p(G)} \meg C r^{-\deg(X)-\deg(Y)+(1/p-1)Q_*}\abs{X}\abs{Y}\norm{m}_{B^{\infty,\infty}_s(\R)}.
		\]
	\end{prop}

	\begin{deff}
		For every $p\in (0,+\infty]$, for every $r>0$, and for every $\beta$-measurable function $f$ from $G$ into $\C$, we define 
		\[
		(\Nc_{p,a,r} f)(x)\coloneqq r^{-Q_*/p}\norm{ f(y) (1+d(x,y)/r)^{-a}  }_{L^p_y(G)}=r^{-Q_*/p}\norm{ \abs{f(y)} (1+d(x,y)/r)^{-a}  }_{L^p_y(G)}
		\]
		for every $x\in G$. We define also $(\Nc_{p,a}  f)(x)\coloneqq \sup_{r>0} (1+r)^{(Q_*-D)/p} (\Nc_{p,a,r} f)(x)$.
	\end{deff}

	Cf.~\cite[Corollary 2.30]{Calzi4} for a proof of the following result.
	
	\begin{prop}\label{cor:18}
		Take $p_0\in (0,\infty)$, $p\in (p_0,\infty)$, $q\in [p_0,\infty]$, $a>D/p_0$, $\mi$ a positive  Radon measure on $(0,+\infty)$ with bounded support,  and $\kappa\colon (0,+\infty )\to (0,+\infty)$ a $\mi$-measurable function. Then, there is a constant $C>0$ such that
		\[
		\norm{ (1+\kappa(t))^{(Q_*-D)/p_0} \Nc_{p_0,a,\kappa(t)}(f(t,\,\cdot\,))(x)  }_{L^{q,p}_{t,x}(\mi,\beta)}\meg C\norm{f}_{L^{q,p}(\mi,\beta)}
		\]
		for every $(\mi\otimes \beta)$-measurable function $f\colon (0,+\infty)\times G\to \C$.
	\end{prop}

	We shall now present some technical lemmas which show how one may  use the estimates provided in Proposition~\ref{cor:12} to deal with convolution kernels associated with Schwartz multipliers, and also how one may control the convolution of convolution kernels associated with the functional calculi of two distinct formally self-adjoint weighted sub-coercive operators.

	\begin{deff}
		Take $\lambda,\mi, a\Meg0 $, $\rho>0$, and $p\in [1,\infty]$. We say that a family $(f_t)_{t\in (0,\rho]}$ of elements of $W^{\mi,p}_\loc(G)$ is a $(\lambda,\mi,a,\rho,p)$-standard system if the following hold:
		\begin{itemize}
			\item
			\[
			\norm{Y f_{t} (1+\abs{\,\cdot\,}_*/t )^a  }_{L^p(G)}\meg \abs{Y} t^{-\deg(Y)-Q_*/p'}
			\]
			for every $t\in (0,\rho]$ and for every $Y\in U_\mi$;
			
			\item there are a basis $(X_j)_{j\in J}$ of    $U_\lambda$ and a family $(\widetilde f_{j,t})_{j\in J, t\in (0,\rho)}$ of elements of $L^{1}_\loc(G)$ such that
			\[
			f_t= \sum_{j\in J} X_j \widetilde f_{j,t}
			\]
			for every $t\in (0,\rho)$, and such that\footnote{Notice that $\deg(X_j)\meg \lambda$, and that equality need not occur.}
			\[
			\sum_{j\in J}\norm*{  \widetilde f_{j,t} (1+\abs{\,\cdot\,}_*/t)^a  }_{L^1(G)}\meg \frac{1}{\abs{X_j}} t^{\lambda}
			\]
			for every $t\in (0,\rho)$.
		\end{itemize}
	\end{deff}

	Cf.~\cite[Lemma 2.35 and Corollary 2.36]{Calzi4} for a proof of the following results.
	
	\begin{lem}\label{lem:40}
		Take $\lambda ,\mi\Meg 0$, $a,\rho,\rho'>0$, and $p\in [1,\infty]$. Then, there is a constant $C>0$ such that
		\[
		\norm{ (f_t*g^*_s)  (1+\abs{\,\cdot\,}_*/(t+s))^a }_{L^p(G)}\meg C  \min\left((t/s)^\lambda, (s/t)^\mi  \right)    (t+s)^{-Q_*/p'}
		\]
		for every $t\in (0,\rho]$, for every $s\in (0,\rho]$, for every $(\lambda,\mi,a,\rho,p)$-standard system $(f_t)$, and for every $(\mi,\lambda,a,\rho',p)$-standard system $(g_s)$.
	\end{lem}
	 
	\begin{lem}\label{cor:13}
		Take $\lambda, \mi  \Meg 0$, $a,\rho,\rho'>0$, $\phi\in C^\infty_c(\R)$, and $p\in [1,\infty]$. Then, there is a constant $C>0$ such that
		\[
		\norm{ (f_t*g^*_s)  (1+\abs{\,\cdot\,}_*/t)^a }_{L^p(G)}, \norm{ (g_s*f^*_t)  (1+\abs{\,\cdot\,}_*/t)^a }_{L^p(G)}\meg C   \min((t/s)^\lambda,(s/t)^\mi)  (t+s)^{-Q_*/p'}
		\]
		for every $t\in (0,\rho]$, for every $s\in (0,\rho]$, for every $(\lambda+a,\mi,a,\rho,p)$-standard system $(f_t)$, and for every $(\mi,\lambda+a,a,\rho',p)$-standard system $(g_s)$.
	\end{lem}
	
	 Cf.~\cite[Lemma 2.46]{Calzi4} for a proof of the following result.
	
	\begin{lem}\label{lem:25d}
		Take $\eta,\delta>0$, $\eps\in (0,1)$, and $q\in (0,\infty]$. Then, there is a constant $C>0$ such that
		\[
		\norm*{\sum_{j\in \N} \eps^{\delta(j-j')_+ + \eta (j'-j)_+} \abs{a_j}}_{\ell^q_{j'}(\N)}\meg C \norm{a_j }_{\ell^q_j(\N)}
		\]
		for every $(a_j)\in \C^{\N}$.  
	\end{lem}

	\subsection{The Spaces $\Cc^q(\eps)$}
	
	In order to deal with the Triebel--Lizorkin spaces $F^{\infty,q}_\alpha(G)$, we shall need to introduce the spaces $\Cc^q(\mi,a)$ in order to replace the spaces $L^{q,\infty}(\mi,\beta)$. We shall present a simplified version of~\cite{Calzi3,Calzi4}, which, in turn, follow~\cite{FrazierJawerth,Rychkov}.  
	
	\begin{deff}
		Take $\eps\in (0,1)$ and $q\in (0,\infty]$. Then, we define $\Cc^q(\eps)$ as the space of $(f_j)\in L^1_\loc(G)^{\N}$ such that  
		\[
		\sup_{N\in \N, x\in G} \eps^{-N Q_*/q}\norm{  \chi_{[N,+\infty)\times B(x,\eps^{N})}(j,x') f_j(x') }_{L^q_{(j,x')}(\N\times G)}
		\]
		is finite, endowed with the corresponding quasi-norm. 
	\end{deff}
	 
	The following result is an analogue of Proposition~\ref{cor:18}. Cf.~\cite[Lemma 2.48]{Calzi4} for a proof.
	
	\begin{lem}\label{lem:2c} 
		Take $ p_0\in (0,\infty]$, $\eps\in (0,1)$, $b>D/p_0$, $c>1$,    and a sequence $(\kappa_j)$ such that $\kappa_j\in (0,c \eps^j]\cup [1/c,+\infty)$ for every $j\in\N$. 
		Then, there is   $C>0$ such that, for every $q\in [p_0,\infty]$ and for every  $(f_j)\in L^1_\loc(G)^\N$,
		\[
		\begin{split}
			\norm{ (1+\kappa_j)^{-b+Q_*/p_0}[\Nc_{p_0,b,\kappa_j}  f_j](x) }_{\Cc^q_{(j,x)}(\eps)} \meg C  \norm{  f_j(x) }_{\Cc^q_{(j,x)}(\eps)} .
		\end{split}
		\]
	\end{lem}
	 
	Cf.~\cite[Lemma 2.49]{Calzi4} for a proof of the following result.
	
	\begin{lem}\label{lem:25c}
		Take $q\in (0,\infty]$, $\eps,\eps'\in (0,1)$, $ c>1$, $a,  \eta,\delta>0$,  and $b>D/q$.  Then, there is a constant $C>0$ such that 
		\[
		\norm*{  \sum_{j\in\N} \frac{\eps^{j\delta} \eps'^{k\eta}}{(\eps^j+\eps'^{k})^{\delta+\eta}} (\Nc_{q,b,c \eps^j}f_j)(x)}_{\Cc^q_{(k,x)}(\eps')}\meg  C \norm{   f_j(x)  }_{\Cc^q_{(j,x)}(\eps)}
		\]
		for every  $(f_j)\in L^1_\loc(G)^{\N}$.
	\end{lem}

	\begin{deff}
		Take $q\in (0,\infty]$ and $\eps,\eta\in (0,1)$. Define 
		\[
		(m^q_{\eta,\eps,j} f)(x)\coloneqq \sup\Set{\lambda>0\colon  \beta \Big(\Set{y\in B(x,\eps^j) \colon  \norm{   f_k(x) }_{\ell^q_k(j+\N)}>\lambda  }\Big)> \eta\beta(B(e,\eps^j))}
		\]
		and
		\[
		(m^q_{\eta,\eps} f)(x)\coloneqq \sup_{j\in\N} (m^q_{\eta,\eps,j} f )(x)
		\]
		for every  $f=(f_j)\in L^1_\loc(G)^\N$ and for every $x\in G$.
	\end{deff}
	
	\begin{deff}
		Take $\eps\in (0,1)$, $\delta>0$, and $R\Meg 2$. Given $K\subseteq \N\times \N$, we say that $(x_{j,k})\in G^{K}$ is a reduced $(\eps,\delta,R)$-lattice on $G$ if the $B(x_{j,k},\delta\eps^j)$, $j\in K_k$, are pairwise disjoint and the $B(x_{j,k}, \delta R\eps^j)$, $j\in K_k$, cover $G$ for every $k\in\N$, where $K_k\coloneqq \Set{j\in\N\colon (j,k)\in K}$.
	\end{deff}
	
	Existence of reduced $(\eps,\delta,2)$-lattices is trivial, since it suffices to take $(x_{j,k})$ so that $(x_{j,k})_k$ is a maximally $2\delta\eps^j$-separated for every $j\in \N$. If $G$ is not compact, we may therefore assume that $K=\N\times \N$.
	
	Cf.~\cite[Remark 3.7]{Calzi3} for a discussion of the measurability of $m^q_{\eta,\eps,j}$ and $m^q_{\eta,\eps}$. Cf.~\cite[Lemmas 2.53 and 5.5]{Calzi4} for a proof of the following results.
	
	\begin{lem}\label{lem:59b}
		Take $p\in (0,\infty)$, $q\in (0,\infty]$, and  $\eps,\eta\in (0,1)$. Then there is a constant $C>0$ such that for every $f=(f_j)\in L^1_\loc(G)^\N$ there is a sequence $(E_j)$ of $\beta$- measurable subsets  of $  G$ such that 
		\[
		\beta( E_j\cap B(x,\eps^j))\Meg (1-\eta) \beta(B(e,\eps^j))
		\]
		for every $x\in G$ and for  every $j\in\N$, and such that
		\[
		\norm{\chi_{E_j} f_j(x)}_{L^{q,p}_{j,x}(\N,G)}\meg \norm{m^q_{\eta,\eps}f}_{L^p(G)}\meg C\norm{f_j(x)}_{L^{q,p}_{j,x}(\N,G)}
		\]
		and
		\[
		\frac{1}{C}\norm{\chi_{E_j} f_j(x)}_{L^{q,\infty}_{j,x}(\N,G)}\meg \norm{m^q_{\eta,\eps} f}_{L^\infty(G)}\meg C\norm{f}_{\Cc^{q}(\eps)}.
		\]
		More precisely,
		\[
		 \norm{  (\chi_{E_k} f_k)(x) }_{\ell^q_k( \N)}\meg (m^q_{\eta,\eps} f)(x)
		\]
		for every $x\in G$.
	\end{lem}
	
	\begin{lem}\label{lem:60b}
		Take $p\in (0,\infty)$, $q\in (0,\infty]$, $\eta,\eps\in (0,1)$, $\delta_0>0$ and $R\Meg 2$. Then, there is a constant $C>0$ such that the following hold. Take a reduced $(\eps,\delta,R)$-lattice $(x_{j,k})_{(j,k)\in J}$ on $G$, with $\delta\in (0,\delta_0]$, and define $\iota\colon \C^{J}\to \C^{\N\times G}$ so that $\iota(\lambda) (j,x)\coloneqq \sum_{k} \lambda_{j,k}\chi_{B(x_{j,k}, \delta \eps^j)}(x)$ for every $\lambda\in \C^{J}$ and for every $(j,x)\in \N\times G$. In addition, for every $j\in\N$ take a $\beta$-measurable subsets $E_j$ of $  G$ such that 
		\[
		\beta(E_j\cap B(x_{j,k},\delta\eps^j))\Meg \eta \beta(B(e,\delta\eps^j))  
		\]
		for every  $(j,k)\in J$. Then,  
		\[
		\norm{ \iota(\lambda) }_{L^{q,p}(\N,G)}\meg C \norm{\chi_{E_j}(x) \iota(\lambda)(j,x) }_{ L^{q,p}_{(j,x)}(\N,G)}
		\]
		and
		\[
		\norm{ \iota(\lambda) }_{\Cc^{q}(\eps)}\meg C \norm{\chi_{E_j}(x) \iota(\lambda)(j,x) }_{ L^{q,\infty }_{j,x}(\N, G)}
		\]
		for every $\lambda\in \C^{J}$.
	\end{lem}
	
\subsection{Besov and Triebel--Lizorkin Spaces}

\begin{deff}\label{def:12}
	Take $\eps\in (0,1)$ and a formally self-adjoint weighted subcoercive operator $\Lc'$ on $G$. We define $\Psi_\eps(\Lc')$ the set of $(\psi_j)\in \Sr(G)^\N$ such that there is some bounded sequence $(\phi_j)$ of elements of $C^\infty_c(\R)$ such that 
	\[
	0\not \in \overline{\bigcup_{j\Meg 1}\supp \phi_j }, \qquad 
	\inf_{\lambda\in \R} \sum_{j\in\N} \abs{\phi_j(\eps^{j }\lambda)}>0,
	\]
	and
	\[
	\psi_j= \Kc_{\eps^{j }\Lc'}(\phi_j )
	\]
	for every $j\in\N$.
	We define $\widetilde \Psi_\eps(\Lc')$ as the set of $(\psi_j)\in \Psi_\eps(\Lc')$ such that $\sum_{j\in\N} \psi_j=\delta_e$ in $\Sr'(G)$; equivalently, with the above notation, $\sum_{j\in\N} \phi_j(\eps^j\,\cdot\,)=1$ on $\sigma(\Lc')$.
	
	We shall simply write $\Psi_\eps$ and $\widetilde \Psi_\eps$ instead of $\Psi_\eps(\Lc)$ and $\widetilde \Psi_\eps(\Lc)$, respectively.
\end{deff}

\begin{deff}
		Take $p,q\in (0,\infty]$ and $\alpha\in \R$. Take $\eps\in (0,1)$ and $(\psi_{j})\in \Psi_\eps$. Then, we define $B^{p,q}_\alpha(G)$ as the space of $f\in \Sr'(G)$ such that
	\[
	\norm*{\eps^{-j\alpha/d} \norm{ f*\psi_{j}}_{L^p(G)}}_{\ell^q_j(\N)}<\infty,
	\]
	endowed with the corresponding quasi-norm. We define $\mathring B^{p,q}_\alpha(G)$ as the closure of $\Sr(G)$ in $B^{p,q}_\alpha(G)$.
	
	If  $p<\infty$, then we define $F^{p,q}_\alpha(G)$ as the space of $f\in \Sr'(G)$ such that
	\[
	\norm*{ \norm{\eps^{-j\alpha/d} f*\psi_{j}}_{\ell^q_j(\N)}}_{L^p(G)}<\infty,
	\]
	endowed with the corresponding quasi-norm. We define $F^{\infty,q}_\alpha(G)$ as the space of $f\in \Sr'(G)$ such that
	\[
	\norm{\eps^{-j\alpha/d} f*\psi_{j}}_{\Cc^q_{(j,x)}(\eps^{1/\grado})}<\infty,
	\]
	endowed with the corresponding quasi-norm.
	We define $\mathring F^{p,q}_\alpha(G)$ as the closure of $\Sr(G)$ in $F^{p,q}_\alpha(G)$.
\end{deff}

\subsection{Discretization}

\begin{deff}
	Take $p,q\in (0,\infty]$, $\alpha\in \R$, a subset $K$ of $\N\times \N$ and $\eps\in (0,1)$. Then, we define  $b^{p,q}_{\alpha}(\eps)$ as the space of $\lambda\in \C^{\N\times \N}$ such that
	\[
	\norm{\eps^{j (Q_*/p-\alpha)}\lambda_{j,k}  }_{\ell^{p,q}_{k,j}(\N,\N)}<\infty,
	\]
	endowed with the corresponding quasi-norm.  We define $b^{p,q}_\alpha(\eps,K)$ as the space of $\lambda\in \C^K$ such that $\lambda^0\in b^{p,q}_\alpha(\eps)$, endowed with the corresponding quasi-norm, where $\lambda^0\in \C^{\N\times \N}$ is chosen so that $\lambda^0_{j,k}=\lambda_{j,k}$ when $(j,k)\in K$, while $\lambda^0_{j,k}=0$ otherwise. 
	
	In addition, take $\delta>0$, $R\Meg 2$, and a reduced $(\eps,\delta,R)$-lattice $(x_{j,k})_{(j,k)\in K'}$. Then, for $p\in (0,\infty)$ we define $f^{p,q}_\alpha((x_{j,k}))$ as the space of $\lambda\in \C^{K'}$ such that
	\[
	\norm{ \chi_{K'}(j,k) \eps^{-j \alpha} \lambda_{j,k} \chi_{B(x_{j,k}, \delta  \eps^j)}(x) }_{L^{q,q,p}_{j,k,x}(\N,\N,G)}<\infty,
	\]
	(with some abuse of notation) endowed with the corresponding quasi-norm. Finally, we define $f^{\infty,q}_\alpha((x_{j,k}))$ as the space of $\lambda\in \C^{K'}$ such that 
	\[
	\norm*{ \norm{ \chi_{K'}(j,k)\eps^{-j \alpha} \lambda_{j,k} \chi_{B(x_{j ,k}, \delta  \eps^{j})}(x)}_{\ell^q_k(\N)} }_{\Cc^{q}_{(j,x)}(\eps)}<\infty,
	\]
	endowed with the corresponding quasi-norm.
	
	We denote with  $\mathring b^{p,q}_\alpha(\eps)$, $\mathring b^{p,q}_\alpha(\eps,K)$, and $\mathring f^{p,q}_\alpha((x_{j,k}))$ the closures of $\C^{(\N\times \N)}$, $\C^{(K)}$, and $\C^{(K')}$ in  $  b^{p,q}_\alpha(\eps )$, $b^{p,q}_\alpha(\eps,K)$, and $f^{p,q}_\alpha((x_{j,k}))$, respectively.
\end{deff}

Observe that, even though $(x_{j,k})$ determines $\eps$ uniquely, it does not determine $\delta,R$. As Remark~\ref{oss:6} below shows, nonetheless, the space $f^{p,q}_\alpha((x_{j,k}))$ does not depend on $\delta,R$ (its norm does, though).

We also observe that $b^{p,q}_\alpha(\eps,K)$ and $\mathring b^{p,q}_\alpha(\eps,K)$ are retracts of $b^{p,q}_\alpha(\eps)$ and $\mathring b^{p,q}_\alpha(\eps)$, respectively  (the retraction being the restriction mapping $b^{p,q}_\alpha(\eps)\to b^{p,q}_\alpha(\eps,K)$ and the corresponding section being the mapping $\lambda \mapsto \lambda^0$), so that basically every assertion concerning the former spaces follow from the corresponding one on the latter spaces. In other words, we shall generally prove our assertions only for the spaces $b^{p,q}_\alpha(\eps)$ and $\mathring b^{p,q}_\alpha(\eps)$. In addition, for simplicity, we shall \emph{not} write $ \chi_{K'}(j,k)$ in the quasi-norms appearing in the definition of  $f^{p,q}_\alpha((x_{j,k}))$; in other words, we shall convene to consider $0$ the terms in those quasi-norms which are not defined.  In fact, even with the term $ \chi_{K'}(j,k)$ there would still be some abuse of notation.

Cf.~\cite[Remark 4.5]{Calzi4} for a proof of the following result.

\begin{oss}\label{oss:6}
	Take $p,q\in (0,+\infty]$ with $p<\infty$, $\alpha\in \R$, $\eps\in (0,1)$, $\delta_0>0$, $R\Meg 2$, and $c>1$. Then, there is a constant $C>0$ such that
	\[
	\norm{ \eps^{-j \alpha} \lambda_{j,k} \chi_{B(x_{j,k}, c\delta \eps^j)}(x) }_{L^{q,q,p}_{j,k,x}(\N,\N,G)}\meg C \norm{ \eps^{-j \alpha} \lambda_{j,k} \chi_{B(x_{j,k}, (\delta/c) \eps^j)}(x) }_{L^{q,q,p}_{j,k,x}(\N,\N,G)}
	\]
	and 
	\[
	\norm*{ \norm{\eps^{-j \alpha} \lambda_{j,k} \chi_{B(x_{j,k},c \delta  \eps^j)}(x)}_{\ell^q_k(\N)} }_{\Cc^{q}_{(j,x)}(\eps)}\meg C\norm*{ \norm{\eps^{-j \alpha} \lambda_{j,k} \chi_{B(x_{j,k}, (\delta /c) \eps^j)}(x)}_{\ell^q_k(\N)} }_{\Cc^{q}_{(j,x)}(\eps)}
	\]
	for every reduced $(\eps,\delta,R)$-lattice $(x_{j,k})_{(j,k)\in K}$, with $\delta\in (0,\delta_0]$, and for every $\lambda\in \C^{K}$. 
\end{oss}

In particular, it does not matter which radius we take for the balls $B(x_{j,k}, \delta \eps^j)$ appearing in the definition of the spaces $f^{p,q}_\alpha((x_{j,k}))$, as long as it is comparable to $\delta \eps^j$. Taking exactly $\delta \eps^j$ makes the balls $B(x_{j,k}, \delta \eps^j)$  disjoint, so that we may replace the inner $\ell^q_k(\N)$ norm with a sum in $k$ (even without taking modules); taking $R\delta \eps^j$ makes the quasi-norm of  $f^{p,q}_\alpha((x_{j,k}))$ better for some proofs.

Cf.~\cite[Proposition 4.8]{Calzi4} for a proof of the following `sampling theorem.'

\begin{prop}\label{prop:14}
	Take $p,q\in (0,+\infty]$, $\alpha\in \R$, $\eps\in (0,1)$, $\delta_0>0$ and $R\Meg 2$.  
	Take $(\psi_{j})\in \Psi_\eps$. Then, there is a constant $C>0$ such that, for every reduced $(\eps^{1/\grado},\delta,R)$-lattice $(x_{j,k})_{(j,k)\in K}$ on $G$, with $\delta\in (0,\delta_0]$,
	if we define $U_{\delta R,\eps}$ as the set of $(z_{j,k})\in G^{K}$ such that $z_{j,k}\in \overline B(e,\delta R\eps^j)$ for every $(j,k)\in K$, and
	\[
	S_{(x_{j,k})} \colon \Sr'(G)\ni f \mapsto ((f*\psi_{2,j})(x_{j,k}))_{j,k}\in \C^{K},
	\]
	then
	\[
	\frac{1}{C} \norm{f}_{B^{p,q}_\alpha(G)}\meg \delta^{Q_*/p} \max_{z\in U_{\delta R,\eps}}\norm{ S_{(x_{j,k}z_{j,k})}f}_{b^{p,q}_\alpha(\eps^{1/\grado},K)}\meg C \norm{f}_{B^{p,q}_\alpha(G)}
	\] 
	and
	\[
	\frac{1}{C} \norm{f}_{F^{p,q}_\alpha(G)}\meg \max_{z\in U_{\delta R,\eps}}  \norm{ S_{(x_{j,k}z_{j,k})}f }_{f^{p,q}_\alpha((x_{j,k}))}\meg C \norm{f}_{F^{p,q}_\alpha(G)}
	\]
	for every $f\in\Sr'(G)$. In addition, there is $\delta_->0$ such that, if $\delta\meg \delta_-$, then
	\[
	\frac{1}{C} \norm{f}_{B^{p,q}_\alpha(G)}\meg \delta^{Q_*/p} \min_{z\in U_{\delta R,\eps}}\norm{ S_{(x_{j,k}z_{j,k})}f}_{b^{p,q}_\alpha(\eps^{1/\grado},K)}\meg C \norm{f}_{B^{p,q}_\alpha(G)}
	\] 
	for every $f\in B^{p,q}_\alpha(G)$,  and
	\[
	\frac{1}{C} \norm{f}_{F^{p,q}_\alpha(G)}\meg \min_{z\in U_{\delta R,\eps}}  \norm{ S_{(x_{j,k}z_{j,k})}f }_{f^{p,q}_\alpha((x_{j,k}))}\meg C \norm{f}_{F^{p,q}_\alpha(G)}
	\]
	for every $f\in F^{p,q}_\alpha(G)$.
	
	Finally, the mapping $S_{(x_{j,k})}$ induces isomorphisms of $\mathring B^{p,q}_\alpha(G)$ and $\mathring F^{p,q}_\alpha(G)$ onto closed vector subspaces of $\mathring b^{p,q}_\alpha(\eps^{1/\grado},K) $ and $\mathring f^{p,q}_\alpha((x_{j,k}))$, respectively. 
\end{prop}

\begin{deff}\label{def:11}
	Take $\eps\in (0,1)$, $\delta,N>0$, $R\Meg 2$, $K,S\in \N$, $p\in [1,\infty]$ and a reduced $(\eps ,\delta, R)$-lattice $(x_{j,k})_{(j,k)\in J}$. We say that a family $(a_{j,k})_{(j,k)\in J}$ is a system of $(K,S,N,p,\Lc)$-atoms associated with $(x_{j,k})$ if there is a family $(b_{j,k})_{(j,k)\in J, j\Meg 1}$ such that following hold:
	\begin{enumerate}
		\item $a_{0,k'},b_{j,k}\in \Sr'(G)$   and $a_{j,k}=\Lc^S b_{j,k}$  for every $(0,k'),(j,k)\in J$ with $j\Meg 1$;
		
		\item $a_{0,k'}, \Lc^K a_{0,k'}, b_{j,k}, \Lc^{S+K}b_{j,k}\in L^p_\loc(G)$,
		\[
		\norm{(1+d(\,\cdot\,,x_{0,k}) )^{N} a_{0,k}}_{L^p(G)}, \norm{(1+d(\,\cdot\,,x_{0,k}) )^{N}\Lc^K a_{0,k}}_{L^p(G)}\meg 1,
		\]
		and
		\[\norm{(1+d(\,\cdot\,,x_{j,k})/\eps^j )^{N}  b_{j,k}}_{L^p(G)}, \norm{(1+d(\,\cdot\,,x_{j,k})/\eps^j )^{N}  (\eps^{j\grado}\Lc)^{K+S}b_{j,k}}_{L^p(G)}\meg \eps^{j(S\grado+Q_*/p)}
		\]
		for every $(0,k'),(j,k)\in J$ with $j\Meg 1$.
	\end{enumerate} 
	
	We say that a family $(a_{j,k})_{(j,k)\in J}$ is a system of strong $(K,S,\Lc)$-atoms associated with $(x_{j,k})$ if there is a family $(b_{j,k})_{(j,k)\in J, j\Meg 1}$ such that following hold:
	\begin{enumerate}
		\item   $a_{0,k'},b_{j,k}\in C^\infty(G)$  and $a_{j,k}=\Lc^S b_{j,k}$   for every $(0,k'),(j,k)\in J$ with $j\Meg 1$;
		
		\item $a_{0,k'}$ is supported in $B(x_{0,k'},2\delta R)$ and $b_{j,k}$ is supported in $B(x_{j,k}, 2\delta R \eps^{j })$  for every $(0,k'),(j,k)\in J$ with $j\Meg 1$;
		
		\item  $\abs{X a_{0,k'}}\meg \abs{X}$ for every $X\in U_{K\grado}$ and  $\abs{X b_{j,k}}\meg \abs{X}\eps^{j(S\grado-\deg(X) )}$ for every $X\in U_{(K+S)\grado}$, and for every $(0,k'),(j,k)\in J$ with $j\Meg 1$.
	\end{enumerate}  
\end{deff}

When $p=1$, one may replace $L^1(G)$ with the space $\Mc^1(G)$ of bounded measures on $G$ in the above estimates.

Cf.~\cite[Theorems 4.14 and 4.15]{Calzi4} for a proof of the following `atomic decomposition theorems.' 

\begin{teo}\label{teo:13bis}
	Take $p,q\in (0,\infty]$  and $\alpha\in \R$, and take $K,S\in \N$, $N>0$, and  $p_0\in [1,\infty]$ so that $K>(Q_*/p_0+\alpha)/\grado $, $S>   (Q_*(1/\min(1,p)-1/p_0) -\alpha) /\grado $, and $N>D/\min(1,p)$. Take $\eps\in (0,1) $, $\delta>0$ and $R\Meg 2$. Then, there is a constant $C>0$ such that, for every system of $(K,S,N,p_0,\Lc)$-atoms $(a_{j,k})$ associated with some  reduced $(\eps ,\delta,R)$-lattice $(x_{j,k})_{(j,k)\in J}$,
	\[
	\norm*{\sum_{j,k} \lambda_{j,k} a_{j,k}  }_{B^{p,q}_\alpha(G)}\meg C \norm{\lambda }_{b^{p,q}_{\alpha}(\eps,J )}
	\]
	for every $\lambda\in  b^{p,q}_{\alpha}(\eps,J )$, where the sum converges in $\Sr'(G)$  and also in $\mathring B^{p,q}_\alpha(G)$ if $\lambda\in  \mathring b^{p,q}_{\alpha}(\eps,J )$. In addition, for every $f\in B^{p,q}_\alpha(G)$ there are a system $(a_{f,j,k})_{j,k}$ of strong $(K,S,\Lc)$-atoms associated with $(x_{j,k})$ and $\lambda_f\in b^{p,q}_{\alpha}(\eps,J )$ such that 
	\[
	\norm{\lambda_f }_{b^{p,q}_{\alpha}(\eps,J )}\meg C\norm{f}_{B^{p,q}_\alpha(G)}
	\]
	and such that $f=\sum_{j,k} \lambda_{f,j,k} a_{f,j,k} $.
\end{teo}

\begin{teo}\label{teo:13}
	Take $p,q\in (0,\infty]$  and $\alpha\in \R$, and take $K,S\in \N$, $N>0$, and $p_0\in [1,\infty]$ so that $K>(Q_*/p_0+\alpha) /\grado$, $S>   (Q_*(1/\min(1,p,q)-1/p_0) -\alpha)/\grado $, and $N>D/\min(1,p,q)$. Take $\eps\in (0,1) $, $\delta>0$ and $R>2$. Then, there is a constant $C>0$ such that,  for every system of $(K,S,N,p_0,\Lc)$-atoms $(a_{j,k})$ associated with some reduced $(\eps ,\delta,R)$-lattice $(x_{j,k})_{(j,k)\in J}$,  
	\[
	\norm*{\sum_{j,k} \lambda_{j,k} a_{j,k}  }_{F^{p,q}_\alpha(G)}\meg C \norm{\lambda}_{f^{p,q}_\alpha((x_{j,k}))}
	\]
	for every $\lambda\in f^{p,q}_\alpha((x_{j,k}))$, where the sum converges in $\Sr'(G)$ and also in $\mathring F^{p,q}_\alpha(G)$ if $ \lambda\in f^{p,q}_\alpha((x_{j,k}))$. In addition,  for every $f\in F^{p,q}_\alpha(G)$ there are a system $(a_{f,j,k})_{j,k}$ of strong $(K,S,\Lc)$-atoms associated with $(x_{j,k})$ and $\lambda_f\in f^{p,q}_\alpha((x_{j,k}))$ such that 
	\[
	\norm{\lambda_f }_{f^{p,q}_\alpha((x_{j,k}))}\meg C\norm{f}_{F^{p,q}_\alpha(G)}
	\]
	and such that $f=\sum_{j,k} \lambda_{f,j,k} a_{f,j,k} $.
\end{teo}

\section{Interpolation}

We first recall some basic definitions on the notion of complex interpolation of quasi-Banach pairs we shall use, cf.~\cite{KaltonMitrea,KMM}.

\begin{deff}
	A quasi-Banach space $X$ is said to be analytically convex (or $A$-convex) if there is a constant $C>0$ such that $\norm{P(0)}_X\meg C \max_{\abs{z}=1} \norm{P(z)}_X$ for every (complex) polynomial mapping $P\colon \C\to X$.
\end{deff}

Cf.~\cite[Theorem 7.4]{KMM} for a list of equivalent formulation of analytic convexity.

\begin{deff}
	Let $U$ be an open subset of $\C$, and let $X$ be a quasi-Banach space. A function $f\colon U\to X$ is holomorphic if it can be written as the sum of a uniformly convergent power series in the neighbourhood of each $x\in U$.
\end{deff}

Cf., e.g.,~\cite{Kalton} for more information on the theory of holormphic functions with values in a quasi-Banach space.

\begin{deff}
	Let $(X_0,X_1)$ be a pair of quasi-Banach spaces, with $X_0+X_1$ analytically convex, and take $\theta\in (0,1)$. Define $S\coloneqq \Set{z\in \C\colon 0\meg \Re z \meg 1}$ and let $U$ be the interior of $S$. Define $\Fc(X_0,X_1)$ as the space of bounded continuous functions $f\colon S\to X_0+X_1$ which are holomorphic on $U$ and induce bounded continuous functions $\R\ni t\mapsto f(j+i t)\in X_j$ ($j=0,1$), endowed with the quasi-norm $\max_{j=0,1}\sup_{t\in \R}\norm{f(j+it)}_{X_j}$.
	Then, $(X_0,X_1)_{[\theta]}$ is the image of $\Fc(X_0,X_1)$ under the mapping $f\mapsto f(\theta)$, endowed with the corresponding quasi-norm.
\end{deff}

\begin{lem}\label{lem:74}
	Take $p,q\in (0,\infty]$, $\alpha\in \R$, $\eps\in (0,1)$, and a reduced $(\eps,\delta,R)$-lattice $(x_{j,k})_{(j,k)\in J}$ on $G$ for some $\delta>0$ and for some $R\Meg 2$. Then, $b^{p,q}_\alpha(\eps)$, $f^{p,q}_\alpha((x_{j,k}))$, $B^{p,q}_\alpha(G)$,  and $F^{p,q}_\alpha(G)$ are analytically convex. 
\end{lem}
 
\begin{proof}
	By~\cite[Lemma 7.6]{KMM}, $\ell^q(\N)$ is analytically convex. By the same reference, also $\ell^{p,q}(\N,\N)\cong \ell^q(\N; \ell^p(\N))$ and $L^p(G;\ell^q(\N)) $ are analytically convex. It then follows that $b^{p,q}_\alpha(\eps)$ is analytically convex, so that the same holds for $B^{p,q}_\alpha(G)$, thanks to Proposition~\ref{prop:14} and~\cite[Proposition 7.5]{KMM}. In addition, $L^{q,p}_0(\N,G)\cong L^p_0(G;\ell^q_0(\N))$, so that $\mathring f^{p,q}_\alpha((x_{j,k}))$ is analytically convex by~\cite[Proposition 7.5]{KMM} when $p<\infty$.
	Consequently, there is a constant $C>0$ such that $\norm{P(0)}\meg C \max_{\abs{z}=1}\norm{P(z)}$ for every (complex) polynomial $P$ with values in $\mathring f^{p,q}_\alpha((x_{j,k}))$. 
	Then, for every complex polynomial $P$ with values in $  f^{p,q}_\alpha((x_{j,k}))$, $\norm{P(0)}=\sup_{k\in \N} \norm{\chi_k P(0)}\meg C \sup_{k\in \N}\max_{\abs{z}=1} \norm{\chi_k P(z)}\meg \max_{\abs{z}=1}\norm{P(z)}$, where $\chi_k$ is the characteristic function of $J\cap\Set{0,\dots,k}^2$ for every $k\in\N$. 
	It then follows that also $ f^{p,q}_\alpha((x_{j,k}))$ is analytically convex, so that also  $F^{p,q}_\alpha(G)$ is analytically convex by Proposition~\ref{prop:14} and~\cite[Proposition 7.5]{KMM}.   
	Finally, assume that $p=\infty$. If $q=\infty$, then $f^{\infty,\infty}_\alpha((x_{j,k})) $ is canonically isomorphic to $b^{\infty,\infty}_\alpha(\eps)$, so that the assertion follows as before (alternatively, one may simply observe that every Banach space is analytically convex). If, otherwise, $q<\infty$, then it suffices to observe that $\Cc^q(\eps)$ is canonically isomorphic to a closed subspace of  $\ell^\infty(\N\times G;L^q(\N\times G))$, which is analytically convex by~\cite[Lemma 7.6]{KMM}, applied twice as before. One may then conclude as before.
\end{proof}

\begin{deff}
	Let $J$ be a set.
	Take $\theta\in (0,1)$ and let $X_0,X_1$ be two quasi-Banach subspaces of $\C^{J}$ such that, if $a\in X_j$ and $b\in \C^{J}$ are such that $\abs{b}\meg \abs{a}$, then $b\in X_j$ ($j=0,1$). Then, the Calder\'on product $X_0^{1-\theta}X_1^\theta$ is the space of $a\in \C^{J}$ such that  $\abs{a}\meg \abs{a^{(0)}}^{1-\theta} \abs{a^{(1)}}^{\theta}$ for some   $a^{(j)}\in X_j$ ($j=0,1$), endowed with the quasi-norm defined so that $\norm{a}_{X_0^{1-\theta}X_1^\theta}$ is the greatest lower bound of the products $\norm{a^{(0)}}^{1-\theta}_{X_0} \norm{a^{(1)}}^{\theta}_{X_1}$ for every $ a^{(0)}, a^{(1)}$ as before.
\end{deff}

\begin{lem}\label{lem:75}
	Take $p_0,p_1,q_0,q_1\in (0,\infty]$, $\alpha_0,\alpha_1\in \R$, $\theta\in (0,1)$, $\eps\in (0,1)$, and a reduced $(\eps,\delta,R)$-lattice $(x_{j,k})_{(j,k)\in J}$ for some $\delta>0$ and for every $R\Meg 2$. Take 
	\[
	X_0,X_1\in \Set{b^{p_0,q_0}_{\alpha_0}(\eps,J),b^{p_1,q_1}_{\alpha_1}(\eps,J),f^{p_0,q_0}_{\alpha_0}((x_{j,k})),f^{p_1,q_1}_{\alpha_1}((x_{j,k}))},
	\]
	and let $Y_0,Y_1$ be the closures of $\C^{(J)}$ in $X_0,X_1$, respectively.
	Then, the following hold:
	\begin{itemize}
		\item $X_0+X_1$ and $Y_0+Y_1$ are analytically convex;
		\item  $(Y_0,Y_1)_{[\theta]}= Y_0^{1-\theta} Y_1^{\theta}$;
		
		\item $(X_0,X_1)_{[\theta]}$ is the closure of $X_0\cap X_1$ in $ X_0^{1-\theta} X_1^{\theta}$.
	\end{itemize} 
\end{lem}

By Proposition~\ref{prop:14} and~\cite[Proposition 7.5]{KMM}, one may then deduce that also the corresponding sums of Besov and Triebel--Lizorkin spaces are analytically convex, so that one may consider the corresponding complex interpolation spaces in the sense described above.

\begin{proof}
	Using Lemma~\ref{lem:74} and~\cite[Theorem 7.9]{KMM}, we see that $Y_0+Y_1$ is analytically convex, and that $(Y_0,Y_1)_{[\theta]}= Y_0^{1-\theta} Y_1^{\theta}$. Arguing as in the proof of Lemma~\ref{lem:74}, one may then show that also $X_0+X_1$ is analytically convex.  Arguing as in the proof of Lemma~\ref{lem:74},\footnote{In fact, it is readily verified that $\norm{a}_{(X_0,X_1)_{[\theta]}}=\sup_k \norm{\chi_k a}_{(X_0,X_1)_{[\theta]}}=\sup_k \norm{\chi_k a}_{(Y_0,Y_1)_{[\theta]}}$ for every $a\in \C^{J}$, where $\chi_k$ is the characteristic function of $J\cap \Set{0,\dots, k}^2$.} one may then show that $(X_0,X_1)_{[\theta]}$ is a closed vector subspace of $ X_0^{1-\theta} X_1^{\theta}$ (with the induced topology), so that the conclusion follows from the fact that $X_0\cap X_1$ is dense in $(X_0,X_1)_{[\theta]}$. 
\end{proof}

\begin{teo}\label{teo:9b}
	Take $\theta\in (0,1)$, $\alpha,\alpha_0,\alpha_1\in \R$, $p,p_0,p_1,q,q_0,q_1\in (0,\infty]$, and define $\alpha_\theta,p_\theta,q_\theta$ so that $\alpha_\theta=(1-\theta)\alpha_0+\theta \alpha_1$, $\frac{1}{p_\theta}=\frac{1-\theta}{p_0}+\frac{\theta}{p_1}$, and $\frac{1}{q_\theta}=\frac{1-\theta}{q_0}+\frac{\theta}{q_1}$. Then, the following equalities hold (with equivalence of quasi-norms):
	\begin{enumerate}
		\item[\textnormal{(1)}] $(B_{\alpha_0}^{p, q_0}(G),B_{\alpha_1}^{p, q_1}(G))_{\theta,q}= B_{\alpha_\theta}^{p, q}(G)$ ($\alpha_0\neq \alpha_1$);
		
		\item[\textnormal{(2)}] $(B_{\alpha}^{p, q_0}(G),B_{\alpha}^{p, q_1}(G))_{\theta,q_\theta}= B_{\alpha}^{p, q_\theta}$;
		
		\item[\textnormal{(3)}] $(B_{\alpha_0}^{p_0, q_0}(G),B_{\alpha_1}^{p_1, q_1}(G))_{\theta,q_\theta}= B_{\alpha_\theta}^{p_\theta, q_\theta}(G)$ ($q_0,q_1<\infty$, $p_\theta=q_\theta$);
		
		\item[\textnormal{(4)}] $(B_{\alpha_0}^{p_0, q_0}(G),B_{\alpha_1}^{p_1, q_1}(G))_{[\theta]}$ is the closure of $B_{\alpha_0}^{p_0, q_0}(G)\cap B_{\alpha_1}^{p_1, q_1}(G)$ in $ B_{\alpha_\theta}^{p_\theta, q_\theta}(G)$ (which is $ B_{\alpha_\theta}^{p_\theta, q_\theta}(G)$ if either $q_\theta\neq \infty$ or $\alpha_0=\alpha_1$);   
		
		\item[\textnormal{(5)}] $(F_{\alpha_0}^{p, q_0}(G),F_{\alpha_1}^{p, q_1}(G))_{\theta,q}= B_{\alpha_\theta}^{p, q}(G)$ ($\alpha_0\neq \alpha_1$);
		
		\item[\textnormal{(6)}] $(F_{\alpha}^{p_0, q}(G),F_{\alpha}^{p_1, q}(G))_{\theta,p_\theta}= F_{\alpha}^{p_\theta,q}(G)$; 
		
		\item[\textnormal{(7)}] $(F_{\alpha_0}^{p_0, q_0}(G),F_{\alpha_1}^{p_1, q_1}(G))_{[\theta]}$ is the closure of $F_{\alpha_0}^{p_0, q_0}(G)\cap F_{\alpha_1}^{p_1, q_1}(G)$ in $ F_{\alpha_\theta}^{p_\theta, q_\theta}(G)$ (which is $ F_{\alpha_\theta}^{p_\theta, q_\theta}(G)$ if $p_\theta<\infty$ and either $q_\theta\neq \infty$ or $\alpha_0=\alpha_1$).
	\end{enumerate}
	In addition, the following equalities hold (with equivalence of quasi-norms):
	\begin{enumerate}
		\item[\textnormal{(1$'$)}] $(\mathring B_{\alpha_0}^{p, q_0}(G),\mathring B_{\alpha_1}^{p, q_1}(G))_{\theta,q,0}= \mathring B_{\alpha_\theta}^{p, q}(G)$ ($\alpha_0\neq \alpha_1$);
		
		\item[\textnormal{(2$'$)}] $(\mathring B_{\alpha}^{p, q_0}(G),\mathring B_{\alpha}^{p, q_1}(G))_{\theta,q_\theta}= \mathring B_{\alpha}^{p, q_\theta}$;
		
		\item[\textnormal{(3$'$)}] $(\mathring B_{\alpha_0}^{p_0, q_0}(G),\mathring B_{\alpha_1}^{p_1, q_1}(G))_{\theta,q_\theta,0}= \mathring B_{\alpha_\theta}^{p_\theta, q_\theta}(G)$ ($q_0,q_1<\infty$, $p_\theta=q_\theta$);
		
		\item[\textnormal{(4$'$)}] $(\mathring B_{\alpha_0}^{p_0, q_0}(G),\mathring B_{\alpha_1}^{p_1, q_1}(G))_{[\theta]}= \mathring B_{\alpha_\theta}^{p_\theta, q_\theta}(G)$;  
		
		\item[\textnormal{(5$'$)}] $(\mathring F_{\alpha_0}^{p, q_0}(G),\mathring F_{\alpha_1}^{p, q_1}(G))_{\theta,q,0}= \mathring B_{\alpha_\theta}^{p, q}(G)$ ($\alpha_0\neq \alpha_1$);
		
		\item[\textnormal{(6$'$)}] $(\mathring F_{\alpha}^{p_0, q}(G),\mathring F_{\alpha}^{p_1, q}(G))_{\theta,p_\theta}= \mathring F_{\alpha}^{p_\theta,q}(G)$;

		\item[\textnormal{(7$'$)}] $(\mathring F_{\alpha_0}^{p_0, q_0}(G),\mathring F_{\alpha_1}^{p_1, q_1}(G))_{[\theta]}= \mathring F_{\alpha_\theta}^{p_\theta, q_\theta}(G)$.
	\end{enumerate}
\end{teo}

One may wonder whether $(F_{\alpha_0}^{p_0, q_0}(G),F_{\alpha_1}^{p_1, q_1}(G))_{\theta,q_\theta}= F_{\alpha_\theta}^{p_\theta,q_\theta}(G)$ when $p_\theta=q_\theta$ and $p_0,p_1<\infty$. If $p_0,p_1>1$, this follows from~\cite[Theorem 10.1]{BCP}. In the general case, one has $(L^{p_0}(G;\ell^{q_0}(\N)),L^{p_1}(G;\ell^{q_1}(\N)))_{\theta,q_\theta}=L^{p_\theta}(G;\ell^{q_\theta}(\N))$, whence a continuous inclusion $(F_{\alpha_0}^{p_0, q_0}(G),F_{\alpha_1}^{p_1, q_1}(G))_{\theta,q_\theta}\subseteq F_{\alpha_\theta}^{p_\theta,q_\theta}(G)$. Nonetheless, we were not able to obtain a proof of the reverse inclusion in this generality.

This result is a consequence of the following abstract one, which allows us to refer to the interpolation properties of the corresponding spaces of sequences even though we do not know if our Besov and Triebel--Lizorkin spaces are retracts of the corresponding spaces of sequences in full generality.

\begin{prop}\label{prop:16}
	Take  $\alpha,\alpha_0,\alpha_1\in \R$, $p, p_0,p_1, q_0,q_1\in (0,\infty]$, $\eps\in (0,1)$, $\delta>0$, $R\Meg 2$, and a reduced $(\eps, \delta,R)$-lattice $(x_{j,k})_{(j,k)\in J}$. Let $\Fc$ be an interpolation functor on a category of pairs of  quasi-Banach spaces including Besov and Triebel--Lizorkin spaces and the corresponding spaces of sequences.
	
	If 
	\[
	\Fc(b^{p_0,q_0}_{\alpha_0}(\eps),b^{p_1,q_1}_{\alpha_1}(\eps))=b^{p ,q }_{\alpha }(\eps)\qquad \text{(resp.\ $\Fc(\mathring b^{p_0,q_0}_{\alpha_0}(\eps),\mathring b^{p_1,q_1}_{\alpha_1}(\eps))=\mathring b^{p ,q }_{\alpha }(\eps)$),}
	\] 
	then 
	\[
	\Fc(B^{p_0,q_0}_{\alpha_0}(G),B^{p_1,q_1}_{\alpha_1}(G))=B^{p ,q }_{\alpha }(G)\qquad \text{(resp.\ $\Fc(\mathring B^{p_0,q_0}_{\alpha_0}(G),\mathring B^{p_1,q_1}_{\alpha_1}(G))=\mathring B^{p ,q }_{\alpha }(G)$).}
	\] 
	
	Analogously, if   
	\[
	\Fc(f^{p_0,q_0}_{\alpha_0}((x_{j,k})),f^{p_1,q_1}_{\alpha_1}((x_{j,k})))=f^{p ,q }_{\alpha }((x_{j,k}))\qquad \text{(resp.\ $\Fc(\mathring f^{p_0,q_0}_{\alpha_0}((x_{j,k})),\mathring f^{p_1,q_1}_{\alpha_1}((x_{j,k})))=\mathring f^{p ,q }_{\alpha }((x_{j,k}))$),}
	\]
	then 
	\[
	\Fc(F^{p_0,q_0}_{\alpha_0}(G),F^{p_1,q_1}_{\alpha_1}(G))=F^{p ,q }_{\alpha }(G)\qquad \text{(resp.\  $\Fc(\mathring F^{p_0,q_0}_{\alpha_0}(G),\mathring F^{p_1,q_1}_{\alpha_1}(G))=\mathring F^{p ,q }_{\alpha }(G)$).}
	\]	
\end{prop}

Notice that we consider, in the statement, the spaces $b^{p,q}_\alpha(\eps)$ instead of the more `natural' $b^{p,q}_\alpha(\eps,J)$ since it is easier to find interpolation results for the former spaces rather than the latter (and the results for the former spaces imply the corresponding assertions for the latter). This result may be appropriately modified to cover also (4) and (7) of Theorem~\ref{teo:9b}.\footnote{In this case, one has to observe that, if $f$ belongs to the intersection of two Besov (or Triebel--Lizorkin) spaces, then it may be written as a sum of an appropriate system of atoms with coefficients in the intersection of the corresponding spaces of sequences, as one observes in the proof of~\cite[Theorems 4.14 and 4.15]{Calzi4}. More precisely, the system of atoms and the corresponding coefficients may be chosen in an `universal way,' that is, so that the coefficients belong to the spaces of sequences corresponding to any Besov or Triebel--Lizorkin space to which $f$ belongs, for every fixed $f$, provided that the conditions on $K,S,N$ are met.}

\begin{proof}
	We prove only the first assertion; the other ones may be proved similarly. Observe first that, since $b^{p_2,q_2}_{\alpha_2}(\eps,J)$ and $\mathring b^{p_2,q_2}_{\alpha_2}(\eps,J)$ are retracts of $b^{p_2,q_2}_{\alpha_2}(\eps)$ and $\mathring  b^{p_2,q_2}_{\alpha_2}(\eps)$, respectively, for every $p_2,q_2\in (0,+\infty]$ and for every $\alpha_2\in \R$, with the same retractions and sections, we also have
	\[
	\Fc(b^{p_0,q_0}_{\alpha_0}(\eps,J),b^{p_1,q_1}_{\alpha_1}(\eps,J))=b^{p ,q }_{\alpha }(\eps,J)\qquad \text{(resp.\ $\Fc(\mathring b^{p_0,q_0}_{\alpha_0}(\eps,J),\mathring b^{p_1,q_1}_{\alpha_1}(\eps,J))=\mathring b^{p ,q }_{\alpha }(\eps,J)$).}
	\]
	Set $X_0=B^{p_0,q_0}_{\alpha_0}(G)$, $X_1=B^{p_1,q_1}_{\alpha_1}(G)$, $X=B^{p ,q }_{\alpha }(G)$, $Y_0=b^{p_0,q_0}_{\alpha_0}(\eps,J)$, $Y_1=b^{p_1,q_1}_{\alpha_1}(\eps,J)$, $Y=b^{p ,q }_{\alpha }(\eps,J) $. 
	Fix $(\psi_j)\in \Psi_{\eps^{\grado}}$ and let $U$ be the set of $(z_{j,k})\in G^J$ such that $z_{j,k}\in \overline B(e, \delta R \eps^j)$ for every $(j,k)\in J$. In addition, fix $K,S,N\in\N$ so that $K,S>Q_*/\min(1,p_0,p_1,p,q_0,q_1,q)+\max(\abs{\alpha_0},\abs{\alpha_1},\abs{\alpha})$ and $N>D/\min(1,p_0,p_1,p,q_0,q_1,q)$, and	 let $V$ be the set of systems of $(K,S,N,\infty,\Lc)$-atoms.
	For every $z\in U$ and for every $a\in V$, consider the mappings $\Phi_z\colon \Sr'(G)\ni u \mapsto ((u*\psi_j)(x_{j,k}z_{j,k}))\in \C^{J}$, and $\Psi_a\colon \lambda \mapsto \sum_{j,k} \lambda_{j,k} a_{j,k}$. Then, Proposition~\ref{prop:14} shows that the $\Phi_z$, $z\in U$, induce equicontinuous linear mappings $X_0\to Y_0$, $X_1\to Y_1$, and $X\to Y$, while Theorem~\ref{teo:13bis} shows that  the $\Psi_a$, $a\in V$, induce  equicontinuous linear mappings $Y_0\to X_0$, $Y_1\to X_1$, and $Y\to X$. 
	By assumption, $\Fc(Y_0,Y_1)=Y$. Then, the $\Phi_z$, $z\in U$, induce equicontinuous linear mappings $\Fc(X_0,X_1)\to \Fc(Y_0,Y_1)=Y$, while the $\Psi_a$, $a\in V$, induce equicontinuous linear mappings $Y=\Fc(Y_0,Y_1)\to \Fc(X_0,X_1)$. Since for every $f\in X$ there are $\lambda_f\in Y$ with comparable quasi-norm and  $a_f\in V$  such that $f=\Psi_{a_f}(\lambda_f)$, thanks to Theorem~\ref{teo:13bis},  this proves that $X\subseteq \Fc(X_0,X_1)$ continuously. On the other hand, if $f\in  \Fc(X_0,X_1)$, then the quasi-norms of $\Phi_z f$, $z\in U$, in $Y=\Fc(Y_0,Y_1)$ are uniformly controlled by the quasi-norm of $f$, so that $f\in X$  by Proposition~\ref{prop:14}. More precisely, $\Fc(X_0,X_1)\subseteq X$ continuously.  The proof is therefore complete.
\end{proof}

\begin{proof}[Proof of Theorem~\ref{teo:9b}.]
	\textsc{Step I} Observe first that, by  Proposition~\ref{prop:16}, we may reduce to proving the various assertions for the corresponding spaces of sequences. For simplicity, we shall simply write $b^{p,q}_\alpha$ and $f^{p,q}_\alpha$ instead of $b^{p,q}_\alpha(\eps)$ and $f^{p,q}_\alpha((x_{j,k}))$, where $(x_{j,k})_{(j,k)\in J}$ is a (fixed) reduced $(\eps,1,2)$-lattice on $G$.
	Then, observe that (1$'$)--(3$'$) and (5$'$)--(6$'$) follow from (1)--(3) and (5)--(6) and from the fact that $K(t,a,\mathring b^{p_0,q_0}_{\alpha_0}, \mathring b^{p_1,q_1}_{\alpha_1})= K(t,a, b^{p_0,q_0}_{\alpha_0},   b^{p_1,q_1}_{\alpha_1} )$ and $K(t,a,\mathring f^{p_0,q_0}_{\alpha_0} , \mathring f^{p_1,q_1}_{\alpha_1} )= K(t,a, f^{p_0,q_0}_{\alpha_0} ,   f^{p_1,q_1}_{\alpha_1} )$ for every $a\in \C^{(\N\times \N)}$ and for every $t>0$, where 
	\[
	K(t,a,A_0,A_1)=\inf \Set{\norm{a_0}_{A_0}+t\norm{a_1}_{A_1}\colon a=a_0+a_1, a_0\in A_0, a_1\in A_1}
	\]
	for every  pair of normed abelian groups $(A_0,A_1)$, for every $t>0$, and for every $a$ in the sum $\Sigma(A_0,A_1)$. Indeed, it suffices to observe that, if $a\in \C^{(J)}$ may be written as $a^{(0)}+a^{(1)}$ and if $K$ is the (finite) support of $a$, then $a=\chi_K a^{(0)}+\chi_K a^{(1)}$ and $\chi_K a^{(0)}$ and $\chi_K a^{(1)}$ have smaller quasi-norms than $a^{(0)}$ and $a^{(1)}$, respectively, in the considered spaces. We recall, incidentally, that $\norm{a}_{(A_0,A_1)_{\theta,q}}=\norm{t^{-\theta} K(t,a,A_0,A_1)}_{L^q (\mi)}$, where $\mi$ denotes the Haar measure on $(0,+\infty)$ such that $\mi([1,e])=1$. 
	
	Assertions (1)--(3)   follow from Proposition~\ref{prop:16} and from the known interpolation properties of the spaces $b^{p,q}_\alpha $. Cf.~\cite[Theorem 5.6.1]{BerghLofstrom} for (1), (2), and~\cite[Theorems 5.6.2 and 5.2.1]{BerghLofstrom} for (3). 
	Assertion  (5)   follows from assertion  (1)   and~\cite[Proposition 4.9]{Calzi4}.
	
	\textsc{Step II} We now prove assertion  (6). Observe that we may assume that $p_\theta\neq \infty$, for otherwise there is nothing to prove. In addition, we may reduce to the case $\alpha=0$. We proceed as in~\cite{FrazierJawerth} and  show that $(f^{0} , f^{\infty,q}_0 )_{p/(1+p),1+p}^{(1+p)/p}=f^{p,q}_\alpha $ for every $p\in (0,\infty)$, where $f^0 $ is the space of $a\in \C^{J}$ such that 
	\[
	\beta\bigg(\bigcup_{j,k\colon a_{j,k}\neq 0} B(x_{j,k}, \eps^j)\bigg)<\infty,
	\]
	considered as a normed abelian group, and where $X^{c}$ denotes the normed abelian group $X$ endowed with the `quasi-norm' $\norm{\,\cdot\,}^{c}_X$, for $c>0$.
	Using the power theorem (cf.~\cite[Theorem 3.11.6]{BerghLofstrom}) one may then deduce
	\[
	(f^{p_0,q}_0, f^{p_1,q}_0)_{\theta,q}=( (f^{0} , f^{\infty,q}_0 )_{p_0/(1+p_0),1+p_0},(f^{0} , f^{\infty,q}_0 )_{p_1/(1+p_1),1+p_1}  )_{\eta, 1+p_\theta}^{(1+p_\theta)/p_\theta},
	\]
	where $\eta=\theta\frac{p_\theta(1+p_1)}{p_1(1+p_\theta)}$. By means of the reiteration theorem (cf.~\cite[Theorem 3.11.5]{BerghLofstrom}) one may then deduce
	\[
	(f^{p_0,q}_0, f^{p_1,q}_0)_{\theta,q}=(f^0,f^{\infty,q}_0)_{p_\theta/(1+p_\theta),1+p_\theta}^{(1+p_\theta)/p_\theta}=f^{p_\theta,q}_0,
	\]
	whence the conclusion. Now, set $\mi=\sum_{j\in\N} \delta_{\eps^j}$, define $L^0$ as the normed abelian group of the $\beta$-measurable functions which are concentrated in a set with finite measure, endowed with the `quasi-norm' 
	\[
	\norm{f}_{L^0}=\beta(\Set{x\in G\colon f(x)\neq 0}) 
	\]
	and simply write $L^\infty$ instead of $L^\infty(G)$. In addition, set $\iota(a)( j,x)=\sum_{k} a_{j,k}\chi_{B(x_{j,k}, \eps^j)}$ for every $j\in\N$ and for every $x\in G$, and observe that, by Lemmas~\ref{lem:59b} and~\ref{lem:60b}, $\norm{a}_{f^{p,q}_0}$ is equivalent to $\norm{m^q_{1/4,\eps}(\iota(a))}_{L^p(G)}$ for every $p\in (0,\infty)$. 
	Since $(L^0,L^\infty)_{p/(1+p),1+p}^{(1+p)/p}=L^p(G)$ (as one sees arguing as in the proof of~\cite[Lemma 6.1]{FrazierJawerth}), in order to prove that $(f^{0} , f^{\infty,q}_0)_{p/(1+p),1+p}^{(1+p)/p}=f^{p,q}_0 $, it will then suffice   to show that there is a constant $C_1>0$ such that
	\[
	\frac{1}{C_1} K(t,m^q_{1/4,\eps}(\iota(a)),L^0,L^\infty)\meg K(t,a,f^0, f^{\infty,q}_\alpha)\meg C_1  K(t,m^q_{1/4,\eps}(\iota(a)),L^0,L^\infty).
	\]
	Then, take $a\in f^0+ f^{\infty,q}_0$ and $t>0$. Take $a^{(0)}\in f^0$ and $a^{(1)}\in f^{\infty,q}_0$ so that $a=a^{(0)}+a^{(1)}$, and let us prove that
	\[
	m^q_{1/4,\eps}(\iota(a))\meg 2^{(1/q-1)_+}(m^q_{1/8,\eps}(\iota(a^{(0)}))+m^q_{1/8,\eps}(\iota(a^{(1)}))).
	\]
	To this aim, take $j\in\N$, $x\in G$, and $\lambda_k>2^{(1/q-1)_+}m^{q}_{1/8,\eps,j}(\iota(a^{(k)}))(x)$ for $k=0,1$, so that
	\[
	\beta\Big( \Set{y\in B(x,\eps^j) \colon \norm{\iota(a^{(k)})(h,y)}_{\ell^q_h(j+\N)}>2^{-(1/q-1)_+}\lambda_k  } \Big)\meg \frac 1 8 \beta(B(e,\eps^j))
	\]
	for $k=0,1$. Since $\norm{\iota(a )(h,y)}_{\ell^q_h(j+\N)}\meg 2^{(1/q-1)_+}(\norm{\iota(a^{(0)})(h,y)}_{\ell^q_h(j+\N)}+\norm{\iota(a^{(1)})(h,y)}_{\ell^q_h(j+\N)})$ for every $y\in G$, this implies that
	\[
	\beta\Big( \Set{y\in B(x,\eps^j) \colon \norm{\iota(a )(h,y)}_{\ell^q_h(j+\N)}>\lambda_0+\lambda_1  } \Big)\meg \frac 1 4  \beta(B(e,\eps^j)),
	\]
	so that
	\[
	m^q_{1/4,\eps,j}(\iota(a))\meg \lambda_0+\lambda_1.
	\]
	The assertion follows from the arbitrariness of $\lambda_0$ and $\lambda_1$, and the by the arbitrariness of $j$.

	Set $f_0\coloneqq \inf(m^q_{1/4,\eps}(\iota(a)), 2^{(1/q-1)_+}m^q_{1/8,\eps}(\iota(a^{(0)})))$ and $f_1\coloneqq (m^q_{1/4,\eps}(\iota(a))-2^{(1/q-1)_+}m^q_{1/8,\eps}(\iota(a^{(0)})))_+$, and observe that $f_0+f_1= m^q_{1/4,\eps}(\iota(a))$, 
	\[
	\norm{f_0}_{L^0}\meg  \norm{m^q_{1/8,\eps}(\iota(a^{(0)}))}_{L^0},
	\]
	and 
	\[
	\norm{f_1}_{L^\infty}\meg 2^{(1/q-1)_+}\norm{m^q_{1/8,\eps}(\iota(a^{(1)}))}_{L^\infty},
	\]
	so that there is a constant $C_2>0$ such that
	\[
	\begin{split}
	K(t,m^q_{1/4,\eps}(\iota(a)),L^0,L^\infty)&\meg \norm{f_0}_{L^0}+t\norm{f_1}_{L^\infty}\\
		&\meg  C_2(\norm{a^{(0)}}_{f^0}+ t\norm{a^{(1)}}_{f^{\infty,q}_0}),
	\end{split} 
	\]
	thanks to Lemma~\ref{lem:59b} (extended to $f_0$ with a similar proof). Consequently, $K(t,m^q_{1/4,\eps}(\iota(a)),L^0,L^\infty)\meg  C_2K(t,a,f^0, f^{\infty,q}_0)$. 
	Conversely, take $f_0\in L^0$ and $f_1\in L^\infty$ so that $m^{q}_{1/4,\eps}(\iota(a))=f_0+f_1$. For every $j,k\in \N$, define $B_{j,k}^+\coloneqq \Set{x\in B(x_{j,k}, \eps^j)\colon f_0(x)\neq 0}$ and $B_{j,k}^-\coloneqq B(x_{j,k}, \eps^j)\setminus B_{j,k}^+$, and set $A\coloneqq\Set{ (j,k)\in J\colon \beta(B^+_{j,k})> \frac 1 2 \beta(B(e, \eps^j))}$. 
	Define $a^{(0)}\coloneqq \chi_A a$ and $a^{(1)}\coloneqq a-a^{(0)}$. By Lemma~\ref{lem:59b}, we may find a sequence $(E_j)$ of $\beta$-measurable subsets of $G$ such that $\beta(  B(x,\eps^j) \cap E_j)>\frac 3 4 \beta(B(e,\eps^j))$ for every $x\in G$ and for every $j\in \N$, and such that,  
	\[
	\norm{\chi_{E_j }(x) \iota(a )( j,x)}_{\ell^q_j(\N)}\meg m^{q}_{1/4,\eps}(\iota(a))(x)
	\]
	for every $x\in G$.

	Set, for every $j\in\N$, 
	\[
	E^{(0)}_j\coloneqq \Set{x\in E_j\colon \exists k\in \N \quad (j,k)\in A\land x\in B_{j,k}^+}
	\]
	and 
	\[
	E^{(1)}_j\coloneqq \Set{x\in E_j\colon \exists k\in \N \quad (j,k)\notin A\land x\in B_{j,k}^-},
	\]
	so that  $\beta(E^{(0)}_j\cap B(x_{j,k}, \eps^j))\Meg \frac 1 4 \beta(B(e, \eps^j))$ for every $(j,k)\in A$, while $\beta(E^{(1)}_j\cap B(x_{j,k}, \eps^j))\Meg \frac 1 4 \beta(B(e, \eps^j))$ for every $(j,k)\not\in A$. In addition, clearly
	\[
	\norm{a^{(0)}}_{f^0}\meg 4 \norm*{\sup_{j\in \N} \chi_{E^{(0)}_j}(x) \iota(a^{(0)})( j,x)}_{L^0_x(G)},
	\]
	while, arguing as in the proof  of Lemma~\ref{lem:60b}, one may show that there is a constant $C_3>0$ such that
	\[
	\norm{a^{(1)}}_{f^{\infty,q}_0}\meg C_3 \norm*{\chi_{E^{(1)}_j}(x) \iota(a^{(1)})(j,x)}_{L^{q,\infty}_{j,x}(\N,G)}.
	\]
	Now, observe that, if $\sup_{j\in \N}\chi_{E^{(0)}_j}(x) \iota(a^{(0)})(j,x)\neq 0$, then there is $(j,k)\in A$ such that $x\in B^+_{j,k}$, so that $f_0(x)\neq 0$. Therefore, $\norm{a^{(0)}}_{f^0}\meg 4 \norm{f_0}_{L^0}$. In addition, if $\chi_{E^{(1)}_j}(x) \iota(a^{(1)})( j,x)\neq 0$, then there is $k $ such that $x\in B^-_{j,k}$, so that $f_1(x)= m^q_{1/4,\eps}(\iota(a))(x)$. 
	It then follows that 
	\[
	\norm{\chi_{E^{(1)}_j}(x) \iota(a^{(1)})( j,x)}_{\ell^q_j(\N)}\meg\norm{\chi_{E_j}(x) \iota(a )( j,x)}_{\ell^q_j(\N)}\meg  m^q_{1/4,1}(\iota(a))(x)=  f_1(x)
	\]
	for every $x\in G$ such that the left hand side is non-zero. Consequently, 
	\[
	\norm{a^{(1)}}_{f^{\infty,q}_0}\meg C_3 \norm*{\chi_{E^{(1)}_j}(x) \iota(a^{(1)})(j,x)}_{L^{q,\infty}_{j,x}(\N,G)}\meg C_3 \norm{f_1}_{L^\infty}.
	\]
	Consequently, $K(t,a,f^0,f^{\infty,q}_0)\meg \max(4,C_3) K(t,m^{q}_{1/4,\eps}(\iota(a)),L^0,L^\infty)$, whence the conclusion.

	\textsc{Step III} We now prove assertion (7), keeping the notation of~\textsc{step II}. By Lemma~\ref{lem:75} (and some elementary remarks), it will suffice to show that $(f^{p_0,q_0}_{\alpha_0})^{1-\theta}(f^{p_1,q_1}_{\alpha_1})^\theta=f^{p_\theta,q_\theta}_{\alpha_\theta}$. Observe that we may again reduce to the case $\alpha_0=\alpha_1=0$, so that $\alpha_\theta=0$ as well.
	Take $a\in (f^{p_0,q_0}_{0})^{1-\theta}(f^{p_1,q_1}_{0})^\theta$, and take $a^{(j)}\in f^{p_j,q_j}_{0}$, $j=0,1$, so that $\abs{a}\meg \abs{a^{(0)}}^{1-\theta} \abs{a^{(1)}}^\theta$.  Take $\eta\in (1/2,1)$ and observe that, by Lemma~\ref{lem:59b}, for every $\ell=0,1$ there are a sequence $(E_j^{(\ell)})$ of $\beta$-measurable subsets of $G$   and a constant $C_5>0$ such that $\norm{\chi_{E^{(\ell)}_j}(x) \iota(a^{(\ell)})(j,x)}_{L^{q_\ell,p_\ell}_{j,x}(\N,G)}\meg C_5 \norm{a^{(\ell)}}_{f^{p_\ell,q_\ell}_0}$ and such that $\beta( B(x\in \eps^j) \cap  E_j^{(\ell)})\Meg \eta \beta(B(e,\eps^j))$ for every $j\in \N$ and for every $x\in G$. Consequently, by H\"older's inequality,
	\[
	\begin{split}
		\norm{\chi_{E^{(0)}_j\cap E^{(1)}_j}(x) \iota(a)(j,x)}_{L^{q_\theta,p_\theta}_{j,x}(\N,G)}&\meg \norm{\chi_{E^{(0)}_j}(x) \iota(a^{(0)})(j,x)}_{L^{q_0,p_0}_{j,x}(\N,G)}^{1-\theta} \norm{\chi_{E^{(1)}_j}(x)\iota(a^{(1)})(j,x)}_{L^{q_1,p_1}_{j,x}(\N,G)}^{\theta}\\
		&\meg C_5\norm{a^{(0)}}_{f^{p_0,q_0}_0}^{1-\theta} \norm{a^{(1)}}_{f^{p_1,q_1}_0}^{\theta}.
	\end{split} 
	\]
	Since clearly $\beta\big( B(x,\eps^j)\cap E^{(0)}\cap E^{(1)} \big)\Meg (1-2\eta)\beta(B(e,\eps^j))$ for every $j\in \N$ and for every $x\in G$, by Lemma~\ref{lem:60b} we see that there is a constant $C_6>0$ such that
	\[
	\norm{a}_{f^{p_\theta,q_\theta}_0}\meg C_6 \norm{\chi_{E^{(0)}_j\cap E^{(1)}_j}(x) \iota(a)(j,x)}_{L^{q_\theta,p_\theta}_{j,x}(\N,G)},
	\]
	so that
	\[
	\norm{a}_{f^{p_\theta,q_\theta}_0}\meg C_5 C_6 \norm{a}_{(f^{p_0,q_0}_{0})^{1-\theta}(f^{p_1,q_1}_{0})^\theta}
	\]
	by the arbitrariness of $a^{(0)}$ and $a^{(1)}$.

	Conversely, take $a\in f^{p_\theta,q_\theta}_0$, and assume, for the sake of definiteness, that $p_\theta/p_0-q_\theta/q_0\Meg 0$ (interpreting $p_\theta/p_0$ as $1$ when $p_\theta=\infty$, and analogously for similar expressions), so that $p_\theta/p_1-q_\theta/q_1\meg 0$. By Lemma~\ref{lem:59b}, we may take a sequence $(E_j)$ of $\beta$-measurable  subsets of $G$ such that $\beta(  B(x,\eps^j)\cap E_j)\Meg \eta\beta(B(e,\eps^j))$ for every $j\in\N$ and for every $x\in G$, and such that $\norm{\chi_{E_j}(x) \iota(a)(j,x)}_{L^{q_\theta,p_\theta}_{j,x}(\mi,\beta)} \meg C_5 \norm{a}_{f^{p_\theta,q_\theta}_0}$. For every $k\in\Z$, define 
	\[
	E'_k\coloneqq \Set{x\in G\colon \norm{\chi_{E_j}(x)\iota(a)(j,x)}_{\ell^{q_\theta}_j(\N)}>2^{k} } 
	\]
	and
	\[
	A_k\coloneqq \Set{(j,k')\in J\colon \beta(B(x_{j,k'},\eps^j)\cap E'_k)\Meg \frac 1 2 \beta(B(e,\eps^j))> \beta(B(x_{j,k'},\eps^j)\cap E'_{k+1})}.
	\]
	Observe that, if $a_{j,k'}\neq 0$ and we take $k\in\Z$ so that $ \abs{a_{j,k'}}> 2^{k}$, then $B(x_{j,k'},\eps^j)\cap E_j \subseteq E'_k$, so that $ \beta(B(x_{j,k'},\eps^j)\cap E'_k)\Meg \frac 1 2 \beta(B(e,\eps^j))$ (since $\eta>\frac 1 2 $). Since $\bigcap_k E'_k$ is $\beta$-negligible, we then see that $(j,k')\in A_k$ for some $k\in\Z$. 
	Define $a^{(0)},a^{(1)}\in [0,+\infty)^{J}$ so that
	\[
	a^{(\ell)}_{j,k'}= 2^{k(p_\theta/p_\ell-q_\theta/q_\ell)}\abs{a_{j,k'}}^{q_\theta/q_\ell}  
	\]
	if $(j,k')\in A_k$ for some $k\in\Z$, while $a^{(\ell)}_{j,k'}=0$ otherwise, with the convention $\frac{q_\theta}{q_\ell}=1$ when $q_\theta=\infty$ (so that $q_0=q_1=\infty$) and, analogously, $\frac{p_\theta}{p_\ell}=1$ when $p_\theta=\infty$. By the previous remarks, we have $\abs{a}= \abs{a^{(0)}}^{1-\theta}\abs{a^{(1)}}^\theta$. 
	
	Now, set $E'^{(0)}_k\coloneqq E'_k$  and  $E'^{(1)}_k\coloneqq G\setminus E'_{k+1}$ for every $k\in\Z$, and define 
	\[
	E^{(\ell)}_j\coloneqq E_j\cap \bigg(\bigcup_{k'\colon (j,k')\not \in \bigcup_{k\in\Z} A_k}   B(x_{j,k'},\eps^j)\cup \bigcup_{k\in\Z}\bigcup_{k'\colon (j,k')\in A_k}  [E'^{(\ell)}_k\cap B(x_{j,k'},\eps^j)]\bigg).
	\] 
	Observe that 
	\[
	\beta \big( B(x_{j,k'},\eps^j)\cap E^{(\ell)}_j \big)\Meg (\eta-1/2)\beta(B(e,\eps^j))
	\] 
	for every $(j,k')\in J$. Consequently, by Lemma~\ref{lem:60b}   there is a constant $C_7>0$ such that 
	\[
	\begin{split}
	\norm{a^{(\ell)}}_{f^{p_\ell,q_\ell}_0}&\meg C_7 \norm{\chi_{E^{(\ell)}_j}(x) \iota(a^{(\ell)})(j,x)}_{L^{q_\ell,p_\ell}_{j,x}(\N,G)}\\
		&\meg C_7 \norm*{\norm*{\norm{\chi_{E_k'^{(\ell)}\cap B(x_{j,k'},\eps^j)}(x) \chi_{E_j}(x) a_{j,k'}^{(\ell)} }_{\ell^{q_\ell}_{(j,k')}(A_k)}}_{\ell^{q_\ell}_k(\Z)}}_{L^{p_\ell}_x(G)}\\
		&\meg C_7 \norm*{\norm*{ 2^{k(p_\theta/p_\ell -q_\theta/q_\ell)}\norm{\chi_{E_k'^{(\ell)}\cap B(x_{j,k'},\eps^j)}(x) \chi_{E_j}(x) a_{j,k'}  }_{\ell^{q_\theta}_{(j,k')}(A_k)}^{q_\theta/q_\ell} }_{\ell^{q_\ell}_k(\Z)}}_{L^{p_\ell}_x(G)}\\
		&=C_7\norm*{\norm*{ 2^{k(p_\theta/p_\ell -q_\theta/q_\ell)}\chi_{E_k'^{(\ell)}}(x)\norm{\chi_{E_j}(x)\iota(a)( j,x) }_{\ell^{q_\theta}_{j}(\N)}^{q_\theta/q_\ell} }_{\ell^{q_\ell}_k(\Z)}}_{L^{p_\ell}_x(G)}\\
			&\meg C_7 2^{(p_\theta/p_\ell-q_\theta/q_\ell)_-}\norm*{\norm*{\chi_{E_k'^{(\ell)}}(x)\norm{\chi_{E_j}(x)\iota(a)(j,x) }_{\ell^{q_\theta}_{j}(\N)}^{p_\theta/p_\ell  }    }_{\ell^{q_\ell}_k(\Z)}}_{L^{p_\ell}_x(G)}\\
			&\meg C_7 2^{(p_\theta/p_\ell-q_\theta/q_\ell)_-}\norm{\chi_{E_j}(x) \iota(a)(j,x)}_{L^{q_\theta,p_\theta}_{j,x}(\N,G)},
	\end{split}
	\]
	where the fifth inequality follows from the following facts:
	\begin{itemize}
		\item  if $x\in E'^{(0)}_k=E'_k$, then $2^k<\norm{\chi_{E_j}(x)\iota(a)(j,x)}_{\ell^{q_\theta}_j(\N)}$, so that 
		\[
		2^{k(p_\theta/p_0 -q_\theta/q_0)}<\norm{\chi_{E_j}(x)\iota(a)(j,x)}_{\ell^{q_\theta}_j(\N)}^{p_\theta/p_0 -q_\theta/q_0}
		\]
		since $p_\theta/p_0 -q_\theta/q_0\Meg 0$;
		
		\item  if $x\in E'^{(1)}_k=G\setminus E'_{k+1}$, then $2^{k+1}\Meg \norm{\chi_{E_j}(x)\iota(a)(j,x)}_{\ell^{q_\theta}_j(\N)}$, so that 
		\[
		2^{k(p_\theta/p_1 -q_\theta/q_1)}\meg 2^{(p_\theta/p_1 -q_\theta/q_1)_-} \norm{\chi_{E_j}(x)\iota(a)(j,x)}_{\ell^{q_\theta}_j(\N)}^{p_\theta/p_1 -q_\theta/q_1}
		\] 
		since $p_\theta/p_1 -q_\theta/q_1\meg 0$. 
	\end{itemize}
	Therefore,
	\[
	\norm{a^{(\ell)}}_{f^{p_\ell,q_\ell}_0}\meg C_5  C_7 2^{(p_\theta/p_\ell-q_\theta/q_\ell)_-} \norm{a}_{f^{p_\theta,q_\theta}_0}.
	\]
	It then follows that
	\[
	\norm{a}_{(f^{p_0,q_0}_{0})^{1-\theta}(f^{p_1,q_1}_{0})^\theta}\meg C_5  C_7 2^{ \theta \abs{p_\theta/p_1-q_\theta/q_1}} \norm{a}_{f^{p_\theta,q_\theta}_0},
	\]
	whence the conclusion. 
	
	\textsc{Step IV} We now prove assertions (4), (4$'$), and (7$'$). Observe first that, by~\cite[Proposition 9.3]{KMM}, $(b^{p_0,q_0}_{\alpha_0})^{1-\theta}(b^{p_1,q_1}_{\alpha_1})^\theta=b^{p_\theta,q_\theta}_{\alpha_0}$. Assertion (4) then follows from Lemma~\ref{lem:75} and some elementary density remarks. For what concerns assertion (4$'$), observe that, arguing as in~\textsc{step I}, one may see that 
	\[
	(\mathring b^{p_0,q_0}_{\alpha_0})^{1-\theta}(\mathring b^{p_1,q_1}_{\alpha_1})^\theta=\mathring b^{p_\theta,q_\theta}_{\alpha_0},
	\]
	so that the assertion follows from Lemma~\ref{lem:75} again. Assertion (7$'$) is proved analogously.
\end{proof}

\section{Algebra Properties}

\begin{teo}\label{teo:6c}
	Take   $p, q\in (0,\infty]$ and $\alpha, \alpha'\in\R$ such that $\alpha'>\alpha_++((1/p-1)_+Q_*-\alpha)_+ $. Then,  the following hold:
	\begin{enumerate}
		\item[\textnormal{(i)}] if   $\alpha>Q_*(1/p-1)_+$, then $B^{p,q}_\alpha(G)\cap L^\infty(G)$ is an algebra under pointwise multiplication;

		\item[\textnormal{(ii)}] the mapping $(f,g)\mapsto fg$ induces a  continuous bilinear mapping 
		\[
		B^{p,q}_\alpha(G)\times (B^{\infty,\infty}_{\alpha'}\cap C^\infty(G))\to B^{p,q}_\alpha(G)  
		\]
		where $B^{\infty,\infty}_{\alpha'}\cap C^\infty(G)$ is endowed with the topology induced by $B^{\infty,\infty}_{\alpha'}(G)$. 
	\end{enumerate}
\end{teo}

\begin{teo}\label{teo:6d}
	Take   $p, q\in (0,\infty]$ and $\alpha, \alpha'\in\R$ such that $\alpha'> \alpha_++((1/\min(p,q)-1)_+Q_*-\alpha)_+ $.   Then,  the following hold:
	\begin{enumerate} 
		\item[\textnormal{(i)}]  if  $\alpha>Q_*(1/\min(p,q)-1)_+$, then $F^{p,q}_\alpha(G)\cap L^\infty(G)$ is an algebra under pointwise multiplication;
		
		\item[\textnormal{(ii)}]  the mapping $(f,g)\mapsto fg$ induces a  continuous bilinear mapping 
		\[
		F^{p,q}_\alpha(G)\times (B^{\infty,\infty}_{\alpha'}\cap C^\infty(G))\to F^{p,q}_\alpha(G),
		\]
		where $B^{\infty,\infty}_{\alpha'}\cap C^\infty(G)$ is endowed with the topology induced by $B^{\infty,\infty}_{\alpha'}(G)$. 
	\end{enumerate}
\end{teo}
 
Notice that
\[
\alpha_++(A-\alpha)_+=\max(\alpha_+,A+\alpha_-)=\begin{cases}
	\alpha & \text{if $\alpha \Meg A$}\\
	A & \text{if $0\meg \alpha \meg A$}\\
	A-\alpha & \text{if $\alpha \meg 0$}
\end{cases}
\]
for every $\alpha\in \R$ and for every $A\Meg 0$. In addition, if $D=Q_*$ (that is, if $Q_G\meg Q_*$), then the assumptions of (i) reduce to $\alpha>Q_*(1/\min(p,q)-1)_+$, while the assumptions of (ii) become empty. The rather cumbersome assumptions of (i) and (ii) are due to some technical difficulties in the proof.

Notice that, in (ii) of both theorems, we consider $B^{\infty,\infty}_{\alpha'}(G)\cap C^\infty(G)$ instead of $B^{\infty,\infty}_{\alpha'}(G)$ in order to ensure that $fg$ is well defined for every $f\in \Sr'(G)$. One may remove this assumption as follows. Take $\alpha''\in (\max(\alpha_+,\alpha_-+(1/p-1)_+Q_*) ,\alpha')$ for Besov spaces and $\alpha''\in (\max(\alpha_+,\alpha_-+(1/\min(p,q)-1)_+Q_*) ,\alpha')$ for Triebel--Lizorkin spaces, so that $B^{\infty,\infty}_{\alpha'}(G)\subseteq B^{\infty,1}_{\alpha''}(G)\subseteq B^{\infty,\infty}_{\alpha''}(G)$ continuously, thanks to~\cite[Proposition 4.10]{Calzi4}. By (ii), $(f,g)\mapsto f g$ induces a continuous bilinear mapping $X\times (B^{\infty,1}_{\alpha''}(G)\cap C^\infty(G))\to X$, where $X\in \Set{B^{p,q}_\alpha(G),F^{p,q}_\alpha(G)}$. Since $B^{\infty,1}_{\alpha''}(G)\cap C^\infty(G)$ is \emph{dense} in $B^{\infty,1}_{\alpha''}(G)$ by~\cite[Proposition 3.17]{Calzi4},  one may then extend canonically this mapping, and then induce a continuous bilinear mapping $X\times B^{\infty,\infty}_{\alpha'}(G)\to X$.

Concerning the restriction $q\Meg 1$ when $p=\infty$ in Theorem~\ref{teo:6d}, this is due to the weakness of Lemma~\ref{lem:2c} in comparison with the stronger version of Proposition~\ref{cor:18} which may be obtained combining~\cite[Lemma 2.29]{Calzi4} with~\cite{Singular}. 

We prove only Theorem~\ref{teo:6d}, since the proof of Theorem~\ref{teo:6c} is similar but simpler.
 
Before we pass to the proof, we need a few technical lemmas.
 
\begin{lem}\label{lem:28c}
	Take $p,q\in [1,\infty]$ with $\frac1p+\frac 1 q\meg 1$, $N\Meg 0$ and $(\psi_j)\in \widetilde\Psi_\eps$ for some $\eps\in (0,1)$, and set $\psi'_j\coloneqq \sum_{k=0}^{j-1}\psi_k$ for every $j\Meg 0$.  Then, for every $f\in (1+\abs{\,\cdot\,}_*)^N L^p(G)$ and for every $g\in (1+\abs{\,\cdot\,}_*)^N L^q(G) $, 
	\[
	f g=\lim_{j'\to \infty}(\Pi^{(j')}_f g+\Pi^{(j')}_g f)+\Pi(f,g)+\widetilde \Pi(f,g),
	\]
	where
	\[
	\begin{aligned}
		\Pi_f ^{(j')}g&=   \sum_{1\meg j\meg j'}  [(  f*\psi_j)(g*\psi'_j)]*\psi'_j ,\\
		\Pi_g ^{(j')}f&=  \sum_{1\meg j\meg j'}  [(  f*\psi'_j)(g*\psi_j)]*\psi'_j,
	\end{aligned}
	\]
	and
	\[
	\begin{split}
	\Pi(f,g)&=  \sum_{j\Meg 0}   [( f*\psi'_{j+1})(g*\psi'_{j+1})]*\psi_j\\
	\widetilde\Pi(f,g)&=  \sum_{j\Meg 1}   [( f*\psi_j)(g*\psi_j)]*\psi'_j
	\end{split}
	\]
	in $\Sr'(G)$.
\end{lem}

Notice that $\psi'_0=0$. In addition,   we do not know whether the limits $\lim\limits_{j'\to \infty}\Pi^{(j')}_f g$ and $\lim\limits_{j'\to \infty} \Pi^{(j')}_g f $ exist `individually' in the generality of the statement.

\begin{proof}
	As in the proof of~\cite[Corollary 3.3]{Calzi2}, we observe that $p'\meg q$, so that $(1+\abs{\,\cdot\,}_*)^N L^q(G)\subseteq (1+\abs{\,\cdot\,}_*)^{N+D+1}L^{p'}(G)$. Consequently, up to replacing $N$ with $N+D+1$, we may assume that $q=p'$.
	We wish to apply~\cite[Lemma 3.2]{Calzi2} with $\mi=\sum_{j=0}^\infty \delta_{\eps^j}$, $B_1=(1+\abs{\,\cdot\,}_*)^N L^p(G)$, $B_2=(1+\abs{\,\cdot\,}_*)^N L^{p'}(G)$, $B_3=(1+\abs{\,\cdot\,}_*)^{2N} L^1(G) $, $W'_j(t)f=  f*\psi_{[\log_\eps t]}$, and $W_j(t)f=f*\psi'_{[\log_\eps t]}$ for every $f\in B_j$, for every $t\in (0,1]$, and for every $j=1,2,3$, where $[\,\cdot\,]$ denotes the integer part.  The proof is then almost identical to that of~\cite[Corollary 3.3]{Calzi2}. The main difference is that we have to deal with the additional term $\widetilde\Pi(f,g)$, while the term $[(f*\psi'_0)(g*\psi'_0)]*\psi'_0$ vanishes.
\end{proof}
  
\begin{proof}[Proof of Theorem~\ref{teo:6c}.]
	For (i), we  take $f,g\in F^{p,q}_\alpha(G)\cap L^\infty(G)\subseteq L^{\max(1,p)}(G)\cap L^\infty(G)$ (cf.~\cite[Proposition 4.10]{Calzi4}), while for (ii) we take $f\in F^{p,q}_\alpha(G)\cap L^{\max(p,1)}(G)$ and $g\in B^{\infty,\infty}_{\alpha'}(G)\cap C^\infty(G) \subseteq L^\infty(G)$. This will allow us to define canonically $f g$ and to apply Lemma~\ref{lem:28c}. Notice that the extra restriction on $f$ for (ii) is harmless, since our estimates will only depend on the quasi-norm of $F^{p,q}_\alpha(G)$. In fact, if $(\psi_j)\in \widetilde \Psi_{\eps^{\grado}}$ for some $\eps\in (0,1)$ and $\psi'_j\coloneqq \sum_{k=0}^{j-1} \psi_k$, then the sequence $f*\psi'_j$ is uniformly bounded in $F^{p,q}_\alpha(G)$, takes elements in $F^{p,q}_\alpha(G)\cap L^{\max(1,p)}(G)$ by~\cite[Proposition 4.10]{Calzi4}, and converges to $f$ in $\Sr'(G)$. Since the $(f*\psi'_j)g$ are uniformly bounded in $F^{p,q}_\alpha(G)$, they are in particular uniformly bounded in $\Sr'(G)$, hence converge to $f g$ in $\Sr'(G)$ (and not only in the space of general distributions),\footnote{Notice that $X g$ need not grow polynomially when $X\in U(G)$ and $\deg(X)>\alpha'$, so that this extra step is necessary. For example, when $G=\R$, one may take $g=(1-D^2)^{-N}(\sin(\ee^{\,\cdot\,}))$, where $D$ denotes the usual derivative, so that $g\in B^{\infty,\infty}_{2N}(\R)$, but $D(1-D^2)^N g$ does \emph{not} have polynomial growth.} so that $f g\in F^{p,q}_\alpha(G)$ by lower semi-continuity.
	
	\textsc{Step I} Take a reduced $(\eps,1/4,2)$-lattice $(x_{j,k})_{(j,k)\in J}$ on $G$, and choose, for every $j\in\N$, a Borel partition $(B_{j,k})_{k}$ of $G$ such that $B(x_{j,k}, \eps^j/4)\subseteq B_{j,k}\subseteq B(x_{j,k}, \eps^j/2)$ for every $k$ with $(j,k)\in J$. Take $(\psi_j)$ and $(\psi'_j)$ as above, and define $\widetilde f, \widetilde g\colon \N\times G\to [0,\infty)$ so that
	\[
	\widetilde f (j,x)\coloneqq \max_{\overline B(x_{j,k}, \eps^j/2)} \abs{f*\psi_j} \chi_{B_{j,k}}(x) \qquad \text{and} \qquad \widetilde g (j,x)\coloneqq \max_{\overline B(x_{j,k}, \eps^j/2)} \abs{g*\psi'_j} \chi_{B_{j,k}}(x)
	\]
	for every $(j,x)\in \N\times G$, so that $\abs{f*\psi_j}\meg \widetilde f(j,\,\cdot\,)$ and $\abs{g*\psi'_j}\meg \widetilde g(j,\,\cdot\,)$ for every $j\in\N$. Take $p_0\in (0,\min(1,p,q))$ so that $\alpha>Q_*(1/p_0-1)$ for (i) and $\alpha'>\alpha_++((1/p_0-1)Q_*-\alpha)_+$ for (ii).
	In addition, take  $a>D/p_0$ and take  $m\in \N $ such that $m\grado>2 \alpha-Q_*/p_0+a$.   Observe that, by Lemma~\ref{cor:13}, there is a constant $C_1>0$ such that
	\[
	\abs{(\psi'_j*\psi_{j'})(x)}\meg \frac{C_1 \eps^{-\min(j,j')Q_*}}{(1+\abs{x}_*/\eps^{j'})^a} \min(1,\eps^{(j'-j)m\grado})  =\frac{C_1\eps^{-\min(j,j')Q_*}}{(1+\abs{x}_*/\eps^{j'})^a}  \eps^{(j'-j)_+ m\grado}
	\]
	Consequently, 
	\[
	\begin{split}
		\abs{ [(f*\psi_j)(g*\psi'_j)]*\psi'_j*\psi_{j'} }&\meg C_1\eps^{(j'-j)_+ (m \grado +Q_*)} \Nc_{1,a,\eps^{j'}}((f*\psi_j)(g*\psi'_j))\\
			&\meg C_1\eps^{(j'-j)_+ (m \grado+Q_*)} \Nc_{1,a,\eps^{j'}}(\widetilde f(j,\,\cdot\,)\widetilde g(j,\,\cdot\,)).
	\end{split}
	\]
	Now, observe that, setting $f_{j,k}\coloneqq \widetilde f(j,x_{j,k})$ and $g_{j,k}\coloneqq \widetilde g(j,x_{j,k})$ to simplify the notation,
	\[
	\begin{split}
		\Nc_{1,a,\eps^{j'}}(\widetilde f(j,\,\cdot\,)\widetilde g(j,\,\cdot\,))&= \eps^{-j' Q_*}\sum_k \int_{B_{j,k}}\frac{f_{j,k}g_{j,k}}{(1+\abs{x}_*/\eps^{j'})^a}\,\dd \beta(x)\\
		&\meg \eps^{-j' Q_*}\sum_k \beta(B_{j,k})\frac{f_{j,k}g_{j,k}}{(1/2+\abs{x_{j,k}}_*/\eps^{\min(j,j')})^a}\\
		&\meg 2^a\eps^{-j' Q_*}\bigg(\sum_k \beta(B_{j,k})^{p_0}\frac{f_{j,k}^{p_0}g_{j,k}^{p_0}}{(1+\abs{x_{j,k}}_*/\eps^{\min(j,j')})^{ap_0}}\bigg)^{1/p_0}\\
		&\meg 2^a C_1' \eps^{-j' Q_*+jQ_*(1-1/p_0)}\bigg(\sum_k \beta(B_{j,k}) \frac{f_{j,k}^{p_0}g_{j,k}^{p_0}}{(1+\abs{x_{j,k}}_*/\eps^{\min(j,j')})^{ap_0}}\bigg)^{1/p_0}\\
		&\meg 4^a C_1' \eps^{-j' Q_*+jQ_*(1-1/p_0)}\bigg(\int_G \frac{\widetilde f(j,x)^{p_0} \widetilde g(j,x)^{p_0}}{(1+\abs{x}_*/\eps^{\min(j,j')})^{ap_0}}\,\dd \beta(x)\bigg)^{1/p_0}
	\end{split}
	\]
	where $C_1'=\sup_{j,k} (\eps^{-jQ_*}\beta(B_{j,k}))^{1-1/p_0}$. Consequently,
	\[
	\abs{ [(f*\psi_j)(g*\psi'_j)]*\psi'_j*\psi_{j'} }\meg C_1''\eps^{(j'-j)_+  m \grado -(j-j')_+(1/p_0-1)_+Q_*} \Nc_{p_0,ap_0,\eps^{\min(j,j')}}(\widetilde f(j,\,\cdot\,)\widetilde g(j,\,\cdot\,))
	\] 
	where $C_1''=C_1 4^a C_1'$, since
	\[
	(j'-j)_+-j'+j(1-1/p_0)+\min(j,j')/p_0=(1/p_0-1)(\min(j,j')-j)=-(j-j')_+(1/p_0-1).
	\]
	Next, observe that, since
	\[
	1+\abs{x}_*/\eps^{\min(j,j')}\Meg \eps^{(j'-j)_+}(1+\abs{x}_*/\eps^{j'})
	\]
	for every $x\in G$, one has
	\[
	\Nc_{p_0,ap_0,\eps^{\min(j,j')}}(\widetilde f(j,\,\cdot\,)\widetilde g(j,\,\cdot\,))\meg \eps^{(j'-j)_+(Q_*/p_0-a)}\Nc_{p_0,ap_0,\eps^{j'}}(\widetilde f(j,\,\cdot\,)\widetilde g(j,\,\cdot\,))
	\]
	for every $j,j'\in\N$. Consequently,
	\[
	\abs{ [(f*\psi_j)(g*\psi'_j)]*\psi'_j*\psi_{j'} }\meg C_1''\eps^{(j'-j)_+ ( m \grado+Q_*/p_0-a) -(j-j')_+(1/p_0-1)_+Q_*} \Nc_{p_0,ap_0,\eps^{j'}}(\widetilde f(j,\,\cdot\,)\widetilde g(j,\,\cdot\,)).
	\]

	Assume first that $p<\infty$. For case (i), set $\alpha''\coloneqq \alpha$, whereas for case (ii) fix $\alpha''\in (\max(\alpha,(1/p_0-1)Q_*),\alpha'-\alpha_-)$. This is possible since $\alpha'-\alpha_-> \alpha_++((1/p_0-1)Q_*-\alpha)_+-\alpha_-=\alpha+((1/p_0-1)Q_*-\alpha)_+=\max(\alpha,(1/p_0-1)Q_*)$. 
	Then, by means of Proposition~\ref{cor:18} and Lemma~\ref{lem:25d}, we see that there is a constant $C_2>0$ such that, for every $j''\in\N$,
	\[
	\begin{split}
		&\norm*{\eps^{-j'\alpha}[(\Pi_f^{(j'')} g)*\psi_{j'}](x)}_{L^{q,p}(\N,G)}\\
			&\qquad\meg   C_1'' \norm*{ \eps^{-j'\alpha}\sum_{j\in\N} \eps^{(j'-j)_+ (m \grado+Q_*/p_0-a)-(j-j')_+(1/p_0-1)Q_*}\Nc_{p_0,a p_0,\eps^{j'} }(\widetilde f(j,\,\cdot\,)\widetilde g(j,\,\cdot\,))(x)   }_{L^{q,p}_{j',x}(\N,G)}\\
			&\qquad\meg   C_1'' \norm*{ \eps^{-j'\alpha} \Nc_{p_0,a p_0,\eps^{j'} }\Big(\sum_{j\in\N}\eps^{(j'-j)_+ (m \grado+Q_*/p_0-a)-(j-j')_+(1/p_0-1)Q_*}\widetilde f(j,\,\cdot\,)\widetilde g(j,\,\cdot\,)\Big)(x)   }_{L^{q,p}_{j',x}(\N,G)}\\
			&\qquad\meg C_2\norm*{ \eps^{-j'\alpha}  \sum_{j\in\N}\eps^{(j'-j)_+ (m \grado+Q_*/p_0-a) -(j-j')_+(1/p_0-1)Q_*}\widetilde f(j,x)\widetilde g(j,x)   }_{L^{q,p}_{j',x}(\N,G)}\\ 
			&\qquad\meg C_2^2 \norm*{  \eps^{-j\alpha'' } 
				\widetilde f(j,x) \widetilde g(j,x)}_{L^{q,p }_{j,x}(\N,G)},
 	\end{split}
	\]
	where the second inequality follows from the subadditivity of the mapping $s\mapsto s^{p_0}$, the third inequality follows from Proposition~\ref{cor:18} while the last inequality follows from Lemma~\ref{lem:25d}, since $ (j'-j)_+ (m \grado+Q_*/p_0-a)-(j-j')_+(1/p_0-1)Q_*-j'\alpha+j\alpha''\Meg (j'-j)_+(m \grado+Q_*/p_0-a-\alpha)+(j-j')_+(\alpha''-(1/p_0-1)Q_*)$ for every $j,j'\in \N$, since $\alpha''\Meg \alpha$.
	For what concerns case (i), one has $\alpha''=\alpha$, so that by Propositions~\ref{prop:14} and~\ref{cor:12}, and Young's inequality, there is a constant $C_3>0$ such that
	\[
	\begin{split} 
		\norm*{  \eps^{-j\alpha } 
			\widetilde f(j,x) \widetilde g(j,x)}_{L^{q,p }_{j,x}(\N,G)}&\meg \norm*{\eps^{-j\alpha} \widetilde f(j,x)}_{L^{q,p  }_{j,x}(\N,G)} \norm*{\sup_{j\in\N} \widetilde g(j,x)}_{L^{\infty}(G)}\\
			&\meg C_3\norm{f}_{F^{p,q}_\alpha(G)} \norm{g}_{L^\infty(G)}
	\end{split}
	\]
	since $f,g\in F^{p,q}_\alpha(G)\cap L^\infty(G)$. For what concerns case (ii), one has  $f\in F^{p,q}_\alpha(G)\cap L^{\max(1,p)}(G)$ and $g\in B^{\infty,\infty}_{\alpha'}(G)\cap C^\infty(G)$, so that one may  find a constant $C_3'>0$ such that
	\[
	\begin{split}
	\norm*{  \eps^{-j\alpha'' } 
	 	\widetilde f(j,x) \widetilde g(j,x)}_{L^{q,p }_{j,x}(\N,G)}&\meg \norm*{\eps^{-j\alpha} \widetilde f(j,x)}_{L^{q,p  }_{j,x}(\N,G)} \norm*{\sup_{j\in\N} \eps^{j(\alpha-\alpha'')}\widetilde g(j,x)}_{L^{\infty}(G)}\\
	 &\meg C_3'\norm{f}_{F^{p,q}_\alpha(G)} \norm{g}_{B^{\infty,\infty}_{\alpha'}(G)}
	\end{split}
	\]
	since   $\alpha'>\alpha''+\alpha_-\Meg \alpha''-\alpha$.
	
	We shall also consider the case $f\in B^{\infty,\infty}_{\alpha'}(G)\cap C^\infty(G)$ and $g\in F^{p,q}_\alpha(G)$, which is relevant when applying symmetric arguments to estimate $\Pi_g^{(j'')} f$.
	Then, as above we see that there is a constant $C_3''>0$ such that\footnote{Notice that, even though we cannot apply Proposition~\ref{prop:14} to control $ \norm{\eps^{j(\alpha'-\alpha'')}  \widetilde g(j,x)}_{L^{q,p  }_{j,x}(\N,G)}$ with $\norm{  \eps^{j(\alpha'-\alpha'')}(g*\psi'_j)(x)}_{L^{q,p  }_{j,x}(\N,G)}$, we may still use the more general~\cite[Proposition 2.40]{Calzi4}.}
	\[
	\begin{split} 
		\norm*{  \eps^{-j\alpha'' } 
			\widetilde f(j,x) \widetilde g(j,x)}_{L^{q,p }_{j,x}(\N,G)}&\meg\norm*{\sup_{j\in\N}\eps^{-j\alpha'}\widetilde f(j,x)}_{L^{\infty}(G)} \norm*{\eps^{j(\alpha'-\alpha'')}  \widetilde g(j,x)}_{L^{q,p  }_{j,x}(\N,G)} \\
		&\meg C_3''\norm{f}_{B^{\infty,\infty}_{\alpha'}(G)} \norm*{  \eps^{j(\alpha'-\alpha'')}(g*\psi'_j)(x)}_{L^{q,p  }_{j,x}(\N,G)} .
	\end{split}
	\]
	In order to deal with the second term, observe that, by Lemma~\ref{lem:25d}, there is a constant $C_4>0$ such that
	\[
	\begin{split}
		 \norm*{  \eps^{j(\alpha'-\alpha'')}(g*\psi'_j)(x)}_{L^{q,p  }_{j,x}(\N,G)}&\meg \norm*{\eps^{j(\alpha'-\alpha'')}\sum_{j'<j} \abs{(g*\psi_{j'})(x)}   }_{L^{q,p  }_{j,x}(\N,G)}\\
		 	&\meg  \norm*{\sum_{j'\in\N} \eps^{\abs{j-j'} (\alpha'-\alpha'' ) } \eps^{-j'\alpha}\abs{(g*\psi_{j'})(x)}   }_{L^{q,p  }_{j,x}(\N,G)}\\
		 	&\meg C_4 \norm*{ \eps^{-j\alpha}  \abs{(g*\psi_{j})(x)}   }_{L^{q,p  }_{j,x}(\N,G)}
	\end{split}
	\]
	since $\alpha''< \alpha'-\alpha_-$.
	 The desired inequality follows. 
	
	Assume now that $p=\infty$.  
	In order to estimate $\norm{[(\Pi^{(j'')}_f g )*\psi_{j'}](x)}_{\Cc^q_{(j',x)}(\eps)}$, one may   proceed as in the above computations, replacing Proposition~\ref{cor:18} with  Lemma~\ref{lem:25d}   and observing that Lemma~\ref{lem:25d} may be replaced by Lemma~\ref{lem:25c} since there is $c>0$ such that  $\widetilde f(j,\,\cdot\,) \widetilde g(j,\,\cdot\,)\meg c\Nc_{p_0,ap_0,\eps^j}(f(j,\,\cdot\,) \widetilde g(j,\,\cdot\,))$ (cf.~\cite[proof of Remark 4.5]{Calzi4}).

	The term $\widetilde \Pi(f,g)$ may be dealt with  the same techniques, whereas the term $\Pi^{(j'')}_g f$ may be dealt with swapping the roles of $f$ and $g$ in the above computations.

	\textsc{Step II} We now deal with the term $\Pi(f,g)$. In this case, we observe that we may take $(\widetilde \psi_j)\in  \Psi_{\eps^\grado}$ such that $\psi_j=\eps^{j m\grado}\Lc^m \widetilde \psi_j=\eps^{j m \grado} (\Lc^m)^R \widetilde \psi_j$ for every $j\Meg 1$. In addition, there are two finite families $(X_k)_{k\in K}$ and $(Y_k)_{k\in K}$ of elements of $U(G)$ such that $\Lc^m(\phi\psi)=\sum_{k\in K} (X_k \phi)(Y_k\psi)$ for every $\phi,\psi\in C^\infty(G)$, and such that $\deg(X_k)+\deg(Y_k)\meg m\grado$ for every $k\in K$. Then,
	\[
	\begin{split}
		[(f*\psi'_{j+1})(g*\psi'_{j+1})]*\psi_j*\psi_{j'}=\eps^{j m \grado}\sum_k [(f*X_k \psi'_{j+1})(g*Y_k \psi'_{j+1})]*\widetilde \psi_j*\psi_{j'}
	\end{split}
	\]
	for every $j ,j'\in \N$. One may then proceed as in the proof of~\cite[Theorem 3.1]{Calzi2} with the above modifications. The proof is therefore complete. 
\end{proof}

\begin{cor}\label{cor:20b}
	Take $p,q\in (0,\infty]$ and $\alpha> Q_*/p$. Then $B^{p,q}_\alpha(G)$ and, if $p<\infty$, $B^{p,\min(1,p,q)}_{Q_*/p}(G)$,  are algebras under pointwise multiplication. 
	
	Assume, in addition, that    $\alpha>(1/\min(p,q)-1)_+Q_*$.
	Then, the same holds for $F^{p,q}_\alpha(G)$ and, if $p<1$, for $F^{p,q}_{Q_*/p}(G)$.
\end{cor}

\begin{proof}
	This follows from Theorem~\ref{teo:6c} and~\cite[Propositions 4.9 and 4.10]{Calzi4}.
\end{proof}

\begin{cor}
	There is a sequence $(T_j)$ of continuous linear mappings from $\Sr'(G)$ into $C^\infty_c(G)$ such that $T_j f\to f $ in $\Sr'(G)$ for every $f\in \Sr'(G)$ and such that, for every  $p,q\in (0,\infty]$ and for every $\alpha\in \R$, the $T_j$  induce  equicontinuous endomorphisms of $B^{p,q}_\alpha(G)$ and of $F^{p,q}_\alpha(G)$.
\end{cor}

The operators $T_j$ may be used to transfer results from $\mathring B^{p,q}_\alpha(G)$ and $\mathring F^{p,q}_\alpha(G)$ to $B^{p,q}_\alpha(G)$ and $F^{p,q}_\alpha(G)$, respectively, arguing by lower semi-continuity.

\begin{proof}
	Take $(\psi_j)\in \widetilde \Psi_\eps$ for some $\eps\in (0,1)$, and set $\tau_j\coloneqq \chi_{B(e,j+1)}*\phi$, where $\phi$ is a positive element of $C^\infty_c(G)$ supported in $ B(e,1)$ and with integral $1$, so that $\chi_{B(e,j)}\meg \tau_j\meg \chi_{B(e,j+2)}$ for every $j\in\N$. Define $T_j f\coloneqq \tau_j \sum_{k=0}^j f*\psi_k$ for every $f\in \Sr'(G)$, so that clearly $T_j f\in C^\infty_c(G)$ and $T_j f\to f$ in $\Sr'(G)$ (cf.~\cite[Lemma 3.7]{Calzi4}). It then suffices to show that the $T_j$ induce equicontinuous endomorphisms of $B^{p,q}_\alpha(G)$ and of $F^{p,q}_\alpha(G)$. This, in turn, follows easily by means of Theorem~\ref{teo:6c}.
\end{proof}

\section{Localization}

The content of this section is an elaboration of the techniques developed in~\cite{Jordi}. 

\begin{deff}\label{def:5}
	Take $\eps\in (0,1)$, $S\in\N$ and $M\in\N$. We say that a family $(\psi_j)_{j\Meg -M}$ of elements of $\Sr(G)$ satisfies the $(\eps,S,\Lc)$-Calder\'on condition if there are families $(\widetilde \psi_j)_{j\Meg 1}$, $(\eta_j)_{j\Meg -M}$, and $(\widetilde \eta_j)_{j\Meg 1}$ of elements of $\Sr(G)$ such that the following hold:
	\begin{enumerate}
		\item[\textnormal{(1)}]   $\psi_j= (\eps^{j\grado}\Lc^R)^S \widetilde \psi_j$ and $\eta_j=(\eps^{j\grado}\Lc)^{3S} \widetilde \eta_j$ for every $j\Meg 1$;
		
		\item[\textnormal{(2)}] $\sum_{j=-M}^{+\infty} \psi_j*\eta_j =\delta_e$ in $\Oc'_{C,R}(G)$ and in $ \Oc'_{C,L}(G)$;
		
		\item[\textnormal{(3)}] for every $X,Y\in U(G)$ and for every $ N\in\N$, there is a constant $C>0$ such 
		\[
		\abs{(X Y^R \widetilde\psi_j)(x)}, \abs{(X Y^R  \widetilde\eta_j)(x)}\meg C\frac{\eps^{-j(Q_*+\deg(X)+\deg(Y))}}{(1+\abs{x}_*/\eps^j)^N}
		\]
		for every $x\in G$ and for every $j\Meg 1$.
	\end{enumerate}
	
	Given $\rho>0$, we say that $(\psi_j)$ satisfies the strong $(\eps,S,\rho,\Lc)$-Calder\'on condition if, in addition, the following condition hold:
	\begin{enumerate}
		\item[\textnormal{(4)}] $\supp \psi_j$, $\supp \eta_j$,   $\supp \widetilde \psi_{j'}$, $\supp \widetilde \eta_{j'} \subseteq B(e,\rho)$ for every $j\Meg -M$ and for every $j'\Meg 1$.
	\end{enumerate} 
\end{deff}

\begin{lem}\label{lem:64}
	Take $S\in\N$, $\rho>0$, and $\eps\in (0,1)$. Then, there are $M\in\N$ and a family $(\psi_j)_{j\Meg -M}$ satisfying the strong $(\eps,S,\rho,\Lc)$-condition. 
\end{lem}

\begin{proof}
	Take $\rho'\in (0,\rho]$ so that $B(e,\rho')^2\subseteq B(e,\rho)$, and take $(\tau_j)\in \Psi_{\eps^{\grado}}$ such that $\sum_{j=0} \tau_j*\tau_j=\delta_e$ in $\Oc'_{C,R}(G)$ and in $\Oc'_{C,L}(G)$, thanks to~\cite[Lemma 3.7]{Calzi4}. Then, take $\chi\in C^\infty_c(G)$ so that $\chi_{B(e,\rho'/3)} \meg \chi \meg \chi_{B(e,(2/3)\rho')}$, and set $\widetilde\eta_j\coloneqq \chi (\eps^{j\grado} \Lc)^{-3S}\tau_j$, $\eta_j=(\eps^{j \grado}\Lc)^{3S}\widetilde \eta_j $, $\widetilde\psi_j\coloneqq \chi (\eps^{j \grado}\Lc)^{-S} \tau_j$ and $\psi_j\coloneqq (\eps^{j\grado} \Lc^R)^S \widetilde \psi_j$ for every $j\Meg 1$. Then, (1), (3), (4) in Definition~\ref{def:5} follow by construction and Proposition~\ref{cor:12}. Next, observe that, at least formally,
	\[
	\begin{split}
		\delta_e-\tau_0*\tau_0&= \sum_{j=1}^\infty \tau_j*\tau_j\\
			&=\sum_{j=1}^\infty   \psi_j *\eta_j + \sum_{j=1}^\infty [(\Lc^R)^S((1-\chi)\Lc^{-S}\tau_j)*\eta_j + \tau_j*\Lc^{3S}((1-\chi)\Lc^{-3S}\tau_j)].
	\end{split}
	\] 
	Let us show that the sum $\sum_{j=1}^\infty [(\Lc^R)^S((1-\chi)\Lc^{-S}\tau_j)*\eta_j + \tau_j*\Lc^{3S}((1-\chi)\Lc^{-3S}\tau_j)]$ converges to some $\psi$ in $\Sr(G)$. Take $X\in U(G)$ and $N\in\N$, and assume that $N> Q_*+\deg(X)$.   
	Observe that, by Proposition~\ref{cor:12}, there is a constant $C_1>0$ such that 
	\[
	\abs{YZ^R\Lc^{-3S}\tau_j(x)}\meg C_1\abs{Y}\abs{Z} \frac{\eps^{-j(Q_*-3S\grado+\deg(Y)+\deg(Z))}}{(1+\abs{x}_*/\eps^j)^{2N+D+1}}
	\]
	for every $x\in G$, for every $j\Meg 1$, and for every $Y,Z\in U_{3 S\grado+\deg(X)}$.
	Consequently, there is a constant $C_2>0$ such that 
	\[
	\begin{split}
		\abs{(\Lc^R)^S((1-\chi)\Lc^{-S}\tau_j)(x)}&\meg C_2 \frac{\eps^{-j Q_*}\chi_{G\setminus B(e, \rho'/3)(x)}}{(1+\abs{x}_*/\eps^j)^{2N+D+1}}\meg C_2\frac{\eps^{-j Q_*}(3 \eps^j/\rho')^N}{(1+\abs{x}_*/\eps^j)^{N+D+1}},\\
		\abs{X\eta_j(x)}&\meg C_2 \abs{X}\frac{\eps^{-j (Q_*+\deg(X))} }{(1+\abs{x}_*/\eps^j)^{2N+D+1}},\\
		\abs{X\Lc^{3S}((1-\chi)\Lc^{-3S}\tau_j)(x)}&\meg C_2 \abs{X}\frac{\eps^{-j (Q_*+\deg(X))}\chi_{G\setminus B(e, \rho'/3)(x)}}{(1+\abs{x}_*/\eps^j)^{2N+D+1}}\meg C_2\abs{X}\frac{\eps^{-j (Q_*+\deg(X))}(3 \eps^j/\rho')^N}{(1+\abs{x}_*/\eps^j)^{N+D+1}},
	\end{split}
	\]
	for every $j\Meg 1$ and for every $x\in G$.
	Consequently,  by Young's inequality and Lemma~\ref{lem:37}, there is a constant $C_3>0$ such that
	\[
	\begin{split}
	\abs{  [(\Lc^R)^S((1-\chi)\Lc^{-S}\tau_j)*X\eta_j + \tau_j*X\Lc^{3S}((1-\chi)\Lc^{-3S}\tau_j)](x)} 	&\meg C_3 \frac{\eps^{j(N-Q_*-\deg(X))}}{(1+\abs{x}_*/\eps^j)^N}\\
		&\meg C_3 \frac{\eps^{j(N-Q_* -\deg(X))}}{(1+\abs{x}_*)^N}
	\end{split}
	\]
	for every $j\Meg 1$ and for every $x\in G$. It then follows that the sum 
	\[
	\sum_{j=1}^\infty \norm*{ (1+\abs{\,\cdot\,}_*)^N X[(\Lc^R)^S((1-\chi)\Lc^{-S}\tau_j)* \eta_j + \tau_j* \Lc^{3S}((1-\chi)\Lc^{-3S}\tau_j)] }_{L^\infty(G)}
	\]
	converges, so that our claim follows by the arbitrariness of $N$ and $X$. 
	Observe that $\psi+\tau_0*\tau_0=\delta_e- \sum_{j=1}^\infty \psi_j*\eta_j $, so that $\psi+\tau_0*\tau_0$ is supported in $B(e,\rho')^2\subseteq B(e,\rho)$. Consequently, by~\cite[Théorème 3.1]{DixmierMalliavin} there are $M\in\N$ and $\psi_j,\eta_j\in C^\infty_c(G)$, $j=-M,\dots,0$, such that $\supp\psi_j,\supp \eta_j\subseteq B(e,\rho)$ for every $j=-M,\dots,0$ and such that $\tau_0*\tau_0+\psi=\sum_{j=-M}^0 \psi_j*\eta_j$. The assertion follows.
\end{proof}
 
 \begin{deff}
 	For every $\alpha\Meg 0$ we define $\Bc_\alpha(G)$ as the space of $f\in \Sr'(G)$ such that there is a constant $C>0$ such that
 	\[
 	\abs{\langle f,\phi\rangle}\meg C   \sup_{X\in U_\alpha, \abs{X}\meg 1} \norm{(1+\abs{\,\cdot\,}_*)^\alpha X\phi }_{L^1(G)}
 	\]
 	for every $\phi \in \Sr(G)$.
 \end{deff}
 
 Notice that, since $  W^{\alpha,1}(G)\subseteq B^{1,\infty}_{[\alpha/\dd] \dd}(G)$ by~\cite[Theorem 3.20 and Propositions 3.19 and 4.10]{Calzi4}, one has $B^{\infty,1}_{-[\alpha/\dd]\dd}(G)\subseteq   \Bc_{\alpha}(G)$ by~\cite[Theorem 5.1]{Calzi4}. In particular, if $\alpha$ is an integer multiple of $\dd$, then $B^{\infty,1}_{-\alpha}(G)\subseteq \Bc_{\alpha}(G)$. By means of~\cite[Propositions 4.9 and 4.10]{Calzi4} we then see that, if $\alpha'>Q_*/p-[\alpha/\dd]\dd$, then $B^{p,q}_{\alpha'}(G), F^{p,q}_{\alpha'}(G)\subseteq \Bc_\alpha(G)$ for every $p,q\in (0,\infty]$. In other words, every Besov or Triebel--Lizorkin space is contained in some $\Bc_\alpha$.

\begin{lem}\label{lem:65}
	Take $N>0$, $\eps\in (0,1)$,   and a family $(\psi_j)_{j\Meg -M}$ of elements of $\Sr(G)$. Then, there is a constant $C>0$ such that  the following holds. If  $(\psi_j)_{j\Meg -M}$ satisfies the $(\eps,S,\Lc)$-Calder\'on condition for some integer $S\Meg N/\grado$, then 
	\begin{equation}\label{eq:14}
		\Nc_{\infty,N,\eps^j}(f*\psi_{j}) \meg  \Big(C\sum_{j'=-M}^\infty [\eps^{\abs{j-j'}N }  \Nc_{p,N, \eps^{j'}}(f*\psi_{j'})]^p\Big)^{1/p}
	\end{equation}
	for every $p\in (0,1]$, for every $j\Meg -M$, and for every $f\in \Bc_{S\grado}(G)$ (for every $f\in \Sr'(G)$ when $p=1$).
\end{lem}

This is an elaboration of~\cite[Lemma 3.5]{Jordi}. Notice that, if $(\psi_j)$ satisfies the $(\eps,S,\Lc)$-Calder\'on condition for every $S\in \N$, then~\eqref{eq:14} holds for every $f\in \bigcup_{\alpha>0}\Bc_\alpha(G)=\Sr'(G)$ (since the constant $C$ depends on $N$ but not on $S$). Nonetheless, if we wish the $\psi_j$ to be compactly supported for $j\Meg 1$, then we cannot hope to achieve this stronger assertion.

\begin{proof}
	\textsc{Step I} Observe first that
	\[
	f*\psi_{j}= \sum_{j'\Meg -M} [f*(\psi_{j'}*\eta_{j'})]*\psi_{j}= \sum_{j'\Meg -M} (f*\psi_{j'})*(\eta_{j'}*\psi_{j}),
	\]
	with convergence in $\Sr'(G)$ and in $C^\infty(G)$. Next, observe that, by Lemma~\ref{lem:40},    there is a constant $C_1>0$ such that
	\[
	\abs{\eta_{j'}*\psi_{j}}\meg C_{1} \eps^{ (j-j')_+N+ (j'-j)_+3N-\min(j,j')Q_*} (1+\abs{\,\cdot\,}_*/\eps^{\min(j,j')})^{-N} 
	\]
	for every $j,j'\Meg -M$. Therefore,
	\[
	\abs{(f*\psi_{j'})*(\eta_{j'}*\psi_{j})}\meg  C_{1} \eps^{(j-j')_+N+ (j'-j)_+3N}\Nc_{1,N, \eps^{\min(j,j')}}(f*\psi_{j'}) 
	\]
	for every $j,j'\Meg -M$. Since
	\[
	(1+\eps^{-\min(j,j')}\abs{x}_*)^{N}(1+\eps^{-j}\abs{y}_*)^{N}\Meg   (1+\eps^{-\min(j,j')}\abs{x y}_*)^{N}\Meg \eps^{(j'-j)_+ N} (1+\eps^{-j'}\abs{x y}_*)^{N}
	\]
	for every $x,y\in G$, it then follows that\footnote{Here, we should in fact get $2N+Q_*$, which would lead to weaker requirements on $(\eta_j)$. Since, however, $Q_*$ is not an integer, in general, we preferred to avoid the  more precise, but substantially more cumbersome, conditions that would have arisen.}
	\[
	\Nc_{\infty,N,\eps^{j}}[(f*\psi_{j'})*(\eta_{j'}*\psi_{j})]\meg C_{1} \eps^{(j-j')_+N+ (j'-j)_+2N} \Nc_{1,N,\eps^{j'}}(f*\psi_{j'})
	\]
	for every $j,j'\Meg -M$.  The assertion follows when $p=1$ (for every $f\in \Sr'(G)$). 
	
	\textsc{Step II} Now, define $(\Nc^*_{j} f)(x)\coloneqq \sup_{j'\Meg-M} \eps^{(j-j')_+S\grado+(j'-j)_+2 S\grado} \Nc_{\infty,S\grado,\eps^{j'}} (f*\psi_{j'})(x)$ for every $x\in G$ and for every $j\Meg -M$, and let us prove that  $\Nc^*_{j } f$ is finite  everywhere for every $j\Meg -M$ (for $f\in \Bc_{S\grado}(G)$). To this aim, observe that, since $f\in \Bc_{S\grado}(G)$, there is $C_2>0$ (depending on $f$) such that 
	\[
	\abs{\langle f, \phi\rangle}\meg C_2  \sup_{X\in U_{S\grado}, \abs{X}\meg 1} \norm{(1+\abs{\,\cdot\,}_*)^{S\grado}X \phi}_{L^1(G)}
	\]
	for every $\phi\in \Sr(G)$. Therefore, by Lemma~\ref{lem:37} there is a constant $C_3>0$ such that
	\[
	\begin{split}
		\abs{(f*\psi_{j'})(x)}&=\abs{\langle f, \psi_{j'}(\,\cdot\,^{-1}x)\rangle}\\
		&\meg C_2  \sup_{X\in U_{S\grado}, \abs{X}\meg 1} \norm{(1+\abs{\,\cdot\,}_*)^{S\grado}(X^R \psi_{j'})(\,\cdot\,^{-1}x)}_{L^1(G)}\\
		&\meg C_2 C_3 \eps^{-j'S\grado} (1+\abs{x}_*)^{S\grado}
	\end{split}
	\]
	for every $j'\Meg -M$ and for every $x\in G$, so that
	\[
	(\Nc^*_{j } f)(x)\meg C_2 C_3  \eps^{-j S\grado} (1+\abs{x}_*)^{S\grado},
	\]
	for every $j\Meg -M$ and for every $x\in G$.
	
	\textsc{Step III} We now conclude the proof for $p<1$. Observe first that, arguing as in~\textsc{Step I}, we see that there is a constant $C_4>0$ (depending on $S$) such that
	\[
	\Nc_{\infty,S\grado,\eps^j}(f*\psi_j) \meg C_4 \sum_{j'\Meg-M} \eps^{(j-j')_+S\grado+ (j'-j)_+2S\grado} \Nc_{1,S\grado,\eps^{j'}}(f*\psi_{j'}).
	\] 
	Using the fact that
	\[
	\eps^{(j-j')_+S\grado+ (j'-j)_+2S\grado}\eps^{(j'-j'')_+S\grado+ (j''-j')_+2S\grado}\meg \eps^{(j-j'')_+S\grado+ (j''-j)_+2S\grado}
	\]
	for every $j,j',j''\Meg -M$, we then see that
	\[
	\Nc^*_{j } f\meg C_4 \sum_{j'\Meg-M} \eps^{(j-j')_+S\grado+ (j'-j)_+2S\grado} \Nc_{1,S\grado,\eps^{j'}}(f*\psi_{j'})
	\]
	for every $j\Meg -M$. Consequently,
	\[
	\Nc_{j  }^* f\meg  C_{4} (\Nc_{j }^* f)^{1-p} \sum_{j'\Meg-M}  [\eps^{(j-j')_+S\grado+ (j'-j)_+2S\grado }\Nc_{p,S\grado,\eps^{j'}}(f*\psi_{j'})]^{p},
	\]
	whence
	\[
	(\Nc_{j }^* f)^{1/p} \meg C_4^{1/p}   (\Nc_{j }^* f)^{1/p-1} \Big(\sum_{j'\Meg-M}   [\eps^{(j-j')_+S\grado+ (j'-j)_+2S\grado }\Nc_{p,S\grado,\eps^{j'}}(f*\psi_{j'})]^{p}\Big)^{1/p}.
	\]
	The assertion follows if $N=S\grado$ since $\Nc_{j }^* f$ is everywhere finite for every $j\Meg -M$ by~\textsc{step II}.
	
	Now, assume that $N< S\grado$. Then,
	\[
	\begin{split}
		\abs{f*\psi_j}&\meg \Nc_{j }^* f\\
		&\meg C_4^{1/p} \Big(\sum_{j'\Meg-M}   [\eps^{(j-j')_+S\grado+ (j'-j)_+2S\grado }\Nc_{p,S\grado,\eps^{j'}}(f*\psi_{j'})]^{p}\Big)^{1/p}\\
		&\meg C_4^{1/p}\Big(\sum_{j'\Meg-M}   [\eps^{(j-j')_+N+ (j'-j)_+2N }\Nc_{p, N,\eps^{j'}}(f*\psi_{j'})]^{p}\Big)^{1/p}
	\end{split}
	\]
	for every $j\Meg -M$. Notice that we cannot conclude yet, in this case, since $C_4$ may depend on $S$. 
	Since
	\[
	(1+d(x,y)/\eps^j)^N\Meg \eps^{(j'-j)_+ N}(1+d(x,y)/\eps^{j'})^N
	\]
	for every $x,y\in G$ and for every $j,j'\Meg -M$, we then deduce that 
	\[
	\Nc_{\infty,N,\eps^j}(f*\psi_j) \meg C_4^{1/p}\Big(\sum_{j'\Meg-M}   [\eps^{\abs{j-j'}N  }\Nc_{p, N,\eps^{j'}}(f*\psi_{j'})]^{p}\Big)^{1/p}
	\]
	for every $j,j'\Meg -M$.
	Consequently, if we set $\Nc^{**}_{j} f\coloneqq \sup_{j'\Meg -M} \eps^{\abs{j-j'}N} \Nc_{\infty,N,\eps^{j'}}(f*\psi_{j'})$, then
	\[
	\Nc^{**}_{j} f\meg C_4^{1/p}\Big(\sum_{j'\Meg-M}   [\eps^{\abs{j-j'}N  }\Nc_{p, N,\eps^{j'}}(f*\psi_{j'})]^{p}\Big)^{1/p}.
	\]
	In particular, if $\sum_{j'\Meg-M}   [\eps^{\abs{j-j'}N  }\Nc_{p, N,\eps^{j'}}(f*\psi_{j'})(x)]^{p}$ is finite, then also $(\Nc^{**}_j f)(x)$ is finite. Arguing as before (but using the estimates of~\textsc{step I}) we then see that
	\[
	(\Nc^{**}_j f)(x)\meg C_1^{1/p} \Big(\sum_{j'\Meg-M}   [\eps^{\abs{j-j'}N  }\Nc_{p, N,\eps^{j'}}(f*\psi_{j'})(x)]^{p}\Big)^{1/p},
	\]
	whence the conclusion in this case. If, otherwise, $\sum_{j'\Meg-M}   [\eps^{\abs{j-j'}N  }\Nc_{p, N,\eps^{j'}}(f*\psi_{j'})(x)]^{p}=\infty$, then the assertion is trivial.
\end{proof}

\begin{lem}\label{lem:70}
	Take $\eps\in (0,1)$. Then,  
	\[
	\sum_{ j\in\Z}  \eps^{\abs{j-j'}+\abs{j-j''}} = \Big(\frac{2}{1-\eps^2}+\abs{j'-j''}\Big)\eps^{\abs{j'-j''}}
	\]
	for every $j',j''\in\Z$.
\end{lem}

\begin{proof}
	Assume first that $j'\meg j''$. Then,
	\[
	\begin{split}
		\sum_{ j\in\Z}  \eps^{\abs{j-j'}+\abs{j-j''}}&=\sum_{j\meg j'}\eps^{j'+j''-2j}+\sum_{j=j'+1}^{j''-1}\eps^{-j'+j''}+\sum_{j\meg j''}\eps^{2j-j'-j''}\\
			&=\frac{1}{1-\eps^2}\eps^{j'+j''-2j'}+(j''-j')\eps^{j''-j'}+\frac{1}{1-\eps^2}\eps^{2j''-j'-j''}\\
			&=\Big(\frac{2}{1-\eps^2}+j''-j'\Big) \eps^{j''-j'},
	\end{split}
	\]
	whence the result. The other case is treated similarly.
\end{proof}

\begin{prop}\label{lem:69}
	Take $N>0$, $p,q\in (0,\infty]$, $\alpha\in (-N,N)$, $\eps\in (0,1)$, and a family $(\psi_j)_{j\Meg -M}$ of elements of $\Sr(G)$. Then, there is a constant $C>1$ such that, if $(\psi_j)$ satisfies the $(\eps,S,\Lc)$-Calder\'on condition for some integer $S\Meg N/\grado$, then the following hold:
	\begin{enumerate}
		\item[\textnormal{(1)}] if $N>D/\min(1,p)$, then for every $f\in \Bc_{S\grado}(G)$ (for every $f\in \Sr'(G)$ if $p\Meg 1$)
		\[
		\frac{1}{C} \norm{f}_{B^{p,q}_\alpha(G)}\meg \norm{ \eps^{-j\alpha} (f*\psi_j)(x)  }_{L^{p,q}_{x,j}(G, \N-M)}\meg C \norm{f}_{B^{p,q}_\alpha(G)};
		\]
		
		\item[\textnormal{(2)}] if $p<\infty$ and $N>D/\min(1,p,q)$, then for every $f\in \Bc_{S\grado}(G)$ (for every $f\in \Sr'(G)$ if $p>1$ and $q\Meg 1$)
		\[
		\frac{1}{C} \norm{f}_{F^{p,q}_\alpha(G)}\meg \norm{ \eps^{-j\alpha} (f*\psi_j)(x)  }_{L^{q,p}_{j,x}( \N-M,G)}\meg C \norm{f}_{F^{p,q}_\alpha(G)};
		\]
		
		\item[\textnormal{(3)}] if  $N>D/\min(1,q)$, then for every $f\in \Bc_{S\grado}(G)$ (for every $f\in \Sr'(G)$ if $q\Meg 1$)
		\[
		\frac{1}{C} \norm{f}_{F^{\infty,q}_\alpha(G)}\meg \norm{ \eps^{-j\alpha} (f*\psi_{j-M})(x)  }_{\Cc^{q}_{(j,x)}(\eps)}\meg C \norm{f}_{F^{\infty,q}_\alpha(G)}.
		\]
	\end{enumerate}
\end{prop}

This result is an elaboration of~\cite[Theorem 4.2]{Jordi}. The additional restrictions on $f$ come from the analogous restrictions in Lemma~\ref{lem:65}. Notice, though, that $B^{p,q}_\alpha(G), F^{p,q}_\alpha(G)\subseteq \Bc_{S\grado}(G)$ if $ S\grado>Q_*/p-[\alpha/\dd]\dd$, so that the above result actually provides equivalent quasi-norms on $B^{p,q}_\alpha(G), F^{p,q}_\alpha(G)$ (and also characterizes the whole spaces when $p\Meg 1$ and when $p>1$ and $q\Meg 1$, respectively).

\begin{proof}
	Observe first that the right inequalities follow from~\cite[Lemmas 3.5 and 3.8]{Calzi4}. Then, take $(\psi'_j)\in \Psi_{\eps^{\grado}}$ and observe that
	\[
	f*\psi'_{j'}= \sum_{j\Meg -M} (f*\psi_j)*(\eta_j*\psi'_{j'})
	\]
	for every $j'\in \N$. In addition, by Lemma~\ref{cor:13}  there is a constant $C_1>0$ such that
	\[
	\norm{(1+\abs{\,\cdot\,}_*/\eps^{j})^{N}   (\eta_{j}*\psi'_{j'})}_{L^1(G)}\meg C_1\eps^{ \abs{j-j'} N}  
	\]
	for every $j,j'\Meg -M$, so that 
	\[
	\abs{  (f*\psi_j)*(\eta_j*\psi'_{j'}) }\meg C_1  \eps^{ \abs{j-j'} N} \Nc_{\infty, N, \eps^{j}}(f*\psi_j) 
	\] 
	Take $p_0\in (0,1]$ so that $N>D/p_0$ and so that: $p_0=\min(1,p)$ for (1); $p_0<p$, $p_0\meg q$, and $p_0=1$ if $p>1$ and $q\Meg 1$, for (2); $p_0=q$ for (3).  Therefore, by Lemma~\ref{lem:65}, there is a constant $C_2>0$ such that
	\[
	\begin{split}
		\abs{f*\psi'_{j'}} &\meg C_2\sum_{j\Meg -M} \eps^{\abs{j-j'}N} \Big(  \sum_{j''\Meg -M} [\eps^{\abs{j-j''}N  } \Nc_{p_0, N, \eps^{j''}}(f*\psi_{j''})]^{p_0} \Big)^{1/p_0}\\
			&\meg C_2 \Big( \sum_{j\Meg -M} \eps^{\abs{j-j'}N p_0} \sum_{j''\Meg -M} [\eps^{\abs{j-j''}N } \Nc_{p_0, N, \eps^{j''}}(f*\psi_{j''}) ]^{p_0}\Big)^{1/p_0}
	\end{split}
	\]
	for every $f$ as in the statement. Now, by Lemma~\ref{lem:70}  there is a constant $C_3>0$ such that
	\[
	\sum_{ j\Meg -M}  \eps^{\abs{j-j'}N p_0+\abs{j-j''}N p_0} \meg C_3 (\abs{j'-j''}+1)\eps^{\abs{j'-j''}N p_0}
	\]
	for every $j'\in \N$ and for every $j''\Meg -M$.
	Consequently, given $\kappa\in (0,N-\abs{\alpha})$, we may find a constant $C_4>0$ such that 
	\[
	\begin{split}
		\eps^{-j' \alpha }\abs{f*\psi'_{j'}} &\meg C_2 C_3^{1/p_0} \Big(   \sum_{j''\Meg -M} (\abs{j'-j''}+1)[\eps^{\abs{j'-j''}N  +(j''-j')\alpha  } \eps^{-j''\alpha } \Nc_{p_0, N, \eps^{j''}}(f*\psi_{j''})]^{p_0}  \Big)^{1/p_0}\\
			&\meg C_4 \Big(   \sum_{j''\Meg -M}  [\eps^{(j'-j'')_+( N-\alpha-\kappa)  +(j''-j')_+( N+\alpha-\kappa) }  \eps^{-j''\alpha  } \Nc_{p_0, N, \eps^{j''}}(f*\psi_{j''}) ]^{p_0} \Big)^{1/p_0}
	\end{split}
	\]
	for every $j'\in \N$ and for every $f$ as in the statement. For (1), it suffices to apply Young's inequality and then Lemma~\ref{lem:25d}. For (2), it suffices to apply Lemma~\ref{lem:25d} and then Lemma~\ref{lem:2c}. For (3), one may apply Lemma~\ref{lem:25c}. 
\end{proof}

\begin{prop}\label{prop:30c}
	Take $p,q\in (0,\infty]$ and $\alpha\in \R$. Take $\delta>0$, $R\Meg 2$, and a $(\delta,R)$-lattice $(x_j)_{j\in J}$ on $G$.  
	Take two bounded families $(\phi_j)_{j\in J}$ and $(\psi_j)_{j\in J}$ of elements of $C^\infty_c(G)$. 
	Then, the continuous linear mappings
	\[
	\Ic_{(\phi_j)}\colon \Dc'(G)\ni u\mapsto (\phi_j(x_j^{-1}\,\cdot\,) u)\in \Dc'(G)^J
	\]
	and
	\[
	\Rc_{(\psi_j)}\colon \Dc'(G)^J \ni  (u_j) \mapsto \sum_{j\in J} \psi_j(x_j^{-1}\,\cdot\,) u_j \in \Dc'(G)
	\]
	induce continuous linear mappings 
	\[
	\begin{aligned} 
		F^{p,q}_\alpha(G)&\to \ell^p(J;F^{p,q}_\alpha(G)), & \mathring F^{p,q}_\alpha(G)&\to \ell^p(J;\mathring F^{p,q}_\alpha(G))
	\end{aligned}
	\]
	and
	\[
	\begin{aligned}
		\ell^p(J;F^{p,q}_\alpha(G))&\to F^{p,q}_\alpha(G),& \ell^p(J;\mathring F^{p,q}_\alpha(G))&\to \mathring F^{p,q}_\alpha(G)
	\end{aligned}
	\] 
	respectively.
	If, in addition, $\sum_j (\phi_j\psi_j)(x_j^{-1}\,\cdot\,)=1$, then $\Rc_{(\psi_j)} \Ic_{(\phi_j)}=I$.
\end{prop}

\begin{proof}
	The last assertion is obvious. 	In order to simplify the notation, we set $\widetilde \phi_j\coloneqq \phi_j(x_j^{-1}\,\cdot\,)$ and $\widetilde \psi_j\coloneqq \psi_j(x_j^{-1}\,\cdot\,)$ for every $j\in \N$.  
	
	Take a strong $(\eps,S, \delta,\Lc)$-Calder\'on family $(\eta_j)_{j\Meg -M}$ for some $S>(\abs{\alpha}+Q_*/p)/\grado$ (cf.~Lemma~\ref{lem:64}), so that Proposition~\ref{lem:69} shows that  there is a constant $C_1>0$ such that
	\[
	\frac{1}{C_1}\norm{f}_{F^{p,q}_\alpha(G)}\meg  \norm*{ \eps^{-j\alpha}(f*\eta_j )(x)  }_{L^{q,p}_{j,x}(\N-M,G)}   \meg C_1 \norm{f}_{F^{p,q}_\alpha(G)}
	\]
	for every $f\in F^{p,q}_\alpha(G)$, if $p<\infty$, and
	\[
	\frac{1}{C_1}\norm{f}_{F^{p,q}_\alpha(G)}\meg \norm*{ \eps^{-j\alpha}(f*\eta_{j-M} )(x)  }_{\Cc^{q}_{(j,x)}(\eps)}   \meg C_1 \norm{f}_{F^{p,q}_\alpha(G)}
	\]
	for every $f\in F^{\infty,q}_\alpha(G)$. In addition, take $R' $  so large that the $\phi_j$ and the $\psi_j$ are supported in $B(e,R' )$, and observe that by~\cite[Lemma 4.3]{Calzi2} there is a partition $J_1,\dots, J_N$ of $J$ such that $d(x_j,x_{j'})\Meg 2(R'+\delta)$ for every $j,j'\in J_h$, $j\neq j'$, and for every $h=1,\dots, N$. 
	Then, observe that, if $p<\infty$,
	\[
	\begin{split}
		\norm*{ \norm{ \eps^{-k\alpha} ((\widetilde \phi_j f)*\eta_k )(x)  }_{L^{q,p}_{k,x}(\N-M,G)}   }_{\ell^p_j(J)}^{\min(1,p)}&\meg \sum_{h=1}^N \norm*{ \norm{ \eps^{-k\alpha} ((\widetilde \phi_j f)*\eta_k )(x)   }_{L^{q,p}_{k,x}(\N-M,G)}   }_{\ell^p_j(J_h)} ^{\min(1,p)}\\
		&= \sum_{h=1}^N \norm*{\norm*{ \norm{ \eps^{-k\alpha} ((\widetilde \phi_j f)*\eta_k )(x)  }_{\ell^q_k(\N-M)}}_{\ell^p_j(J_h)}  }_{L^p_x(G)}^{\min(1,p)}\\
		&=\sum_{h=1}^N\norm*{\norm*{ \norm{ \eps^{-k\alpha} ((\widetilde \phi_j f)*\eta_k )(x)  }_{\ell^q_k(\N-M)}}_{\ell^q_j(J_h)}  }_{L^p_x(G)}^{\min(1,p)}\\
		&=\sum_{h=1}^N\norm*{\norm*{ \norm{ \eps^{-k\alpha} ((\widetilde \phi_j f)*\eta_k )(x)  }_{\ell^q_j(J_h)} }_{\ell^q_k(\N-M)} }_{L^p_x(G)}^{\min(1,p)}\\
		&=\sum_{h=1}^N\norm*{\norm*{  \eps^{-k\alpha} \sum_{j\in J_h}((\widetilde \phi_j f)*\eta_k )(x)  }_{\ell^q_k(\N-M)} }_{L^p_x(G)}^{\min(1,p)}\\
		&\meg C_1\sum_{h=1}^N\norm*{  \sum_{j\in J_h}\widetilde \phi_j  f }_{F^{p,q}_\alpha (G)}^{\min(1,p)}
	\end{split}
	\]
	where the second and the fourth equality follow from the fact that, for every $x\in G$ and for every $h=1,\dots,N$, there is at most one $j\in J_h$ such that $ ((\widetilde \phi_j f)*\eta_k )(x)\neq0 $ for some $k\in\N-M$. The continuity of $\Ic_{(\phi_j)}\colon F^{p,q}_\alpha(G)\to \ell^p(J; F^{p,q}_\alpha(G))$ then follows from  Theorem~\ref{teo:6c} and~\cite[Proposition 9.2]{BCP} when $p<\infty$, since $\sum_{j\in J_h}\widetilde \phi_j $ clearly belongs to $W^{\infty,\infty}(G)$ for every $h=1,\dots, N$.  The case $p=\infty$ follows easily from Theorem~\ref{teo:6c}, since the $\widetilde \phi_j$ are uniformly bounded in $B^{\infty,\infty}_{\abs{\alpha}+1}(G)$.
	Since  $\Ic_{(\phi_j)}(C^\infty_c(G))\subseteq C^\infty_c(G)^{(J)}$, it is also clear that $\Ic_{(\phi_j)}$ maps $\mathring F^{p,q}_\alpha(G)$ into $ \ell^p(J; \mathring F^{p,q}_\alpha(G))$ continuously.  
	
	Next, take $(f_j)\in  \ell^p(J;   F^{p,q}_\alpha(G))$, set $f\coloneqq \sum_j \widetilde \psi_j f_j$, and observe that, if $p<\infty$,
	\[
	\begin{split}
		\norm*{ \norm{\eps^{-k\alpha}  (f*\eta_k )(x)  }_{\ell^q_k(\N-M)}  }_{L^p_x(G)}&\meg \norm*{\sum_j \norm{ \eps^{-k\alpha}  [(\widetilde \psi_j f_j)*\eta_k](x)  }_{\ell^q_k(\N-M)}  }_{L^p_x(G)}\\
		&\meg N^{1/p'}\norm*{\norm*{ \norm{ \eps^{-k\alpha}  [(\widetilde \psi_j f_j)*\eta_k ](x)  }_{\ell^q_k(\N-M)}  }_{\ell^p_j(J)}}_{L^p_x(G)}\\
		&= N^{1/p'}\norm*{\norm*{ \norm{ \eps^{-k\alpha}  [(\widetilde \psi_j f_j)*\eta_k ](x)  }_{\ell^q_k(\N-M)}  }_{L^p_x(G)}}_{\ell^p_j(J)}\\
		&\meg C_1 N^{1/p'}\norm*{\norm{  \widetilde \psi_j f_j  }_{F^{p,q}_\alpha(G)}}_{\ell^p_j(J)}.
	\end{split}
	\]
	The continuity of $\Rc_{(\psi_j)}\colon \ell^p(J; F^{p,q}_\alpha(G))\to  F^{p,q}_\alpha(G) $ then follows again from Theorem~\ref{teo:6c} and~\cite[Proposition 9.2]{BCP}, since the $\widetilde \psi_j$ are uniformly bounded in $W^{\infty,\infty}(G)$. 
	When $p=\infty$, one may observe that 
	\[
	\begin{split}
		&\eps^{-N'Q_*/q}\norm{ \chi_{[N',+\infty)\times B(x,\eps^{N'}))}(k,x') \eps^{-k\alpha} (f*\eta_{k-M})(x')  }_{L^q_{(k,x')}(\N\times G)} \\
		&\qquad\meg N^{1/q} \sup_{j\in J} \norm{ \eps^{-k\alpha} [(\widetilde \psi_j f_j)*\eta_{k-M}](x')}_{\Cc^q_{(k,x')}(\eps) }
	\end{split}
	\]
	for every $N'\in \N$ and for every $x\in G$.
	Since  $\Rc_{(\psi_j)}(C^\infty_c(G)^{(J)})\subseteq C^\infty_c(G)$, it is also clear that $\Rc_{(\psi_j)}$ maps $\ell^p(J; \mathring F^{p,q}_\alpha(G))$ into $  \mathring F^{p,q}_\alpha(G) $ continuously.
\end{proof}

\section{Pointwise Multipliers of Triebel--Lizorkin Spaces}

In this section we characterize the space of pointwise multipliers of the Triebel--Lizorkin spaces $F^{p,q}_\alpha(G)$ for $\alpha>Q_*/p$ (under some additional conditions). Since the spaces $F^{p,q}_\alpha(G)$ `localize' well (cf.~Proposition~\ref{prop:30c}), the corresponding space of pointwise multipliers turns out to be the space of tempered distributions which stay locally in $F^{p,q}_\alpha(G )$ with uniformly bounded norms.
We begin with a technical lemma which is necessary to define the space of such distributions.

\begin{lem}\label{lem:51c}
	Take $p,q\in (0,\infty]$ and $\alpha\in \R$. Take two   $\eta_1,\eta_2\in C^\infty_c(G)$ with $\eta_2\neq 0$. 
	Then, there is a constant $C>0$ such that 
	\[
	\norm{\eta_1(x\,\cdot\,) f}_{F^{p,q}_\alpha(G)}\meg C  \norm{f}_{F^{p,q}_\alpha(G)}
	\]
	for every $x\in G$ and for every $f\in \Sr'(G)$, and such that
	\[
	\sup_{x\in G}  \norm{\eta_1(x\,\cdot\,) f}_{F^{p,q}_\alpha(G)}\meg \sup_{x\in G}  \norm{\eta_2(x\,\cdot\,) f}_{F^{p,q}_\alpha(G)} 
	\]
	for every $f\in \Sr'(G)$.
\end{lem}

\begin{proof}  
	\textsc{Step I}	The first assertion follows easily from Theorem~\ref{teo:6c}, since the $\eta_1(x\,\cdot\,)$ are uniformly bounded in $C^\infty(G)\cap B^{\infty,\infty}_{\abs{\alpha} +1}(G)$.
	
	\textsc{Step II}  Observe that there are $c>0$, $x_0\in G$, and $r>0$ such that $\abs{\eta_2(x)}\Meg c$ for every $x\in B(x_0,r)$. In addition, there is a finite subset $J$ of $G$ such that $\supp{\eta_1}\subseteq J B(x_0,r)$. Consequently, there is a family $(\psi_j)_{j\in J}$ of elements of $C^\infty_c(G)$ such that $\eta_1=\sum_{j\in J} \psi_j \eta_2(j^{-1}\,\cdot\,)$. Then, by~\textsc{step I} we see that there is a constant $C_2>0$ such that
	\[
	\norm{\eta_1(x\,\cdot\,) f}_{F^{p,q}_\alpha(G)}\meg C_2\sum_{j\in J} \norm{\eta_2(j^{-1}x\,\cdot\,) f}_{F^{p,q}_\alpha(G)}\meg C_2 \card(J) \sup_{x'} \norm{\eta_2(x'\,\cdot\,) f}_{F^{p,q}_\alpha(G)}
	\]
	for every $x\in G$.  
	The assertion follows. 
\end{proof}

\begin{deff}
	Take $p,q\in (0,\infty]$ and $\alpha\in\R$. We define $F^{p,q}_{\alpha,\unif}(G)$ as the space of $f\in \Sr'(G)$ such that $\sup_{x\in G} \norm{\eta(x\,\cdot\,) f}_{F^{p,q}_\alpha(G)}$ is finite 
	for some non-zero $\eta\in C^\infty_c(G)$, endowed with the corresponding topology.
\end{deff}

\begin{deff}
	Take $p,q\in (0,\infty]$ and $\alpha\in \R$. We define $\Mc( F^{p,q}_\alpha(G))$ as the space of $u\in \Sr'(G)$ such that the mapping  $\Sr(G)\ni f\mapsto u f\in \Sr'(G)$ induces a continuous linear mapping $\mathring F^{p,q}_\alpha(G)\to F^{p,q}_\alpha(G)$, endowed with the corresponding topology.
\end{deff}

\begin{teo}\label{teo:14b}
	Take $p,q\in (0,\infty]$, with $p<\infty$,  $\alpha>Q_*/p,(1/\min(p,q)-1)_+Q_*$, and $\alpha'> (1/q-1)_+ Q_*$. Then, $\Mc(F^{p,q}_\alpha(G))= F^{p,q}_{\alpha,\unif}(G)$ and $\Mc(F^{\infty,q}_{\alpha'}(G))=F^{\infty,q}_{\alpha'}(G)$. In addition, the canonical   bilinear mapping\footnote{Notice that $F^{p,q}_\alpha(G)\subseteq L^\infty(G)$ by~\cite[Proposition 4.10]{Calzi4}, so that $F^{p,q}_{\alpha,\unif}(G)\subseteq L^\infty (G)$ and $u f$ may be defined unambiguously as an element of $L^\infty (G)$.}
	\[
	F^{p,q}_\alpha(G)\times F^{p,q}_{\alpha,\unif}(G)\ni (f,u)\mapsto u f\in F^{p,q}_\alpha(G)
	\] 
	is continuous.
\end{teo}

As a consequence of Proposition~\ref{prop:30c}, $F^{\infty,q}_{\alpha'}(G)=F^{\infty,q}_{\alpha',\unif}(G)$, so that we could have stated the above result without considering separately the case $p=\infty$, but the above formulation seemed more natural.

\begin{proof}
	Let us first prove that $\Mc(F^{\infty,q}_{\alpha'}(G))=F^{\infty,q}_{\alpha'}(G)$. The continuous inclusion $F^{\infty,q}_{\alpha'}(G)\subseteq \Mc(F^{\infty,q}_{\alpha'}(G)) $ follows from Corollary~\ref{cor:20b}. For the converse inclusion, take a positive $\phi\in C^\infty_c(G)$ with integral $1$ and supported in $B(e,1)$, and consider $\tau_k\coloneqq \chi_{B(e,k+1)}*\phi$, so that $\chi_{B(e,k)}\meg \tau_k\meg \chi_{B(e,k+2)}$ for every $k\in\N$. In addition, the $\tau_k$ are uniformly bounded in $W^{\infty,\infty}(G)$,  hence in $\mathring F^{\infty,q}_{\alpha'}(G)$. It then follows that, if $f\in \Mc(F^{\infty,q}_{\alpha'}(G))$, then the $f\tau_k$ are uniformly bounded in $F^{\infty,q}_{\alpha'}(G)$, so that their limit $f$ (in $\Sr'(G)$) belongs to $F^{\infty,q}_{\alpha'}(G)$. The (continuous) inclusion $\Mc(F^{\infty,q}_{\alpha'}(G))\subseteq F^{\infty,q}_{\alpha'}(G)$  follows. 
	
	Next, take a $(1,2)$-lattice $(x_j)_{j\in J}$ on $G$ and a bounded family $(\phi_j)_{j\in J}$ of elements of $C^\infty_c(G)$ such that $\sum_j \widetilde \phi_j=1$, where $\widetilde \phi_j=\phi_j(x_j^{-1}\,\cdot\,)$ for every $j\in J$. In addition, take $\psi\in C^\infty_c(G)$ such that $\psi=1$ on $B(e,R+1)$, where $R>0$ is chosen so that  the $\phi_j$ are supported in $B(e,R)$.
	Observe that, by Proposition~\ref{prop:30c} and Corollary~\ref{cor:20b}, there is a constant $C_1>1$ such that
	\[
	\begin{split}
		\norm{f g}_{F^{p,q}_\alpha(G)}&\meg C_1 \norm*{\norm{ \widetilde \phi_j f g}_{F^{p,q}_\alpha(G)}}_{\ell^p_j(J)}\\
		&= C_1 \norm*{\norm{ \widetilde \phi_j f \psi(x_j^{-1}\,\cdot\,) g}_{F^{p,q}_\alpha(G)}}_{\ell^p_j(J)}\\
		&\meg C_1^2\norm*{\norm{ \widetilde \phi_j f}_{F^{p,q}_\alpha(G)} \norm{\psi(x_j^{-1}\,\cdot\,)g}_{F^{p,q}_\alpha(G)}}_{\ell^p_j(J)}\\
		&\meg C_1^2 \norm*{\norm{ \widetilde \phi_j f}_{F^{p,q}_\alpha(G)}}_{\ell^p_j(J)}\sup_j \norm{\psi(x_j^{-1}\,\cdot\,)g}_{F^{p,q}_\alpha(G)}\\
		&\meg C_1^3 \norm{f}_{F^{p,q}_\alpha(G)} \norm{g}_{F^{p,q}_{\alpha,\unif}(G)},
	\end{split}
	\]
	whence the last assertion and the continuous inclusion $F^{p,q}_{\alpha,\unif}(G)\subseteq \Mc(F^{p,q}_\alpha(G))$. For what concerns the converse inclusion, take $f\in \mathring F^{p,q}_\alpha(G)$, $u\in \Mc(F^{p,q}_\alpha(G))$, and a non-zero $\phi\in C^\infty_c(G)$. Then,   for every $x\in G$,
	\[
	\begin{split}
		\norm{\phi(x\,\cdot\,) u  }_{F^{p,q}_\alpha(G)}& \meg  \norm{u}_{\Mc(F^{p,q}_\alpha(G))} \norm{\phi(x\,\cdot\,)  }_{F^{p,q}_\alpha(G)}\\
		&=\norm{u}_{\Mc(F^{p,q}_\alpha(G))}   \norm{ \phi   }_{F^{p,q}_\alpha(G)},
	\end{split}
	\]
	whence the conclusion.
\end{proof}

\section{Pointwise Multipliers of Besov Spaces}

The spaces of pointwise multipliers of Besov spaces $B^{p,q}_\alpha(G)$ look like their counterparts for Triebel--Lizorkin spaces only when $p\meg q$. When $p=\infty$, they are particularly simple to describe, since they coincide with the space $B^{\infty,q}_\alpha(G)$ itself (for $\alpha>0$). For $p<q$, we may only provide a slightly simplified description, showing that one may `get away' with testing the continuity of pointwise multiplication on a relatively simple subspace.

\begin{deff}
	Take $p,q\in (0,\infty]$ and $\alpha\in \R$. We define $\Mc(B^{p,q}_\alpha(G))$ as the space of $u\in \Sr'(G)$ such that the mapping $\Sr(G)\ni f \mapsto u f\in \Sr'(G)$ induces a continuous linear mapping $\mathring B^{p,q}_\alpha(G)\to B^{p,q}_\alpha(G)$, endowed with the corresponding topology.
\end{deff} 

\begin{prop}\label{prop:21c}
	Take $q\in (0,\infty]$ and $\alpha >0$. Then, $\Mc(B^{\infty,q}_\alpha(G))=B^{\infty,q}_\alpha(G)$.
\end{prop}
 
 The proof essentially repeats the first part of the proof of Theorem~\ref{teo:14b} and is omitted.

\begin{lem}\label{lem:51d}
	Take  $p,q\in (0,\infty]$ and $\alpha\in\R$. Take two   $\eta_1,\eta_2\in C^\infty_c(G)$ with $\eta_2\neq 0$. Then, there is a constant $C>0$ such that 
	\[
	\norm{\eta_1(x\,\cdot\,) f}_{B^{p,q}_\alpha(G)}\meg C  \norm{f}_{B^{p,q}_\alpha(G)}
	\]
	for every $x\in G$ and for every $f\in \Sr'(G)$, and such that
	\[
	\sup_{x\in G}\norm{\eta_1(x\,\cdot\,) f}_{B^{p,q}_\alpha(G)}\meg \sup_{x\in G} \norm{\eta_2(x\,\cdot\,) f}_{B^{p,q}_\alpha(G)}
	\]
	for every $f\in \Sr'(G)$.
\end{lem}

The proof essentially repeats that of Lemma~\ref{lem:51c} and is omitted.

\begin{deff}
	Take  $p,q\in (0,\infty]$ and $\alpha\in\R$.  We define $B^{p,q}_{\alpha,\unif}(G)$ as the space of $f\in \Sr'(G)$ such that $\sup_{x\in G} \norm{\eta(x\,\cdot\,) f}_{B^{p,q}_\alpha(G)}$ is finite 
	for some non-zero $\eta\in C^\infty_c(G)$, endowed with the corresponding topology.
\end{deff}

\begin{lem}\label{lem:53c}
	Take $p,q\in (0,\infty]$ and $\alpha\in\R$.  Then, $\Mc(B^{p,q}_\alpha(G))\subseteq B^{p,q}_{\alpha,\unif}(G)$ continuously. 
	If, in addition, $\alpha\Meg Q_*/p$ and either $q\meg \min(1,p)$ or $\alpha>Q_*/p$, then $B^{p,q}_{\alpha,\unif}(G)\subseteq L^\infty(G)$ continuously.
\end{lem}

\begin{proof}
	Take  $u\in \Mc(B^{p,q}_\alpha(G))$ and a non-zero $\phi\in C^\infty_c(G)$. Then,  for every $x\in G$,
	\[
	\norm{\phi(x\,\cdot\,) u  }_{B^{p,q}_\alpha(G)}\meg \norm{u}_{\Mc(B^{p,q}_\alpha(G))} \norm{\phi(x\,\cdot\,)  }_{B^{p,q}_\alpha(G)}=\norm{u}_{\Mc(B^{p,q}_\alpha(G))}   \norm{ \phi   }_{B^{p,q}_\alpha(G )},
	\]
	whence the continuous inclusion $\Mc(B^{p,q}_\alpha(G))\subseteq B^{p,q}_{\alpha,\unif}(G)$.
	
	Next, assume that $\alpha\Meg Q_*/p$ and either $q\meg \min(1,p)$ or $\alpha>Q_*/p$. Take $f\in  B^{p,q}_{\alpha ,\unif}(G)$ and  assume that $\phi=1$ on $B(e,1)$. Then, by~\cite[Proposition 4.10]{Calzi4}, there is a constant $C>0$ such that
	\[
	\norm{f}_{L^\infty(G)}\meg \sup_{x\in G} \norm{\phi(x\,\cdot\,) f}_{L^\infty(G)}\meg C \sup_{x\in G} \norm{\phi(x\,\cdot\,) f}_{B^{p,q}_{\alpha}(G)}.
	\]
	The assertion follows.
\end{proof}

\begin{teo}\label{teo:20c}
	Take $p,q\in (0,\infty]$ and $\alpha \Meg Q_*/p$ such that $p\meg q$ and either $p=q\meg 1$ or $\alpha>Q_*/p$. Then, $\Mc(B^{p,q}_\alpha(G))=B^{p,q}_{\alpha,\unif}(G)$. 
	In addition, the canonical   bilinear mapping 
	\[
	B^{p,q}_\alpha(G)\times B^{p,q}_{\alpha,\unif}(G)\ni (f,u)\mapsto u f\in B^{p,q}_\alpha(G)
	\] 
	is continuous.
\end{teo}

\begin{proof}
	The case $p=q$ may be treated as in the proof of Theorem~\ref{teo:14b}.
	We may therefore assume that $p<q$, so that $\alpha>Q_*/p$. Since the inclusion $\Mc(B^{p,q}_\alpha(G))\subseteq B^{p,q}_{\alpha,\unif}(G)$ follows from Lemma~\ref{lem:53c}, we may reduce to proving the continuity of $B^{p,q}_\alpha(G)\times B^{p,q}_{\alpha,\unif}(G) \ni (f,g)\mapsto fg\in B^{p,q}_\alpha(G)$. 
	We proceed as in the proof of Theorem~\ref{teo:6c}.
	
	Take a reduced $(\eps,1/4,2)$-lattice $(x_{j,k})_{(j,k)\in J}$ on $G$, and choose, for every $j\in\N$, a Borel partition $(B_{j,k})_{k }$ of $G$ such that $B(x_{j,k}, \eps^j/4)\subseteq B_{j,k}\subseteq B(x_{j,k}, \eps^j/2)$ for every $k \in \N$ with $(j,k)\in J$. 
	Take $(\psi''_j)\in \widetilde\Psi_{\eps^{\grado}}$, $\eta\in C^\infty_c(G)$ such that $\chi_{B(e,1/8)}\meg \eta \meg \chi_{B(e,1/4)}$,   and set $\psi_j\coloneqq \eta \psi''_j$ and $\psi'_j\coloneqq \sum_{k=0}^{j-1}\eta\psi''_k$ for every $j\Meg 0$. Thus, arguing as in the proof of Lemma~\ref{lem:28c} we see that
	\[
	f g=\lim_{j'\to \infty}(\Pi_f^{(j')} g+\Pi^{(j')}_g f)+\Pi(f,g)+\widetilde\Pi(f,g),
	\]
	where
	\[
	\Pi_f ^{(j')}g\coloneqq\sum_{1\meg j\meg j'} [(f*\psi_j)(g*\psi_j')]*\psi'_j, \qquad \Pi_g ^{(j')}f=\sum_{1\meg j\meg j'} [(f*\psi'_j)(g*\psi_j)]*\psi'_j,
	\]
	for every $j'\Meg 1$, and where
	\[
	\Pi(f,g)\coloneqq \sum_{j\Meg 0} [(f*\psi'_{j+1})(g*\psi_{j+1}')]*\psi_j \qquad \text{and} \qquad \widetilde \Pi(f,g)\coloneqq \sum_{j\Meg 1} [(f*\psi_{j})(g*\psi_{j})]*\psi'_j.
	\]
	In addition, we take a bounded family $(\widetilde \phi_j)_{j\in\N}$ of elements of $C^\infty_c(G)$ supported in $B(e,1)$ such that $\sum_{j\in \N} \phi_j=1$ locally uniformly, where $\phi_j\coloneqq \widetilde \phi_j(x_{j,0}^{-1}\,\cdot\,)$ for every $j\in\N$, and a family $(\eta_j)_{j\Meg -M}$ satisfying the strong $(\eps,S,1,\Lc)$-Calder\'on condition for some $M\in\N$ and some $S>\max(D/\min(1,p),\alpha)/\grado$ (cf.~Lemma~\ref{lem:64}). 
	Furthermore, take $\phi'\in C^\infty_c(G)$ so that $\chi_{B(e, 4)}\meg \phi' \meg \chi_{B(e,5)}$, and set $\phi'_j\coloneqq \phi'(x_{j,0}^{-1}\,\cdot\,)$.
	Observe that, since   $B^{p,q}_\alpha(G), B^{p,q}_{\alpha,\unif}(G)\subseteq L^\infty(G)$ by~\cite[Proposition 4.10]{Calzi4} and Lemma~\ref{lem:53c}, we have $fg\in L^\infty(G)\subseteq \Bc_{S\grado}(G)$, so that Proposition~\ref{lem:69} shows that it will suffice to show that there is a constant $C_1>0$ such that
	\[
	\norm{\eps^{-j\alpha} [(fg)*\eta_j](x)}_{L^{p,q}_{x,j}(G,\N-M)}\meg C_1 \norm{f}_{B^{p,q}_\alpha(G)}\norm{g}_{B^{p,q}_{\alpha,\unif}(G)}.
	\]

	Define $\widetilde f, \widetilde g, \widetilde f_{j'}, \widetilde g_{j'}\colon \N\times G\to [0,\infty)$ ($j'\in\N$) so that
	\[
	\widetilde f (j,x)\coloneqq \max_{\overline B(x_{j,k}, \eps^j/2)} \abs{f*\psi_j} \chi_{B_{j,k}}(x) \qquad \text{and} \qquad \widetilde g (j,x)\coloneqq \max_{\overline B(x_{j,k}, \eps^j/2)} \abs{g*\psi'_j} \chi_{B_{j,k}}(x)
	\]
	and also
	\[
	\widetilde f_{j'} (j,x)\coloneqq \max_{\overline B(x_{j,k}, \eps^j/2)} \abs{(\phi_{j'} f)*\psi_j} \chi_{B_{j,k}}(x) \qquad \text{and} \qquad \widetilde g_{j'} (j,x)\coloneqq \max_{\overline B(x_{j,k}, \eps^j/2)} \abs{(\phi'_{j'}g)*\psi'_j} \chi_{B_{j,k}}(x)
	\]
	for every $(j,x)\in \N\times G$, so that $\abs{f*\psi_j}\meg \widetilde f(j,\,\cdot\,)$ and $\abs{g*\psi'_j}\meg \widetilde g(j,\,\cdot\,)$ and, analogously, $\abs{(\phi_{j'}f)*\psi_j}\meg \widetilde f_{j'}(j,\,\cdot\,)$ and $\abs{(\phi'_{j'}g)*\psi'_j}\meg \widetilde g_{j'}(j,\,\cdot\,)$ for every $j,j'\in\N$. 
	Take $p_0\in (0,\min(1,p,q))$, $a>D/p_0$,  and $m\in \N $ with $m\grado>2 \alpha$, and observe that, by Lemma~\ref{cor:13}, there is a constant $C_1>0$ such that
	\[
	\abs{(\psi'_j*\eta_{j'})(x)}\meg \frac{C_1 \eps^{-Q_* \min(j,j')}}{(1+\abs{x}_*/\eps^{j'})^a} \min(1,\eps^{(j'-j)m\grado}) =\frac{C_1\eps^{-Q_* \min(j,j')}}{(1+\abs{x}_*/\eps^{j'})^a}  \eps^{(j'-j)_+ m\grado}
	\]
	for every $x\in G$, for every $j\in \N$, and for every $j'\Meg -M$.
	Consequently, arguing as in the proof of Theorem~\ref{teo:6c},
	\[
	\begin{split}
		\abs{ [(f*\psi_j)(g*\psi'_j)]*\psi'_j*\eta_{j'} }&\meg C_1\eps^{(j'-j)_+ (m \grado+Q_*)} \Nc_{1,a,\eps^{j'}}((f*\psi_j)(g*\psi'_j))\\
		&\meg C_1\eps^{(j'-j)_+ (m \grado+Q_*)} \Nc_{1,a,\eps^{j'}}(\widetilde f(j,\,\cdot\,)\widetilde g(j,\,\cdot\,))\\
		&\meg 4^a C_1''\eps^{(j'-j)_+ m \grado -(j-j')_+(1/p_0-1)Q_*} \Nc_{p_0,a p_0}(\widetilde f(j,\,\cdot\,)\widetilde g(j,\,\cdot\,)),
	\end{split}
	\]
	with $C_1''$ defined as in the proof of Theorem~\ref{teo:6c}.
	Then, by means of Proposition~\ref{cor:18} and Lemma~\ref{lem:25d}, we see that there is a constant $C_2>1$ such that, for every $j''\Meg 1$,
	\[
	\begin{split}
		&\norm{\eps^{-j'\alpha}[(\Pi_f^{(j'')} g)*\eta_{j'}](x)}_{L^{p,q}_{x,j'}(G,\N-M)}\\
		&\qquad\meg C_1'' \norm*{ \eps^{-j'\alpha}\sum_{j\in\N} \eps^{(j'-j)_+ m \grado -(j-j')_+(1/p_0-1)Q_*}\Nc_{p_0,a p_0}(\widetilde f(j,\,\cdot\,)\widetilde g(j,\,\cdot\,))(x)   }_{L^{p,q}_{x,j'}(G,\N-M)}\\
		&\qquad\meg  C_1'' \norm*{ \eps^{-j'\alpha} \Nc_{p_0,a p_0}\Big(\sum_{j\in\N}\eps^{(j'-j)_+ m \grado -(j-j')_+(1/p_0-1)Q_*}\widetilde f(j,\,\cdot\,)\widetilde g(j,\,\cdot\,)\Big)(x)   }_{L^{p,q}_{x,j'}(G,\N-M)}\\
		&\qquad\meg C_2\norm*{  \sum_{j\in\N}\eps^{(j'-j)_+ (m \grado-\alpha) +(j-j')_+(\alpha-(1/p_0-1)Q_*) }\eps^{-j\alpha} \widetilde f(j,x)\widetilde g(j,x)   }_{L^{p,q}_{x,j'}(G,\N-M)}\\
		&\qquad\meg C_2^2 \norm*{  \eps^{-j\alpha } \widetilde f(j,x) \widetilde g(j,x)}_{L^{p,q }_{x,j}(G,\N)}.
	\end{split}
	\]
	Now take, by~\cite[Lemma 4.3]{Calzi2}, a partition $(J_1,\dots, J_{N'})$ of $\N$ so that $d(x_{j,0},x_{j',0})\Meg  8$ for every two distinct $j,j'\in J_h$, $h=1,\dots, N'$. Then,
	\[
	\begin{split}
		\norm*{  \eps^{-j\alpha } \widetilde f(j,x) \widetilde g(j,x)}_{L^{p,q }_{x,j}(G,\N)}&\meg\norm*{\sum_{j'\in \N}  \eps^{-j\alpha } \widetilde f_{j'}(j,x) \widetilde g(j,x)}_{L^{p,q }_{x,j}(G,\N)} \\
			&\meg N'^{1/\min(1,p,q)-1}\sum_{h=1}^{N'} \norm*{\sum_{j'\in J_h}  \eps^{-j\alpha } \widetilde f_{j'}(j,x) \widetilde g_{j'}(j,x)}_{L^{p,q }_{x,j}(G,\N)}\\
			&\meg N'^{1/\min(1,p,q)-1}\sum_{h=1}^{N'}\norm*{ \sum_{j'\in J_h}  \eps^{-j\alpha } \widetilde f_{j'}(j,x) }_{L^{p,q }_{x,j}(G,\N)} \sup_{j'\in \N} \norm{\widetilde g_{j'}}_{L^\infty(G\times \N)},
	\end{split}
	\]
	where the second inequality follows from the fact that $\widetilde f_{j'}(j,\,\cdot\,)$ is supported in $B(x_{j',0},9/4)$, on which $\widetilde g(j,\,\cdot\,)=\widetilde g_{j'}(j,\,\cdot\,)$.
	Now, let  $U$ be the set of $(z_{j,k})\in G^J$ such that $\abs{z_{j,k}}_*\meg \eps^j/2$ for every $(j,k)\in J$, and observe that
	\[
	[(\phi_{j'} f)*\psi_j](x_{j,k}z_{j,k})=\langle \phi_{j'}f \vert \psi_j^*(z_{j,k}^{-1}x_{j,k}^{-1}\,\cdot\,)  \rangle
	\]
	for every $j'\in \N$ and for every $(j,k)\in J$, and that there is a constant $c>0$ such that $(c\eps^{jQ_*}\psi_j^*(z_{j,k}^{-1}x_{j,k}^{-1}\,\cdot\,))$ is a system of $(K,S,N,\infty,\Lc)$-atoms for $K>(Q_*(1/p-1)_+-\alpha)/\grado$, $S>(Q_*+\alpha)/\grado$ and $N>D/\min(1,p)$ and for every $(z_{j,k})\in U$, thanks to Proposition~\ref{cor:12}. Therefore, by means of~\cite[Proposition 5.9]{Calzi4} we see that there is a constant $C_3>0$ such that
	\[
	\begin{split}
	\norm*{ \sum_{j'\in J_h}  \eps^{-j\alpha } \widetilde f_{j'}(j,x) }_{L^{p,q }_{x,j}(G,\N)}& =\sup_{(z_{j,k})\in U} \norm*{\bigg \langle \sum_{j'\in J_h}  \eps^{-j\alpha }\phi_{j'}f \bigg\vert \psi_j^*(z_{j,k}^{-1}x_{j,k}^{-1}\,\cdot\,)  \bigg\rangle }_{L^{p,q }_{x,j}(G,\N)} \\
		&\meg  C_3\norm*{\sum_{j'\in J_h}\phi_{j'} f }_{B^{p,q  }_\alpha(G)},
	\end{split}
	\]
	where the first equality follows from the fact that that
	\[
	\sum_{j'\in J_h} \widetilde f_{j'}(j,x)= \sum_{(j,k)\in J}\max_{B(x_{j,k},\eps^j/2)} \abs*{\sum_{j'\in J_h} \phi_{j'} f*\psi_j  } \chi_{B_{j,k}}(x) 
	\]
	since the $\widetilde f_{j'}(j,\,\cdot\,)$, $j'\in J_h$, have pairwise disjoint supports.

	Then, by  Proposition~\ref{cor:12} and Young's inequality,  Theorem~\ref{teo:6c}, and Lemma~\ref{lem:53c}, we may find a constant $C_3'>1$ such that
	\[
	\begin{split} 
	\norm*{  \eps^{-j\alpha } 
	\widetilde f(j,x) \widetilde g(j,x)}_{L^{q,p }_{j,x}(\N,G)}&\meg C_3'\sum_{h=1}^{N'} \norm*{\sum_{j'\in J_h}\phi_{j'} f }_{B^{p,q  }_\alpha(G)} \norm*{g }_{L^{\infty}(G\times N)}\\
	&\meg C_3'^2\norm{f}_{B^{p,q}_\alpha(G)} \norm{g}_{B^{p,q}_{\alpha,\unif}(G)}.
	\end{split}
	\]
	The term $\widetilde \Pi(f,g)$ may be estimated in a similar way. We then pass to the term $\Pi_g^{(j'')} f$. In order to avoid introducing new notation, we shall estimate the term $\Pi_f ^{(j'')}g$ with $f\in B^{p,q}_{\alpha,\unif}(G)$ and $g\in B^{p,q}_\alpha(G)$ instead.
	
	Then, observe that, by Lemma~\ref{lem:37}, Young's inequality, and Proposition~\ref{prop:14}, there is a constant $C_4>1$ such that (arguing as above)
	\[
	\begin{split} 
		\norm*{  \eps^{-j\alpha } \widetilde f(j,x) \widetilde g(j,x)}_{L^{p,q }_{x,j}(G,\N)}&\meg\norm*{\sum_{j'\in \N}  \eps^{-j\alpha } \widetilde f_{j'}(j,x) \widetilde g(j,x)}_{L^{p,q }_{x,j}(G,\N)} \\ 
		&\meg N'^{1/\min(1,p,q)-1} \sum_{h=1}^{N'} \norm*{\norm*{\sum_{j'\in J_h}  \eps^{-j\alpha } \widetilde f_{j'}(j,x) \widetilde g_{j'}(j,x)}_{L^{p }_{x}(G)}}_{\ell^q_j(\N)}\\
		&\meg N'^{1/\min(1,p,q)-1}\sum_{h=1}^{N'} \norm*{\norm*{ \eps^{-j\alpha }   \norm{\widetilde f_{j'}(j,x) \widetilde g_{j'}(j,x)}_{L^{p }_{x}(G)}}_{\ell^p_{j'}(J_h)}  }_{\ell^q_j(\N)}\\
		&\meg N'^{1/\min(1,p,q)-1}\sum_{h=1}^{N'} \norm*{\norm*{\eps^{-j\alpha }     \norm{\widetilde f_{j'}(j,x)}_{L^\infty_x(G)}\norm{\widetilde g_{j'}(j,x)}_{L^{p }_{x}(G)}}_{\ell^p_{j'}(J_h)}  }_{\ell^q_j(\N)}\\
		&\meg C_4 \sum_{h=1}^{N'} \norm*{\norm*{     \norm{ \phi_{j'}f}_{L^\infty (G)}\eps^{-j\alpha }\norm{\widetilde g_{j'}(j,x)}_{L^{p }_{x}(G)}}_{\ell^p_{j'}(J_h)}  }_{\ell^q_j(\N)}\\
		&\meg C_4 \sum_{h=1}^{N'}  \norm*{  \norm{ \phi_{j'}f}_{L^\infty (G)} \norm*{ \eps^{-j\alpha }\norm{\widetilde g_{j'}(j,x)}_{L^{p }_{x}(G)}  }_{\ell^q_j(\N)}}_{\ell^p_{j'}(J_h)}\\
		&\meg C_4^2 \sum_{h=1}^{N'}  \norm*{  \norm{ \phi_{j'}f}_{L^\infty (G)}    }_{\ell^p_{j'}(J_h)} \norm{g}_{B^{p,q}_{\alpha,\unif}(G)}
	\end{split}
	\]
	where the sixth inequality follows from the fact that $p\meg q$. Then, observe that, by~\cite[Propositions 4.10]{Calzi3} and~\ref{prop:30c}  there is a constant $C_5>1$ such that
	\[
	\begin{split}
		 \norm*{  \norm{ \phi_{j'}f}_{L^\infty (G)}    }_{\ell^p_{j'}(J_h)}&\meg C_5 \norm*{  \norm{ \phi_{j'}f}_{B^{p,p}_{Q_*/p+\kappa} (G)}    }_{\ell^p_{j'}(J_h)}\\
		 	&\meg C_5^2    \norm{ f}_{B^{p,p}_{Q_*/p+\kappa} (G)}     \\
		 	&\meg C_5^2    \norm{ f}_{B^{p,q}_{\alpha} (G)}
	\end{split}
	\]
	for any fixed $\kappa\in (0,\alpha-Q_*/p)$ and for every $h=1,\dots, N'$.

	Finally, the term $\Pi(f,g)$ may be dealt with following the arguments of~\cite[Theorem 6.6]{Calzi2}, with modifications analogous to the ones used above.
\end{proof}

We now pass to the case $p<q$. We begin with a lemma that proves the consistency of the definition of the spaces $M^{p,q}_\alpha$ introduced below.

\begin{lem}\label{lem:57c}
	Take $p,q\in (0,\infty]$ and $\alpha>Q_*(1/p-1)_+$.
	In addition, take $\delta>0$, $R\Meg 2$, a $(\delta,R)$-lattice  $(x_j)_{j\in J}$  on $G$, and two bounded families $(\phi_j)_{j\in J}$ and $(\psi_{j'})_{j'\in J'}$ in $C^\infty_c(G)$. Set $\widetilde \phi_j\coloneqq \phi_j(x_j^{-1}\,\cdot\,)$   for every $j\in J$, and assume that $\sum_j  \widetilde \phi_j=1$. Then, there is a constant $C>0$ such that 
	\[
	\norm*{ \sum_{j'} a'_{j'} \widetilde \psi_{j'} f }_{B^{p,q}_\alpha(G)}\meg C \norm{a'_{j'}}_{\ell^p_{j'}(J')} \sup_{\norm{a_j}_{\ell^p_j(J)}\meg 1} \norm*{  \sum_j a_j \widetilde \phi_j f }_{B^{p,q}_\alpha(G)}
	\]
	for every $(\delta,R)$-lattice  $(x'_{j'})_{j'\in J'}$ on $G$, for every $f\in \Sr'(G)$, and for every $(a'_{j'})\in \ell^p(J')$, where  $\widetilde \psi_{j'}\coloneqq \psi_{j'}(x'^{-1}_{j'}\,\cdot\,)$  for every $j'\in J'$.
\end{lem}

\begin{proof}
	Observe that, by~\cite[Lemma 4.3]{Calzi2}, there is $N\in\N$ such that $\card(J_{j'})\meg N$ for every $j'\in J'$, where $J_{j'}=\Set{j\in J\colon \widetilde \psi_{j'} \widetilde \phi_j\neq 0}$. In addition, by the same reference, we may partition $J'$ in subsets $J'_1,\dots, J'_N$ such that the $J_{j'}$, $j'\in J'_h$, are pairwise disjoint for every $h=1,\dots, N$. Then, take $(a'_{j'})\in \ell^p(J')$ and set $a^{(h)}_j\coloneqq \sum_{j'\in J'_h} \chi_{J_{j'}}(j) a'_{j'}$, so that
	\[
	\begin{split}
		\sum_{j'} a'_{j'} \widetilde \psi_{j'}&=\sum_{h=1}^N \sum_{j'\in J'_h}\sum_{j\in J_{j'}} a'_{j'}\widetilde \psi_{j'}  \widetilde \phi_{j}\\
		&= \sum_{h=1}^N \sum_{j'\in J'_h}\sum_{j\in J_{j'}} \widetilde \psi_{j'}  a^{(h)}_j\widetilde \phi_{j}\\
		&=  \sum_{h=1}^N \sum_{j'\in J'_h}  \widetilde \psi_{j'}  \sum_{j\in J} a^{(h)}_j\widetilde \phi_{j}.
	\end{split}
	\]
	Consequently, by Theorem~\ref{teo:6c} there is a constant $C_1>0$ such that
	\[
	\begin{split} 
		\norm*{ \sum_{j'} a'_{j'} \widetilde \psi_{j'} f }_{B^{p,q}_\alpha(G)}&\meg N^{1/\min(1,p,q)-1} \sum_{h=1}^N \norm*{  \sum_{j'\in J'_h}  \widetilde \psi_{j'}  \sum_{j\in J} a^{(h)}_j\widetilde \phi_{j} f }_{B^{p,q}_\alpha(G)} \\
		&\meg C_1 \sum_{h=1}^N \norm*{  \sum_{j'\in J'_h}  \widetilde \psi_{j'}  }_{B^{\infty,\infty}_{\alpha+1}(G )}\norm*{  \sum_{j\in J} a^{(h)}_j\widetilde \phi_{j} f }_{B^{p,q}_\alpha(G)}.
	\end{split}
	\]
	Since clearly $ \sum_{j'\in J'_h}  \widetilde \psi_{j'} \in W^{\infty,\infty}(G )\subseteq B^{\infty,\infty}_{\alpha+1}(G)$ (cf.~\cite[Proposition 9.2]{BCP}) and   $\norm{a^{(h)}_j}_{\ell^p_j(J)}\meg N^{1/p} \norm{a'_{j'}}_{\ell^p_{j'}(J')}$ for every $h=1,\dots,N$,  the assertion follows.
\end{proof}

\begin{deff}\label{def:4c}
	Take $p,q\in (0,\infty]$ and $\alpha >Q_*(1/p-1)_+$. We define $M^{p,q}_\alpha$ as the space of $f\in \Sr'(G)$ such that
	\[
	\norm{f}_{M^{p,q}_\alpha}\coloneqq \sup_{\norm{a_j}_{\ell^p_j(J)}\meg 1} \norm*{ \sum_j a_j \widetilde \phi_j f }_{B^{p,q}_\alpha(G_L)}<\infty,
	\]
	endowed with the corresponding topology, where the $\widetilde \phi_j$ are as in Lemma~\ref{lem:57c}.
\end{deff}

Observe that, if $N\coloneqq\max \sum_j \chi_{B(x_j,  (R+2)\delta)}$, then we may take $(\phi_j)$ such that $\phi_j\Meg 1/N$ on $B(e,\delta)$ for every $j\in J$.\footnote{For example, one may take $\psi\in C^\infty_c(G)$ so that $\chi_{B(e,R\delta)}\meg\psi\meg \chi_{B(e,(R+1)\delta)}$, so that $\Psi\coloneqq \sum_j \psi(x_j^{-1}\,\cdot\,)$ takes values in $[1,N]$, and then define $\phi_j\coloneqq \psi/\Psi(x_j\,\cdot\,)$ for every $j\in J$. } In this case, it is readily seen that there is a constant $c>1$ such that\footnote{The first inequality follows easily from the inclusion $B^{p,q}_\alpha(G)\subseteq L^p(G)$, cf.~\cite[Proposition 4.10]{Calzi4}, while the second inequality follows from Theorem~\ref{teo:13bis}.}  
\[
\frac 1 c\norm{a_j}_{\ell^p_j(J)}\meg \norm{a_j \widetilde \phi_j}_{B^{p,q}_\alpha(G )} \meg c\norm{a_j}_{\ell^p_j(J)}.
\]
Consequently, the space $M^{p,q}_\alpha$ is defined essentially in the same way as $\Mc(B^{p,q}_\alpha(G ))$, except for the fact that, instead of testing the continuity of pointwise multiplication on the whole of $B^{p,q}_\alpha(G )$, we restrict to a particularly simple closed vector subspace (namely, the one generated by the $\widetilde \phi_j$).

\begin{lem}\label{lem:58c}
	Take $p,q\in (0,\infty]$ and $\alpha >Q_*(1/p-1)_+$. Then, $\Mc(B^{p,q}_\alpha(G)) \subseteq M^{p,q}_\alpha\subseteq B^{p,q}_{\alpha,\unif}(G)$ continuously.
\end{lem}

\begin{proof}
	The continuity of the inclusion $M^{p,q}_\alpha\subseteq B^{p,q}_{\alpha,\unif}(G)$ follows from Lemma~\ref{lem:57c}, choosing $J'=J$, $x'_j\coloneqq x x_j$, $a'_j=\chi_{\Set{j_0}}(j)$ for some $j_0\in J$, and any non-zero $\psi_{j_0}$.  
	For the other inclusion,   it will suffice to show that the mapping
	\[
	\ell^p(J)\ni(a_j)\mapsto \sum_j a_j \widetilde \phi_j \in B^{p,q}_\alpha(G)
	\]
	is continuous. This follows from Theorem~\ref{teo:13bis}.
\end{proof}

\begin{teo}\label{teo:21c}
	Take $p,q\in (0,\infty]$ and take $\alpha\Meg Q_*/p$ such that either $q\meg \min(1,p)$ or $\alpha>Q_*/p$. 
	Then, $\Mc(B^{p,q}_\alpha(G))=M^{p,q}_\alpha$. In addition, the canonical bilinear mapping
	\[
	B^{p,q}_\alpha(G)\times M^{p,q}_\alpha\ni (f,g)\mapsto f g \in B^{p,q}_\alpha(G)
	\]
	is continuous.
\end{teo}

Observe that this result, combined with Theorem~\ref{teo:20c} and Proposition~\ref{prop:21c}, shows that $M^{p,q}_\alpha=B^{p,q}_{\alpha,\unif}(G)$ when $p\meg q$, and that $M^{\infty,q}_\alpha=B^{\infty,q}_\alpha(G)$.

\begin{proof}
	Take $(x_j)$, $(\phi_j)$, and $(\widetilde \phi_j)_{j\in J}$ as in Lemma~\ref{lem:57c}. Observe that, by~\cite[Lemma 4.3]{Calzi2} we may find $N\in\N$ and a partition $J_1,\dots, J_N$ of $J$ such that the $d(x_j,x_{j'})\Meg 2R+6$  for every two distinct $j,j'\in J_h$, $h=1,\dots, N$, where $R>0$ is such that each $\phi_j$ is supported in $B(e, R)$. Take $\psi\in C^\infty_c(G)$ such that $\chi_{B(e,R+2)}\meg \phi'\meg \chi_{B(e,R+3)}$, and define $\widetilde \phi'_j\coloneqq \phi'(x_j^{-1}\,\cdot\,)$ for every $j\in J$.
	
	Observe that the continuous  inclusion $\Mc(B^{p,q}_\alpha(G))\subseteq M^{p,q}_\alpha$ follows from Lemma~\ref{lem:58c}, so that we may reduce to proving that the bilinear mapping $B^{p,q}_\alpha(G)\times M^{p,q}_\alpha \ni(f,g)\mapsto f g\in  B^{p,q}_\alpha(G)$ is continuous. Since $M^{p,q}_\alpha\subseteq B^{p,q}_{\alpha,\unif}(G)$ by Lemma~\ref{lem:58c}, arguing as in the proof of~\cite[Theorem 6.6]{Calzi2}, we may reduce to estimating $ \Pi_g^{(j'')} f$ for every $j''\in \N$ and  $\Pi(f,g)$, with the notation of the proof of~\cite[Theorem 6.6]{Calzi2}.  In fact, we shall only estimate $ \Pi_g^{(j'')} f$.
	
	Take $(\psi_j),(\psi'_j)$, $(x_{j,k})$, $(J_h)_{h=1,\dots, N'}$, $(\eta_j)_{j\Meg -M}$, and $(B_{j,k})$ as in the proof of Theorem~\ref{teo:20c}
	Define the auxiliary functions $\widetilde f, \widetilde g, \widetilde f_{j'}, \widetilde g_{j'}\colon G\times \N\to [0,\infty)$ ($j'\in\N$) so that
	\[
	\widetilde f(x,j)\coloneqq \max_{\overline B(x_{j,k},\eps^j/2)} \abs{f*\psi_j'}\chi_{B_{j,k}}(x) \qquad \text{and} \qquad \widetilde g(x,j)\coloneqq \max_{\overline B(x_{j,k},\eps^j/2)} \abs{g*\psi_j}\chi_{B_{j,k}}(x)
	\]
	and also
	\[
	\widetilde f_{j'}(x,j)\coloneqq \max_{\overline B(x_{j,k},\eps^j/2)} \abs{(\widetilde \phi'_{j'}f)*\psi_j'}\chi_{B_{j,k}}(x) \qquad \text{and} \qquad \widetilde g(x,j)\coloneqq \max_{\overline B(x_{j,k},\eps^j/2)} \abs{(\widetilde \phi_{j'}g)*\psi_j}\chi_{B_{j,k}}(x)
	\]
	for every $(x,j)\in G\times \N$.
	Then, as in the proof of Theorem~\ref{teo:20c} we see that there is a constant $C_1>0$ such that
	\[
	\begin{split}
		\norm{\eps^{-j'\alpha}[(\Pi_g^{(j'')}f)*\eta_{j'}]}_{L^{p,q}_{x,j'}(G,\N-M)}&\meg C_1 \norm*{\eps^{-j\alpha}\sum_{j'\in \N} \widetilde f(j,x) \widetilde g(j,x)}_{L^{p,q}_{x,j}(G,\N)}\\
			&\meg C_1N'^{1/\min(1,p,q)-1} \sum_{h=1}^{N'} \norm*{\eps^{-j\alpha}\sum_{j'\in J_h} \widetilde f_{j'}(j,x) \widetilde g_{j'}(j,x)}_{L^{p,q}_{x,j}(G,\N)}.
	\end{split}
	\]
	Now, using~\cite[Proposition 5.9]{Calzi4} as in the proof of Theorem~\ref{teo:20c}, we see that  there is a constant $C_2>0$ such that
	\[
	\begin{split} 
		\norm*{\eps^{-j\alpha}\sum_{j'\in J_h} \widetilde f_{j'}(j,x) \widetilde g_{j'}(j,x)}_{L^{p,q}_{x,j}(G,\N)}&\meg  \norm*{\eps^{-j\alpha} \sum_{j'\in J_h}\norm{\widetilde \phi'_{j'}f}_{L^\infty(G)} \widetilde g_{j'}(j,x)}_{L^{p,q}_{x,j}(G,\N)}\\
			&\meg C_2 \norm*{\sum_{j\in J_h} a_{j'}  \widetilde \phi_{j'} g }_{B^{p,q}_\alpha(G)}\\
			&\meg C_2 \norm{a }_{\ell^p(J_h)} \norm{ g}_{M^{p,q}_\alpha},
	\end{split}
	\]
	where $a_j\coloneqq \norm{\widetilde \phi'_{j}f}_{L^\infty(G)}$ for every $j\in \N$. Now, if $\alpha>Q_*/p$, then   the desired estimate  follows  from the fact that, by~\cite[Proposition 4.10]{Calzi4} and Proposition~\ref{prop:30c}, there is a constant $C_3>1$ such that  
	\[
	\begin{split}
		\norm{a_j}_{\ell^p_j(J_h)}&\meg C_3 \norm*{\norm{\widetilde \phi'_j f}_{B^{p,p}_{Q_*/p+\eps}(G)} }_{\ell^p_j(J)}\\
		&\meg C_3^2 \norm{f}_{B^{p,p}_{Q_*/p+\kappa}(G)}\\
		&\meg C_3^3 \norm{f}_{B^{p,q}_\alpha(G)}
	\end{split}
	\]
	for any fixed $\kappa\in (0,\alpha-Q_*/p)$. If, otherwise, $\alpha=Q_*/p$, then $q\meg\min(1,p)$,  so that by~\cite[Proposition 4.10]{Calzi4} there is a constant $C_4>0$ such that 
	\[
	\begin{split}
		\norm{a_j}_{\ell^p_j(J_h)}&\meg C_4 \norm*{\norm{\widetilde \phi'_j f}_{B^{p,q}_{Q_*/p}(G)} }_{\ell^p_j(J)}\\
		&\meg C_4^2  \norm*{\norm*{ \norm{(\widetilde \phi'_j f)*\eta_{j'}}_{L^p(G)}}_{\ell^q_{j'}(\N-M)} }_{\ell^p_j(J)}\\
		&\meg C_4^2 N'^{(1/p-1)_+}\sum_{h=1}^{N'} \norm*{\norm*{ \norm{(\widetilde \phi'_j f)*\eta_{j'}}_{L^p(G)}}_{\ell^q_{j'}(\N-M)}  }_{\ell^p_j(J_h)}\\
		&\meg C_4^2N'^{(1/p-1)_+}\sum_{h=1}^{N'} \norm*{\norm*{ \norm{(\widetilde \phi'_j f)*\eta_{j'}}_{L^p(G)}}_{\ell^p_j(J_h)}}_{\ell^q_{j'}(\N-M)} \\
		&= C_4^2N'^{(1/p-1)_+}\sum_{r=1}^N \norm*{  \norm{ (\widetilde\phi'^{(h)}f)*\eta_{j'}  }_{L^p(G)} }_{\ell^q_{j'}(\N-M)},
	\end{split}
	\]
	where $\widetilde\phi'^{(h)}=\sum_{j\in J_h} \widetilde \phi'_j$,
	so that the desired estimate follows from Corollary~\ref{cor:20b}.
\end{proof}

\end{document}